\documentclass[10pt]{amsart}
\usepackage{amsmath}
\usepackage{amssymb,amsfonts}
\input xy
\xyoption{all}

\def\n{\mathrm{N}}
\def\adj{\mathrm{P}\hspace{-5.7pt}\mathrm{P}}
\def\dpi{\ensuremath{\Pi\hspace{-7pt}\Pi}}
\def\dn{\ensuremath{\mathrm{N\hspace{-6pt}N}}}
\def\w{\ensuremath{\overline{\mathrm{W}}}}
\def\dia{\ensuremath{\,\mathrm{diag}}}
\def\dec{\ensuremath{\,\mathrm{Dec}}}
\def\s{\ensuremath{\,\mathrm{S}}}
\newtheorem{theorem}{Theorem}
\newtheorem{proposition}[theorem]{Proposition}
\newtheorem{facts}[theorem]{Facts}
\newtheorem{fact}[theorem]{Fact}

\newtheorem{corollary}[theorem]{Corollary}

\newtheorem{lemma}[theorem]{Lemma}

\numberwithin{equation}{section}

\numberwithin{theorem}{section}

\begin{document}
\title{{Double groupoids and homotopy 2-types}}

\author{A. M. Cegarra \and B. A. Heredia}
\address{Departamento de \'{A}lgebra, Universidad de
Granada, 18071 Granada, Spain} \email{acegarra@ugr.es} \email{fierash@correo.ugr.es}

\author{J. Remedios}
\address{Departamento de Matem\'{a}tica Fundamental, Universidad de La
Laguna. 38271 La Laguna, Spain} \email{jremed@ull.es}

\thanks{The first author acknowledge support from the DGI of Spain (Project: MTM2007-65431); Consejer\'{i}a de Innovaci\'{o}n
de J. de Andaluc\'{i}a (P06-FQM-1889); MEC de Espa\~{n}a, `Ingenio Mathematica(i-Math)' No. CSD2006-00032
(consolider-Ingenio 2010). \\
The second author thanks support to University of Granada (Beca Plan Propio 2009).\\
The third  author acknowledges support from DGI of Spain (Project: MTM2009-12081) and tanks the University of Granada for its support and kind hospitality.}

\begin{abstract} This work contributes to clarifying several relationships between certain higher
categorical structures and the homotopy types of their classifying
spaces. Double categories (Ehresmann, 1963) have well-understood geometric realizations, and here we deal with homotopy types
represented by double groupoids satisfying a natural `filling
condition'. Any such double groupoid characteristically has
associated to it `homotopy groups', which are defined using only
its algebraic structure. Thus arises the notion of `weak
equivalence' between such double groupoids, and a corresponding
`homotopy category' is defined. Our main result in the paper states
that the geometric realization functor induces an equivalence
between  the homotopy category of double groupoids with filling
condition and the category of homotopy 2-types (that is, the homotopy category of all topological spaces
with the property that the $n^{\text{th}}$ homotopy group at any base point vanishes for $n\geq 3$). A quasi-inverse functor is explicitly
given by means of a new `homotopy double groupoid' construction for
topological spaces.
\end{abstract}


\maketitle


\section{Introduction and summary.}
Higher-dimensional categories provide a suitable setting for the treatment of an extensive list of subjects of recognized mathematical interest. The construction of nerves and classifying spaces of higher categorical structures discovers ways to transport categorical coherence to homotopic coherence,  and it has shown its relevance as a tool in algebraic topology, algebraic geometry, algebraic $K$-theory, string field theory,  conformal field theory, and in the study of geometric structures on low-dimensional manifolds.

{\em Double groupoids}, that is, groupoid objects in the category of
groupoids, were introduced by Ehresmann  \cite{eres, eres2} in the
late fifties and later studied by several people because of their
connection with several areas of mathematics.  Roughly, a double
groupoid consists of {\em objects}, {\em horizontal} and {\em
vertical morphisms}, and {\em squares}. Each square, say $\alpha$,
has objects as vertices and morphisms as edges, as in $$
\xymatrix@C=-2pt@R=-1pt{\cdot&&\ar[ll]\cdot&\\&\alpha&\\\cdot\ar[uu]&&\ar[ll]\cdot\ar[uu]&,}\
$$
together with two groupoid compositions- the {\em vertical} and {\em
horizontal compositions}- of squares, and compatible groupoid
compositions of the edges, obeying several conditions (see Section
\ref{h-g} for details).  Any double groupoid $\mathcal{G}$ has a
{\em geometric realization} $|\mathcal{G}|$, which is the
topological space defined by first taking the double nerve $\dn
\mathcal{G}$, which is a bisimplicial set, and then realizing the
diagonal to obtain a space: $|\mathcal{G}|=|\dia\dn \mathcal{G}|$.
In this paper, we address the homotopy types obtained in this way
from double groupoids satisfying a natural {\em filling condition}:
Any filling problem
$$\xymatrix@C=-2pt@R=-1pt{\cdot && \cdot\ar[ll]\\
&\scriptstyle{\exists ?}&\\  \ar@{.>}[uu] & &  \cdot
\ar@{.>}[ll]\ar[uu]}$$ finds  a solution in the double groupoid.
This filling condition on double groupoids is often assumed in the
case of double groupoids arising in different areas of mathematics,
such as in differential geometry  or in weak Hopf algebra theory
(see the papers by Mackenzie \cite{mack} and Andruskiewitsch and
Natale \cite{a-n-2}, for example),  and it is satisfied for those
double groupoids that have emerged with an interest in algebraic
topology,
  mainly thanks to the work of Brown, Higgings, Spencer, {\em et al.}, where the connection of double
  groupoids with crossed modules and a higher Seifert-van Kampen Theory has been established (see, for
  instance, the survey paper \cite{brown} and references therein. Thus, the filling condition is easily
  proven for edge symmetric double groupoids (also called special double groupoids) with connections
  (see, for example \cite{b-h,b-s} or \cite{brown2,b-k-p, b-h-k-k-p}, for more recent instances),
  for double groupoid objects in the category of groups (also termed $\mathrm{cat}^2$-groups,
  \cite{loday,b-c-d,porter}), or, for example, for  2-groupoids (regarded as double  groupoids where one of
  the side groupoids of morphisms is discrete \cite{m-s},\cite{h-k-k}).

When a double groupoid $\mathcal{G}$ has the filling condition, then
there are characteristically associated to it `homotopy groups',
$\pi_i(\mathcal{G},a)$, which we define using only the algebraic
structure of $\mathcal{G}$, and which are trivial for integers $i\geq 3$.
A first major result states that:

\begin{quote}{\em
If $\mathcal{G}$ is a double groupoid with filling condition,
then, for each object $a$,  there are natural
isomorphisms
$\pi_i(\mathcal{G},a)\cong \pi_i(|\mathcal{G}|,|a|),\,\, i\geq 0.$}
\end{quote}

The proof of this result  requires a prior recognition of the significance of the
filling condition on double groupoids in the homotopy theory of simplicial sets; namely, we prove that
\begin{quote}
{\em A double category $\mathcal{C}$ is a double groupoid with
filling condition if and only  if the simplicial set
$\dia\dn\mathcal{C}$ is a Kan complex. }
\end{quote}
 This fact can be seen as a higher version of the  well-known fact that the nerve  of a category
is a Kan complex if and only if the category is a  groupoid (see \cite{illusie}, for example).

Once we have defined the homotopy category of double groupoids satisfying the
filling condition $\mathrm{Ho(\mathbf{DG}_{fc})}$, to be the
localization of the category of these double groupoids, with respect
to the class of {\em weak equivalences} or double functors
$F:\mathcal{G}\to \mathcal{G'}$ inducing isomorphisms
$\pi_iF:\pi_i(\mathcal{G},a)\cong \pi_i(\mathcal{G'},Fa)$ on the
homotopy groups, we then obtain an induced functor
$$|\hspace{5pt}|:\mathrm{Ho(\mathbf{DG}_{fc})}\to \mathrm{Ho(\mathbf{
Top})},\hspace{0.3cm} \mathcal{G}\mapsto |\mathcal{G}|\,,$$ where
$\mathrm{Ho(\mathbf{Top})}$ is  the localization of the category of
topological spaces  with respect to the class of weak equivalences.
Furthermore, we show a new functorial
construction of a {\em homotopy double groupoid} $\dpi X$, for any
topological space $X$, that induces a functor
$$\mathrm{Ho(\mathbf{
Top})}\to \mathrm{Ho(\mathbf{DG}_{fc})},\hspace{0.3cm} X\mapsto \dpi
X.$$

A main goal in this paper is to prove the following result, whose
proof is somewhat indirect since it is given through an explicit
description of a left adjoint functor, $\adj \dashv \dn$, to the
double nerve functor $\mathcal{G}\mapsto \dn \mathcal{G}$:

\begin{quote}{\em
Both induced functors on the homotopy categories $\mathcal{G}\mapsto
|\mathcal{G}|$ and $X\mapsto \dpi X$ restrict by giving mutually
quasi-inverse equivalence of categories
$$\mathrm{Ho(\mathbf{DG}_{fc})}\simeq\mathrm{Ho(\mathbf{
2\text{-}types)}},$$}
\end{quote}
where $\mathrm{Ho(\mathbf{ 2\text{-}types)}}$ is the full
subcategory of the homotopy category of topological spaces given by
those spaces $X$ with $\pi_i(X,a)=0$ for any integer $i\!>2$ and any
base point $a$. From the point of view of this fact, the use of
double groupoids and their classifying spaces in homotopy theory
goes back to Whitehead \cite{whi} and  Mac Lane-Whitehead \cite{mac}
since double groupoids where one of the side groupoids of morphisms
is discrete with only one object (= strict 2-groups, in the
terminology of Baez \cite{baez}) are the same as crossed modules (this observation is attributed to Verdier  in \cite{b-s}). In this context, we should mention the work by Brown-Higgins \cite{b-h} and Moerdijk-Svensson \cite{m-s} since crossed modules over groupoids are essentially the same thing as 2-groupoids and double  groupoids where one of the side groupoids of morphisms is discrete. Along the same line, our result is also a natural 2-dimensional version of the well-known  equivalence between the homotopy category of groupoids and the homotopy category of 1-types (for a useful survey of groupoids in topology, see \cite{brow}).

The plan of this paper is, briefly, as follows. After this
introductory Section 1, the paper is organized in six sections.
Section 2 aims to make this paper as self-contained as possible;
hence, at the same time as
  fixing notations and terminology, we also review  necessary aspects and results from the
 background of (bi)simplicial sets and their geometric realizations that will be used throughout  the paper.
 However, the material in  Section 2 is quite standard, so the expert reader
may skip most of it. The most original part is in Subsection 2.2,
related to the extension condition on bisimplicial sets. In Section
3, after recalling the notion of a double groupoid and fixing notations,
we mainly introduce the homotopy groups $\pi_i(\mathcal{G},a)$, at
any object $a$ of a double groupoid with filling condition
$\mathcal{G}$. Section 4 is dedicated to showing in
detail the construction of the homotopy double groupoid $\dpi X$,
characteristically associated to any topological space $X$. Here, we
prove that a continuous map $X\to Y$ is a weak homotopy
$2$-equivalence (i.e., it induces bijections on the homotopy groups $\pi_i$ for $i\leq 2$) if and only if the induced double functor $\dpi X
\to \dpi Y$ is a weak equivalence. Next, in  Section 5, we
 first address the issue of to have a manageable description  for the bisimplices in
$\dn\mathcal{G}$, the double nerve of a double groupoid, and then we
determine the homotopy type of the geometric realization
$|\mathcal{G}|$ of a double groupoid with filling condition.
Specifically, we prove that the homotopy groups of $|\mathcal{G}|$
are the same as those of $\mathcal{G}$. Our goal in
Section 6 is to prove that the double nerve functor, $\mathcal{G}
\mapsto \dn\mathcal{G}$, embeds, as a reflexive subcategory, the
category of double groupoids satisfying the filling condition into a
certain category of bisimplicial sets. The reflector functor $
K\mapsto \adj K$  works as a bisimplicial version of Brown's
construction in \cite[Theorem 2.1]{brown2}. Furthermore, as we will
prove, the resulting double groupoid $\adj K$  always represents the
homotopy 2-type of the input bisimplicial set $K$, in the sense that
there is a natural weak 2-equivalence $|K|\to |\adj K|$. This result
becomes  crucial in the final Section 7 where, bringing into
play all the previous work, the  equivalence of categories
$\mathrm{Ho(\mathbf{DG}_{fc})}\simeq\mathrm{Ho(\mathbf{
2\text{-}types)}}$ is achieved.

\section{Some preliminaries on bisimplicial sets.}
 This section aims to make this paper as self-contained as possible; Therefore, while
  fixing notations and terminology, we also review  necessary aspects and results from the
 background of (bi)simplicial sets and their geometric realizations used throughout the paper.
 However, the material in this section  is quite standard and, in general, we employ the
 standard symbolism and nomenclature to be found in texts on simplicial homotopy
theory, mainly in \cite{g-j} and \cite{may}, so the expert reader may skip most of it.
The most original part is in Subsection 2.2, related to the extension condition  and the bihomotopy
relation on bisimplicial sets.

\subsection{Kan complexes: Fundamental groupoids and homotopy groups}
~

We start by fixing some notations. In the simplicial category $\Delta$, the generating coface and codegeneracy maps are denoted by $d^i\!:[n-1]\to [n]$ and $s^i\!:[n+1]\to [n]$ respectively. However,  for $L:\Delta^{\!o}\to \mathbf{Set}$ any simplicial set, we write $d_i\!=\!L(d^i)\!:L_n\to L_{n-1}$ and $s_i\!=\!L(s^i)\!:L_n\to L_{n+1}$ for its corresponding face and degeneracy maps.

The {\em standard $n$-simplex} is $\Delta[n]=\Delta(-,[n])$ and, as is usual, we identify any simplicial map $x\!:\Delta[n]\to L$ with the simplex $x(\iota_n)\in L_n$, the image by $x$ of the basic simplex $\iota_n\!=\!id\!:[n]\to [n]$ of $\Delta[n]$. Thus, for example, the {\em $i^\text{th}$-face of $\Delta[n]$} is $d^i\!=\!\Delta(-,d^i)\!:\Delta[n-1]\to\Delta[n]$, the simplicial map with $d^i(\iota_{n-1})=d_i(\iota_n)$. Similarly,  $s^i\!=\!\Delta(-,s^i)\!:\Delta[n+1]\to\Delta[n]$ is the simplicial map that we identify with the degenerated simplex $s_i(\iota_n)$ of $\Delta[n]$.

The {\em boundary} $\partial\!\Delta[n]\subset \Delta[n]$ is the smallest simplicial subset containing all the faces $d^i\!:\Delta[n-1]\to\Delta[n]$, $0\!\leq i\!\leq n$, of $\Delta[n]$. Similarly, for any given $k$ with $0\leq k\leq n$, the {\em $k^{\text{th}}$-horn}, $\Lambda^{\!k}[n]\subset\Delta[n]$, is the smallest simplicial subset containing all the faces $d^i\!:\Delta[n-1]\to\Delta[n]$ for $0\!\leq i\!\leq n$ and $i\neq k$. For a more geometric (and useful) description of these simplicial sets, recall that there are coequalizers
$$
\bigsqcup\limits_{\scriptsize{0\leq i<j\leq n
}}\Delta[n-2]\rightrightarrows\bigsqcup\limits_{0\leq i\leq n}\Delta[n-1]\to\partial\!\Delta[n],
$$
\noindent and
\begin{equation}\label{horns}
\bigsqcup\limits_{\scriptsize{\begin{array}{c}0\!\leq\! i\!<\!j\!\leq\! n\\
i\neq k\neq j\end{array}}}\hspace{-6pt}\Delta[n-2]\rightrightarrows\hspace{-10pt}\bigsqcup\limits_{\scriptsize{\begin{array}{c}0\!\leq\! i\!\leq\! n\\i\neq k\end{array}}}\hspace{-4pt}\Delta[n-1]\to \Lambda^{\!k}[n],\hspace{10pt}~\end{equation}
given by the relations $d^jd^i=d^id^{j-1}$ if $i\!<\!j$.

A simplicial set $L$ is a {\em Kan complex} if it satisfies the so-called {\em extension condition}. Namely, for any simplicial diagram
$$
\xymatrix@C=16pt@R=16pt{\Lambda^{\!k}[n]\ar[r]\ar@{^{(}->}[d]&L\\ \Delta[n]\ar@{.>}[ur]&}
$$
there is a map $\Delta[n]\to L$ (the dotted arrow) making the
diagram commute.

In a Kan complex $L$, two simplices $x,x'\!:\Delta[n]\to L$ are said to be {\em homotopic} whenever they have the same faces and there is a {\em homotopy}  from $x$ to $x'$, that is, a simplex $y\!:\Delta[n+1]\to L$ making this diagram commutative
$$
\xymatrix@C=120pt{\partial\Delta[n+1]\ar[r]^{(xd^0\!s^{n-1}\!,\,\cdots,xd^{n-1}\!s^{n-1}\!,\,x,\,x')}\ar@{^{(}->}[d]&L.\\ \Delta[n+1]\ar[ur]_y&}
$$
  Being homotopic establishes an equivalence relation on the simplices of $L$, and we  write $[x]$ for the homotopy class of a simplex $x$. A useful result is the follows:

\begin{fact}\label{p1} Let $y,y'\!:\!\Delta[n+1]\to L$ be two simplices such that $[yd^i]=[y'd^i]$ for all
$i\neq k$; then $[yd^k]=[y'd^k]$.\end{fact}

The {\em fundamental groupoid} of $L$, denoted  $\mathrm{P} L$, also called its Poincar\'e groupoid, has as objects the
vertices $a\!:\Delta[0]\to L$, and a morphism $[x]\!:a\to b$ is the homotopy class of a simplex $x\!:\Delta[1]\to L$
with $xd^0=a$ and $xd^1=b$. The composition in $\mathrm{P} L$  is defined by
$$[x]\circ [x']=[yd^1],$$ where $y:\Delta[2]\to L$ is any simplex with $yd^2=x$ and $yd^0=x'$, and
the identities are $\mathrm{I}a=[as^0]$.

The set of {\em path components} of $L$, denoted as $\pi_0L$, is the set
of connected components of $\mathrm{P} L$, so it consists of all
homotopy classes of the 0-simplices of $L$. For any given vertex of
the Kan complex  $a\!:\!\Delta[0]\to L$, $\pi_0(L,a)$ is the set
$\pi_0L$, pointed by $[a]$, the component of $a$.  The group of
automorphisms of $a$ in the fundamental groupoid of $L$ is
$\pi_1(L,a)$, the {\em fundamental group} of $L$ at $a$.
Furthermore, denoting every composite map $\Delta[m]\to
\Delta[0]\overset{a}\to L$ by $a$ as well, the $n^{\text{th}}$ {\em
homotopy group} $\pi_n(L,a)$ of $L$ at $a$ consists of homotopy
classes of simplices $x\!:\!\Delta[n]\to L$ for all simplices $x$
with faces $xd^i=a$, for $0\!\leq i\!\leq n$. The multiplication in
the (abelian, for $n\!\geq 2$) group $\pi_n(L,a)$ is given by
$$[x]\circ [x']=[yd^n],$$ where $y\!:\!\Delta[n+1]\to L$ is a (any)
solution to the extension problem
$$
\xymatrix@C=60pt@R=20pt{\Lambda^{\!n}[n+1]\ar[r]^-{(a,\,\cdots,a,\,x',\,-,\,x)}\ar@{^{(}->}[d]&L.\\
\Delta[n+1]\ar@{.>}[ur]_y&}
$$

The following fact is used several times throughout the paper:

\begin{fact}\label{p2}Let $L$ be a Kan complex and $n$ an integer such that the homotopy groups
$\pi_n(L,a)$ vanish for all base vertices $a$. Then, every extension problem
$$
\xymatrix@R=20pt{\partial\Delta[n+1]\ar[r]\ar@{^{(}->}[d]&L\\
\Delta[n+1]\ar@{.>}_{\exists ?}[ur]&}
$$
has a solution. In particular, any two $n$-simplices with the same faces are homotopic.
\end{fact}

We shall end this preliminary subsection by recalling that two simplicial maps ${f,g\!:\!L\to L'}$
are {\em homotopic} whenever there is a map $L\times\Delta[1]\to L'$ which is $f$ on
${L\times{0}}$ and $g$ on $L\times {1}$. The resulting homotopy relation becomes a congruence on
the category {\bf KC} of Kan complexes, and the corresponding quotient category is the
{\em homotopy category of Kan complexes}, Ho({\bf KC}).
A map between Kan complexes is a {\em homotopy equivalence} if it induces an isomorphism in
the homotopy category. There is an analogous result of Whitehead's theorem on CW-complexes to Kan complexes:

\begin{fact} A simplicial map between Kan complexes, ${L\to L'}$, is a
homotopy equivalence if and only if it induces an isomorphism
$\pi_i(L,a)\cong\pi_i(L',fa)$ for all base vertex $a$ of $L$ and
any integer $i\geq 0$.
\end{fact}

\subsection{Bisimplicial sets: The extension condition and the bihomotopy relation}\label{s22}
~

It is often convenient to view a bisimplicial set
$K\!:\Delta^{\!o}\times \Delta^{\!o}\to \mathbf{Set}$ as a
(horizontal) simplicial object in the category of (vertical)
simplicial sets. For this case, we  write
${d_i^{\mathrm{h}}\!=\!K(d^i,id)\!:\!K_{p,q}\to K_{p-1,q}}$ and
$s_i^{\mathrm{h}}\!=\!K(s^i,id):K_{p,q}\to K_{p+1,q}$ for the
horizontal face and degeneracy maps, and, similarly
$d_j^{\mathrm{v}}=K(id,d^j)$ and $s_j^{\mathrm{v}}=K(id,s^j)$ for the
vertical ones.

For simplicial sets $X$ and $Y$, let $X\!\otimes\! Y$ be the
bisimplicial set with $(X\!\otimes\! Y)_{p,q}=X_p\!\times\! Y_q$.
The {\em standard $(p,q)$-bisimplex} is
$$\Delta[p,q]:=\Delta\!\times\!\!\Delta(-,([p],[q]))=\Delta[p]\otimes\Delta[q],$$
the bisimplicial set represented by the object $([p],[q])$, and
usually we identify any bisimplicial map $x:\Delta[p,q]\to K$ with
the $(p,q)$-bisimplex $x(\iota_p,\iota_q)\in K$. The functor
$([p],[q])\mapsto \Delta[p,q]$ is then a co-bisimplicial bisimplicial
set, whose cofaces and codegeneracy operators are denoted by
$d_{\mathrm{h}}^{i}$, $d^j_{\mathrm{v}}$, and so on,  as in the
diagram$$\xymatrix@C=50pt{\Delta[p-1,q]\ar@<0.5ex>[r]^-{d^i_{\mathrm{h}}=d^i\otimes
id}&\Delta[p,q]\ar@<0.5ex>[l]^-{s^i_{\mathrm{h}}=s^i\otimes
id}\ar@<-0.5ex>[r]_-{s^j_{\mathrm{v}}=id\otimes
s^j}&\Delta[p,q-1]\ar@<-0.5ex>[l]_-{d^j_{\mathrm{v}}=id\otimes
d^j}}.$$

The $(k,l)^{\text{th}}$-{\em horn}  $\Lambda^{\!{k,l}}[p,q]$, for
any integers $0\!\leq k\!\leq p$ and  $0\!\leq l\!\leq q$, is the
bisimplicial subset of $\Delta[p,q]$ generated by the horizontal and
vertical faces
$\Delta[p-1,q]\overset{d^i_{\mathrm{h}}}\hookrightarrow \Delta[p,q]$
and $\Delta[p,q-1]\overset{d^j_{\mathrm{v}}}\hookrightarrow
\!\Delta[p,q]$ for all $i\neq k$ and $j\neq l$. There is a natural
pushout diagram
$$
\xymatrix@C=14pt@R=18pt{\Lambda^{\!k}[p]\otimes\Lambda^{\!l}[q]\ar@{^{(}->}[r]\ar@{^{(}->}[d]&\Delta[p]
\otimes\Lambda^{\!l}[q]\ar@{^{(}->}[d]\\
\Lambda^{\!k}[p]\otimes\Delta[q]\ar@{^{(}->}[r]&
\Lambda^{\!{k,l}}[p,q],}
$$
which, taking into account the coequalizers (\ref{horns}), states that the system of data to define a bisimplicial map $x:\Lambda^{\!{k,l}}[p,q]\to K$ consists of a list of bisimplices $$x=(x_0,\dots,x_{k-1},-,x_{k+1},\dots,x_p;x'_0,\dots,x'_{l-1},-,x'_{l+1},\dots,x'_q),$$
where $x_i:\Delta[p-1,q]\to K$ and $x'_{\hspace{-2pt}j}:\Delta[p,q-1]\to K$, such that the following compatibility conditions hold:
\begin{itemize}
\item[-] $x_{\hspace{-2pt}j}d^i_{\mathrm{h}}=x_id^{j-1}_{\mathrm{h}}$, \ for all $0\leq i<j\leq p$ with $i\neq k\neq j$,
\vspace{0.2cm}
\item[-] $x'_{\hspace{-2pt}j}d^i_{\mathrm{v}}=x'_id^{j-1}_{\mathrm{v}}$, \ for all $0\leq i<j\leq q$ with $i\neq l\neq j$,
\vspace{0.2cm}
\item[-]$x'_{\hspace{-2pt}j}d^i_{\mathrm{h}}=x_id^{j}_{\mathrm{v}}$\,, \ \ \ \,for all $0\leq i\leq p$, $0\leq j\leq q$ with $i\neq k$, $j\neq l$.
\end{itemize}

A bisimplicial set $K$ satisfies the {\em extension condition} if all simplicial sets $K_{p,*}$ and $K_{*,q}$ are Kan complexes, that is, if any of the extension problems
$$
\xymatrix@R=14pt@R=18pt{\Delta[p]\otimes\Lambda^{\!l}[q]\ar[r]\ar@{^{(}->}[d]&K\\
\Delta[p,q]\ar@{.>}_{\exists ?}[ur]&} \hspace{0.6cm}
\xymatrix@R=14pt@R=18pt{\Lambda^{\!k}[p]\otimes\Delta[q]\ar[r]\ar@{^{(}->}[d]&K\\
\Delta[p,q]\ar@{.>}_{\exists ?}[ur]&}
$$ has a solution, and, moreover, if there is also a solution for any extension problem of the form
$$
\xymatrix@R=18pt{\Lambda^{\!k,l}[p,q]\ar[r]\ar@{^{(}->}[d]&K\\
\Delta[p,q]\ar@{.>}_{\exists ?}[ur]&}
$$

When a bisimplicial set $K$ satisfies the extension condition, then
every bisimplex ${x:\Delta[p,q]\to K}$, which can be regarded both
as a simplex of the vertical Kan complex $K_{p,*}$ and as a simplex of
the horizontal Kan complex $K_{*,q}$, defines both a {\em vertical
homotopy class},  denoted by $[x]_{\mathrm{v}}$, and a
{\em horizontal homotopy class},  denoted by
$[x]_{\mathrm{h}}$. The following lemma is needed.

\begin{lemma}\label{p3} Let $x,x':\Delta[p,q]\to K$ be bisimplices of bisimplicial set $K$, which satisfies the extension condition.
The following conditions are equivalent:

i) There exists $y:\Delta[p,q]\to K$ such that
$[x]_{\mathrm{h}}=[y]_{\mathrm{h}}$ and
$[y]_{\mathrm{v}}=[x']_{\mathrm{v}}$,

ii) There exists  $z:\Delta[p,q]\to K$ such that
$[x]_{\mathrm{v}}=[z]_{\mathrm{v}}$ and
$[z]_{\mathrm{h}}=[x']_{\mathrm{h}}$.

\end{lemma}
\begin{proof} We only prove that $i)$ implies $ii)$ since the proof for the other implication is similar. Let $\alpha:\Delta[p+1,q]\to K$ be a horizontal homotopy (i.e., a homotopy in the Kan complex $K_{*,q}$) from $x$ to $y$, and let $\beta:\Delta[p,q+1]\to K$ be a vertical homotopy from $y$ to $x'$. Since $K$ satisfies the extension condition, a bisimplicial map $\Gamma:\Delta[p+1,q+1]\to K$ can be found such that the diagram below commutes.
$$
\xymatrix@C=210pt{\Lambda^{\!p,q+1}[p+1,q+1]\ar[r]^-{(\beta
d^0_{\mathrm{h}}\!s^{p-1}_{\mathrm{h}}\!,\dots,\,\beta
d^{p-1}_{\mathrm{h}}\!s^{p-1}_{\mathrm{h}}\!,\,-,\,\beta;\, \alpha
d^0_{\mathrm{v}}\!s^{q-1}_{\mathrm{v}}\!,\dots,\,\alpha
d^{q-1}_{\mathrm{v}}\!s^{q-1}_{\mathrm{v}}\!,\,\alpha,\,-)}\ar@{^{(}->}[d]&K\\
\Delta[p+1,q+1]\ar[ur]_\Gamma&}
$$

Then, by taking $\alpha'=\Gamma
d^{q+1}_{\mathrm{v}}:\Delta[p+1,q]\to K$, $\beta'=\Gamma
d^p_{\mathrm{h}}:\Delta[p,q+1]\to K$, and
$z=\alpha'd^p_{\mathrm{h}}=\beta'd^{q+1}_{\mathrm{v}}:\Delta[p,q]\to
K$, one sees that $\alpha'$ becomes a horizontal homotopy (i.e., a
homotopy in $K_{*,q}$) from $z$ to $x'$ and  $\beta'$ becomes
a vertical homotopy from $x$ to $z$. Therefore,
$[x]_{\mathrm{v}}=[z]_{\mathrm{v}}$ and
$[z]_{\mathrm{h}}=[x']_{\mathrm{h}}$, as required.
\end{proof}

The two simplices $x,x':\Delta[p,q]\to K$ in the above Lemma \ref{p3} are said to be {\em bihomotopic} if the equivalent conditions i) and ii) hold.

\begin{lemma}\label{p4}
If $K$ is a bisimplicial set satisfying the extension condition, then `to be bihomotopic' is an equivalence relation on the bisimplices of bidegree $(p,q)$ of $K$, for any $p,q\geq 0$.
\end{lemma}
\begin{proof}
The relation is obviously reflexive, and it is symmetric thanks to
Lemma \ref{p3}. For transitivity, suppose $x,x',x'':\Delta[p,q]\to
K$ such that $x$ and $x'$ are bihomotopic as well as $x'$ and $x''$
are. Then, for some $y,y':\Delta[p,q]\to K$, we have
$[x]_{\mathrm{h}}=[y]_{\mathrm{h}}$,
$[y]_{\mathrm{v}}=[x']_{\mathrm{v}}$,
$[x']_{\mathrm{h}}=[y']_{\mathrm{h}}$, and
$[y']_{\mathrm{v}}=[x'']_{\mathrm{v}}$. Also, again by Lemma
\ref{p3}, there is $z:\Delta[p,q]\to K$ such that
$[y]_{\mathrm{h}}=[z]_{\mathrm{h}}$ and
$[z]_{\mathrm{v}}=[y']_{\mathrm{v}}$. It follows that
$[x]_{\mathrm{h}}=[z]_{\mathrm{h}}$ and
$[z]_{\mathrm{v}}=[x'']_{\mathrm{v}}$, whence $x$ and $x''$ are
bihomotopic.
\end{proof}

We will write $[[x]]$ for the bihomotopic class of a bisimplex $x:\Delta[p,q]\to K$.

\begin{lemma}\label{p5}
Let $K$ be any bisimplicial set satisfying the extension condition.
There are four well-defined mappings such that $[[x]]\mapsto
[xd^i_{\mathrm{h}}]_{\mathrm{v}}$,\,  $[[x]]\mapsto
[xd^j_{\mathrm{v}}]_{\mathrm{h}}$,\, $[x]_{\mathrm{h}}\mapsto
[[xs^j_{\mathrm{v}}]]$, and $[x]_{\mathrm{v}}\mapsto
[[xs^i_{\mathrm{h}}]]$ respectively, for any $x:\Delta[p,q]\to K$,
$0\leq i\leq p$ and $0\leq j\leq q$.
\end{lemma}
\begin{proof}
Suppose that $[[x]]=[[x']]$. Then,
$[x]_{\mathrm{h}}=[y]_{\mathrm{h}}$ and
$[y]_{\mathrm{v}}=[x']_{\mathrm{v}}$, for some $y:\Delta[p,q]\to K$.
It follows that $xd^i_{\mathrm{h}}=yd^i_{\mathrm{h}}$ and there is a
vertical homotopy, say $z:\Delta[p,q+1]\to K$, from $y$ to $x'$. As
 $zd^i_{\mathrm{h}}:\Delta[p-1,q+1]\to K$ is then a
vertical homotopy from $yd^i_{\mathrm{h}}$ to $x'd^i_{\mathrm{h}}$,
we conclude that
$[xd^i_{\mathrm{h}}]_{\mathrm{v}}=[x'd^i_{\mathrm{h}}]_{\mathrm{v}}$.
The proof that
$[xd^i_{\mathrm{v}}]_{\mathrm{h}}=[x'd^i_{\mathrm{v}}]_{\mathrm{h}}$
is similar. For the third mapping, note that any horizontal homotopy
$y:\Delta[p+1,q]\to K$ from $x$ to $x'$ yields  the horizontal
homotopy $ys^j_{\mathrm{v}}:\Delta[p+1,q+1]\to K$ from
$xs^j_{\mathrm{v}}$ to $x's^j_{\mathrm{v}}$. Therefore,
$[xs^j_{\mathrm{v}}]_{\mathrm{h}}=[x's^j_{\mathrm{v}}]_{\mathrm{h}}$,
whence $[[xs^j_{\mathrm{v}}]]=[[x's^j_{\mathrm{v}}]]$, as required.
Similarly, we see that $[x]_{\mathrm{v}}=[x']_{\mathrm{v}}$ implies
$[[xs^i_{\mathrm{h}}]]=[[x's^i_{\mathrm{h}}]]$.
\end{proof}

We shall end this subsection by remarking that any bisimplicial set $K$, satisfying the extension condition, has
associated {\em horizontal fundamental groupoids} $\mathrm{P} K_{*,q}$, one for each integer $q\geq 0$, whose objects
are the bisimplices $x:\Delta[0,q]\to K$ and morphisms $[y]_{\mathrm{h}}:x'\to x$ horizontal homotopy classes of
bisimplices $y:\Delta[1,q]\to K$ with $yd^0_{\mathrm{h}}=x'$ and  $yd^1_{\mathrm{h}}=x$. The composition in these
groupoids $\mathrm{P} K_{*,q}$ is written  using the symbol $\circ_{\mathrm{h}}$, so  the composite of
$[y]_{\mathrm{h}}$ with $[y']_{\mathrm{h}}:x''\to x'$ is
$$[y]_{\mathrm{h}}\circ_{\mathrm{h}} [y']_{\mathrm{h}}=[\gamma d^1_{\mathrm{h}}]_{\mathrm{h}},$$ where
$\gamma:\Delta[2,q]\to K$ is a (any) bisimplex with $\gamma
d^2_{\mathrm{h}}=y$ and $\gamma d^0_{\mathrm{h}}=y'$. The identities
are denoted as $\mathrm{I}^{\mathrm{h}}x$, that is,
$\mathrm{I}^{\mathrm{h}}x=[xs^0_{\mathrm{h}}]_{\mathrm{h}}$.

And similarly, $K$ also has associated {\em vertical fundamental groupoids} $\mathrm{P} K_{p,*}$, $p\geq 0$, whose
morphisms $[z]_{\mathrm{v}}:zd^0_{\mathrm{v}}\to zd^1_{\mathrm{v}}$ are vertical homotopy classes of bisimplices
$z:\Delta[p,1]\to K$. For these, we  use the symbol $\circ_{\mathrm{v}}$ for denoting the composition and
$\mathrm{I}^{\mathrm{v}}$ for identities.

\vspace{0.2cm}
\subsection{Weak homotopy types: Some related constructions}
~

Let ${\bf Top}$ denote the category of spaces and continuous maps. A
map $X\to X'$ in ${\bf Top}$ is a {\em weak equivalence} if it
induces an isomorphism  $\pi_i(X,a)\cong\pi_i(X',fa)$ for all base
points $a$ of $X$ and $i\geq 0$. The {\em category of weak homotopy
types} is defined as the localization of the category of spaces
with respect to the class of weak equivalences \cite{qui,hovey}
and, for any given integer $n$, the {\em category of homotopy
$n$-types} is its full subcategory given by those spaces $X$ with
$\pi_i(X,a)=0$ for any integer $i\!>n$ and any base point $a$.

There are various constructions on (bi)simplicial sets that traditionally aid in the algebraic study of homotopy $n$-types. Below is a brief review of the constructions  used in this work.

Segal's {\em geometric realization} functor \cite{segal},  for
simplicial spaces $K\!:\Delta^{\!o}\to  \mathbf{Top}$,  is denoted
by $K\mapsto |K|$. Recall that it is defined as the left adjoint to
the functor that associates to a space $X$ the simplicial space $[n]\mapsto
X^{\Delta_n}$, where
$$\Delta_n=\{(t_0,\ldots,t_n)\!\in\!\mathbb{R}^{n\text{+}1}\,|\,\sum\!
t_i\!=\!1,\, 0\!\leq\! t_i\!\leq\! 1\}$$ denotes the affine simplex
having $[n]$ as its set of vertices and $X^{\Delta_n}$ is the the
function space of continuous maps from $\Delta_n$ to $X$, given the
compact-open topology. The underlying simplicial set is the {\em
singular complex} of $X$, denoted by $\s X$.

For instance, by regarding a set as a discrete space, the (Milnor's) geometric realization of a simplicial set $L:\Delta^{\!o}\to \mathbf{Set}$ is $|L|$, which is a CW-complex whose $n$-cells are in one-to-one correspondence with the $n$-simplices of $L$ which are nondegenerate. The following six facts are well-known:
\begin{facts}\label{f1}
\begin{enumerate}\item For any space $X$, $\s X$ is a Kan complex.

\item For any Kan complex $L$, there are natural isomorphisms $\pi_i(L,a)\cong \pi_i(|L|,|a|)$,  for all
base vertices $a:\Delta[0]\to L$ and $n\geq 0$.

\item A simplicial map between Kan complexes $L\to L'$ is a homotopy equivalence if and only if the induced map on realizations $|L|\to|L'|$ is a homotopy equivalence.
\item For any Kan complex $L$, the unit of the adjunction  $L\to S|L|$ is a homotopy equivalence.
 \item A continuous map $X\to Y$ is a weak homotopy equivalence  if and only if  the induced $\s X\to\s Y$ is a homotopy equivalence.
 \item For any space $X$, the counit $|\s X|\to X$ is a weak homotopy equivalence.
\end{enumerate}
\end{facts}
When a bisimplicial set $K\!:\Delta^{\!o}\!\times\!\Delta^{\!o}\to \mathbf{Set}$ is
regarded as a simplicial object in the simplicial set category
 and one takes geometric realizations, then one obtains a simplicial space
 $\Delta^{\!o}\to {\bf Top}$, $[p]\mapsto |K_{p,*}|$, whose
Segal realization  is taken to be $|K|$, the geometric realization of
 $K$. As there are natural homeomorphisms \cite[Lemma in page 86]{qui2}
$$
|[p]\mapsto |K_{p,*}||\cong |\dia K|\cong |[q]\mapsto |K_{*,q}||,
$$ where $\dia K$ is the simplicial set obtained by composing $K$ with the diagonal functor $\Delta\to\Delta\times\Delta$, $[n]\mapsto ([n],[n])$, one usually takes $$|K|=|\dia K|.$$

Composing with the ordinal sum functor
$\mbox{or}:\Delta\times\Delta\to\Delta$, $([p],[q]) \mapsto
[p\!+\!1\!+\!q]$, gives  Illusie's {\em total $\dec$} functor,
$L\mapsto \mbox{Dec} L$, from simplicial to bisimplicial sets
\cite[VI, 1.5]{illusie}. More specifically, for any simplicial set
$L$, $\mbox{Dec} L$ is the bisimplicial set whose bisimplices of
bidegree $(p,q)$ are the $(p\!+\!1\!+\!q)$-simplices of $L$,
$x:\Delta[p\!+\!1\!+\!q]\to L$, and whose simplicial operators are
given by  $xd^i_{\mathrm{h}}=xd^i$, $xs^i_{\mathrm{h}}=xs^i$, for
$0\leq i\leq p$,  and $xd^j_{\mathrm{v}}=d^{p+1+j}$,
$xs^j_{\mathrm{v}}=xs^{p+1+j}$, for $0\leq j\leq q$. The functor
$\mbox{Dec}$ has a right adjoint \cite{dus}
\begin{equation}\label{d-w} \dec\dashv\w ,\end{equation}  often
called the {\em codiagonal} functor, whose description is as follows
\cite[III]{a-m}: for any bisimplicial set $K$, an $n$-simplex of $\w
K$ is a bisimplicial map
$$
\xymatrix@C=45pt{\bigsqcup\limits_{p=0}^{n}\Delta[p,n\!-\!p]\ar[r]^-{(x_{0},\dots,x_{n})}&K}
$$
such that $x_{\!p}d_{\mathrm{v}}^0= x_{p\text{\scriptsize
+}\hspace{-1pt}1}d^{p+1}_{\mathrm{h}}$, for $0\leq p< n$, whose
faces and degeneracies are given by
$$
\begin{array}{l}
  (x_{0},\dots,x_{n})d^i=(x_{0}d^i_{\mathrm{v}},\ldots,x_{i-1}d^1_{\mathrm{v}},x_{i+1}d^i_{\mathrm{h}},
  \ldots, x_{n}d^i_{\mathrm{h}}), \\ [0.5pc]
  (x_{0},\dots,x_{n})s^i= (x_{0}s^i_{\mathrm{v}},\ldots,x_{i}s^0_{\mathrm{v}},x_{i}s^i_{\mathrm{h}},
 \ldots, x_{n}s^i_{\mathrm{h}})\,.
\end{array}
$$
The unit and
the counit of the adjunction, $\mbox{u}:L\to \w\dec L$ and $\mbox{v}:\dec\w K\to
K$, are respectively defined by
\[\begin{array}{lll}\mbox{u}(y)=(ys^0,\dots,ys^n)& ~~~ &(y:\Delta[n]\to L)\\
\mbox{v}(x_0,\dots,x_{p+1+q})=x_{p+1}d^0_{\mathrm{h}}&&((x_0,\dots,x_{p+1+q})\!:\!\Delta[p,q]\to
\dec\w X)\,.\end{array}\]

The following facts are used in our development below:

\begin{facts}\label{f18}
\begin{enumerate}\item
For each $n\geq 0$, there is a natural Alexander-Whitney type diagonal approximation
$$\begin{array}{l}\phi:\dec\Delta[n]\to \Delta[n,n],\\
(\Delta[p\!+\!1\!+\!q]\overset{ x}\to\Delta[n])\ \mapsto \ (\Delta[p]\overset{x(d^{p+1})^q}\longrightarrow\Delta[n], \Delta[q]\overset{x(d^0)^{p+1}}\longrightarrow\Delta[n])\end{array}$$
such that, for any bisimplicial set $K$, the induced simplicial map
$ {\phi^*\!:\!\dia K\rightarrow \w K}$ determines a homotopy equivalence $$|\dia K|\simeq|\w K|$$ on the corresponding geometric realizations    \cite[Theorem 1.1]{c-r}.
\item For any simplicial map $f\!:\!L\to L'$, the induced $|f|\!:\!|L|\to |L'|$ is a homotopy equivalence if and only if the induced $|\dec f|\!:\!|\dec L|\to |\dec L'|$ is a homotopy equivalence \cite[Corollary 7.2]{c-r}.
\item For any simplicial set $L$ and any bisimplicial set $K$, both induced maps $|{\mathrm u}|\!:\!|L|\to |\w \dec L|$ and $|\mathrm{v}|\!:\!|\dec\w K|\to |K|$ are homotopy equivalences \cite[Proposition 7.1 and discussion below]{c-r}.

\item If $K$ is any bisimplicial set satisfying the extension condition, then $\w K$ is a Kan complex \cite[Proposition 2]{c-r2}.

\item If $L$ is a Kan complex, then $\dec L$ satisfies the extension condition {\em (the proof is a
straightforward application of \cite[Lemma 7.4]{may}) or \cite[Lemma
1]{c-r2})} .
\end{enumerate}

\end{facts}

\section{Double groupoids satisfying the filling condition: Homotopy groups.}\label{h-g}

 A (small) double groupoid \cite{eres, eres2, b-s,k-s} is a groupoid object in the category of small groupoids. In general, we employ the standard nomenclature concerning double categories but, for the sake of clarity, we shall fix  some terminology and notations below.

  A (small) category can be described as a system  $(M,O,\mathrm{s},\mathrm{t},\mathrm{I},\circ)$, where $M$ is
  the set of morphisms, $O$ is the set of objects, $\mathrm{s},\mathrm{t}\!:\!M\to O$ are the source
  and target maps, respectively,  $\mathrm{I}\!:\!O\!\to\! M$ is the identities map, and
  $\circ\! :\!M {}_\mathrm{s}\! \!\times_\mathrm{t}\!M\to M$ is the composition map, subject to the usual associativity and identity axioms. Therefore,  a {\em double category} provides us with the following data: a set $O$ of {\em objects}, a set $H$ of {\em horizontal morphisms}, a set $V$ of {\em vertical morphisms}, and a set $C$ of {\em squares},
 together with four category structures, namely, the {\em category of
 horizontal morphisms}
 $(H,O,\mathrm{s}^\mathrm{h},\mathrm{t}^\mathrm{h},\mathrm{I}\mathrm{^h},\circ_\mathrm{h})$, the {\em category of vertical morphisms}
 $(V,O,\mathrm{s}^{\hspace{-1pt}\mathrm{v}},\mathrm{t}^{\hspace{-1pt}\mathrm{v}},\mathrm{I}\mathrm{^v},\circ_\mathrm{v})$, the {\em horizontal category of squares}
 $(C,V,\mathrm{s}^\mathrm{h},\mathrm{t}^\mathrm{h},\mathrm{I}\mathrm{^h},\circ_\mathrm{h})$,
 and the {\em vertical category of squares}
 $(C,H,\mathrm{s}^{\hspace{-1pt}\mathrm{v}},\mathrm{t}^{\hspace{-1pt}\mathrm{v}},\mathrm{I}\mathrm{^v},\circ_\mathrm{v})$. These are subject to the following three axioms:

 \begin{description}
 \item[Axiom 1] $\left\{\begin{array}{cl}\mathrm{(i)}& \mathrm{s}^{\mathrm{h}}\mathrm{s}^{\hspace{-1pt}\mathrm{v}}=
 \mathrm{s}^{\hspace{-1pt}\mathrm{v}}\mathrm{s}^{\mathrm{h}},\  \mathrm{t}^{\mathrm{h}}\mathrm{t}^{\hspace{-1pt}\mathrm{v}}=
 \mathrm{t}^{\hspace{-1pt}\mathrm{v}}\mathrm{t}^{\mathrm{h}},\ \mathrm{s}^{\mathrm{h}}\mathrm{t}^{\hspace{-1pt}\mathrm{v}}=
 \mathrm{t}^{\hspace{-1pt}\mathrm{v}}\mathrm{s}^{\mathrm{h}},\  \mathrm{s}^{\hspace{-1pt}\mathrm{v}}\mathrm{t}^{\mathrm{h}}=
 \mathrm{t}^{\mathrm{h}}\mathrm{s}^{\hspace{-1pt}\mathrm{v}},
 \\\mathrm{(ii)}&  \mathrm{s}^{\mathrm{h}}\mathrm{I}^{\hspace{-1pt}\mathrm{v}}=\mathrm{I}^{\hspace{-1pt}\mathrm{v}}
 \mathrm{s}^{\mathrm{h}},\ \mathrm{t}^{\mathrm{h}}\mathrm{I}^{\hspace{-1pt}\mathrm{v}}=
 \mathrm{I}^{\hspace{-1pt}\mathrm{v}}\mathrm{t}^{\mathrm{h}},\
 \mathrm{s}^{\hspace{-1pt}\mathrm{v}}\mathrm{I}^{\mathrm{h}}=\mathrm{I}^{\mathrm{h}}\mathrm{s}^{\hspace{-1pt}\mathrm{v}},\
 \mathrm{t}^{\hspace{-1pt}\mathrm{v}}\mathrm{I}^{\mathrm{h}}=\mathrm{I}^{\mathrm{h}}\mathrm{t}^{\hspace{-1pt}\mathrm{v}},
 \\ \mathrm{(iii)}& \mathrm{I}^{\mathrm{h}}\mathrm{I}^{\hspace{-1pt}\mathrm{v}}=
 \mathrm{I}^{\hspace{-1pt}\mathrm{v}}\mathrm{I}^{\mathrm{h}}.\end{array}\right.$
 \end{description}

 Equalities in {\bf Axiom 1} allow  a square $\alpha\in C$ to be depicted in the form
 \begin{equation}\label{alfa1}
 \xymatrix@C=-1pt@R=-1pt{d&&\ar[ll]_{g}b\\&\alpha&\\c\ar[uu]^{w}&&\ar[ll]^{f}a\ar[uu]_{u}}
 \end{equation}
 where $\mathrm{s}^{\mathrm{h}}\alpha=u,\
 \mathrm{t}^{\mathrm{h}}\alpha=w,\ \mathrm{s}^{\hspace{-1pt}\mathrm{v}}\alpha=f$
 and $\mathrm{t}^{\hspace{-1pt}\mathrm{v}}\alpha=g$, and the four vertices of the
 square representing $\alpha$ are
 $\mathrm{s}^{\mathrm{h}}\mathrm{s}^{\hspace{-1pt}\mathrm{v}}\alpha=a,\
 \mathrm{t}^{\mathrm{h}}\mathrm{t}^{\hspace{-1pt}\mathrm{v}}\alpha=d,\
 \mathrm{s}^{\mathrm{h}}\mathrm{t}^{\hspace{-1pt}\mathrm{v}}\alpha=b$ and
 $\mathrm{s}^{\hspace{-1pt}\mathrm{v}}\mathrm{t}^{\mathrm{h}}\alpha=c$. Moreover, if we represent identity morphisms by the
 symbol $\xymatrix@C=12pt{\ar@<-0.1ex>@{-}[r]\ar@<0.1ex>@{-}[r]&}$,
 then,   for any horizontal morphism  $f$, any vertical morphism $u$,
 and any object $a$,  the associated identity squares $\mathrm{I}^{\hspace{-1pt}\mathrm{v}}\!f$,
 $\mathrm{I}^{\mathrm{h}}u$ and
 $\mathrm{I}a\!:=\!\mathrm{I}^{\mathrm{h}}\mathrm{I}^{\hspace{-1pt}\mathrm{v}}a=\mathrm{I}^{\hspace{-1pt}
 \mathrm{v}}
 \mathrm{I}^{\mathrm{h}}a$
  are  respectively given in the form
 $$\xymatrix@C=4pt@R=3pt{\cdot&&\cdot\ar[ll]_f \\&&\\ \cdot\ar@<-0.1ex>@{-}[uu]\ar@<0.1ex>@{-}[uu]_
 {\ \textstyle \mathrm{I}^{\hspace{-1pt}\mathrm{v}}\!f}&&\cdot\ar[ll]^f\ar@<-0.1ex>@{-}[uu]
 \ar@<0.1ex>@{-}[uu]}\hspace{1cm}
 \xymatrix@C=4pt@R=3pt{\cdot&&\cdot\ar@<-0.1ex>@{-}[ll]\ar@<0.1ex>@{-}[ll]\\&&\\
 \cdot\ar[uu]^u_{\,\textstyle \mathrm{I}^{\hspace{-1pt}\mathrm{h}} u}&&\cdot
 \ar@<-0.1ex>@{-}[ll]\ar@<0.1ex>@{-}[ll]\ar[uu]_u }\hspace{1cm}
 \xymatrix@C=3pt@R=3pt{\cdot &&\cdot\ar@<-0.1ex>@{-}[ll]\ar@<0.1ex>@{-}[ll] \\&~&\\
  \cdot\ar@<-0.1ex>@{-}[uu]\ar@<0.1ex>@{-}[uu]_{\ \, \textstyle{\mathrm{I}a}}&&\cdot\ar@<-0.1ex>@{-}[ll]\ar@<0.1ex>@{-}[ll]
  \ar@<-0.1ex>@{-}[uu]\ar@<0.1ex>@{-}[uu]&}$$

 The equalities in {\bf Axiom 2} below show the squares are compatible with the boundaries, whereas   {\bf Axiom 3}
 establishes the necessary coherence between the two vertical and horizontal compositions of squares.

 \begin{description}
\item[Axiom 2]  $\left\{\begin{array}{cl}\mathrm{(i)}&\mathrm{s}^{\hspace{-1pt}\mathrm{v}}(\alpha\circ_{\mathrm{h}}\beta)=
\mathrm{s}^{\hspace{-1pt}\mathrm{v}}\alpha\circ_{\mathrm{h}}\mathrm{s}^{\hspace{-1pt}\mathrm{v}}\beta,\
\ \mathrm{t}^{\hspace{-1pt}\mathrm{v}}
 (\alpha\circ_{\mathrm{h}}\beta)=\mathrm{t}^{\hspace{-1pt}\mathrm{v}}\alpha\circ_{\mathrm{h}}\mathrm{t}^{\hspace{-1pt}\mathrm{v}}\beta,
 \\\mathrm{(ii)}&\mathrm{s}^{\mathrm{h}}(\alpha\circ_{\mathrm{v}}\beta)=\mathrm{s}^{\mathrm{h}}\alpha\circ_{\mathrm{v}}
 \mathrm{s}^{\mathrm{h}}\beta,\ \  \mathrm{t}^{\mathrm{h}}(\alpha\circ_{\mathrm{v}}\beta)=
 \mathrm{t}^{\mathrm{h}}\alpha\circ_{\mathrm{v}}\mathrm{t}^{\mathrm{h}}\beta,
 \\ \mathrm{(iii)}&  \mathrm{I}^{\hspace{-1pt}\mathrm{v}}(f\circ_{\mathrm{h}}f')=\mathrm{I}^{\hspace{-1pt}\mathrm{v}}\!f\circ_{\mathrm{h}}\mathrm{I}^{\hspace{-1pt}\mathrm{v}}\!f',
 \ \ \mathrm{I}^{\mathrm{h}}(u\circ_{\mathrm{v}} u')=
 \mathrm{I}^{\mathrm{h}}u\circ_{\mathrm{v}}\mathrm{I}^{\mathrm{h}}u'.\end{array}\right.$
 \item[Axiom 3]{\em In the situation
 $$\xymatrix@C=-1pt@R=-2pt{\cdot&&\ar[ll]\cdot&&\cdot\ar[ll]\\
 &\alpha& &\beta&\\ \ar[uu]\cdot&&\ar[ll]\cdot\ar[uu]&&\ar[uu]\cdot\ar[ll] \\
 &\gamma& &\delta& \\ \ar[uu]\cdot&&\ar[ll]\cdot\ar[uu]&&\ar[ll]\cdot\ar[uu]}$$
 the interchange law holds, that is,}
 $(\alpha \circ_{\mathrm{h}} \beta)\circ_{\mathrm{v}} (\gamma\circ_{\mathrm{h}}\delta) =(\alpha \circ_{\mathrm{v}} \gamma)\circ_{\mathrm{h}}
  (\beta\circ_{\mathrm{v}}\delta).$
 \end{description}

 A {\em double groupoid} is a double category such that all the four
 component categories are groupoids. We shall use the following
 notation for inverses in a double groupoid: $f^{\text{-}\!1_{\mathrm{h}}}$
 denotes the inverse of a horizontal morphism $f$, and
 $u^{\text{-}\!1_{\!\mathrm{v}}}$ denotes the inverse of a vertical morphism
 $u$. For any square $\alpha$ as in (\ref{alfa1}), the first one of
 $$\xymatrix@C=-3pt@R=-3pt{b&&\ar[ll]_{g^{\text{-}\!1_{\mathrm{h}}}}d\\&\alpha^{\text{-}\!1_{\mathrm{h}}}&\\a\ar[uu]^{u}&&\ar[ll]^{f^{\text{-}\!1_{\mathrm{h}}}}c\ar[uu]_{w},} \hspace{0.6cm}
 \xymatrix@C=-3pt@R=-3pt{c&&\ar[ll]_{f}a\\&\alpha^{\text{-}\!1_{\!\mathrm{v}}}&\\d\ar[uu]^{w^{\text{-}\!1_{\!\mathrm{v}}}}&&
 \ar[ll]^{g}b\ar[uu]_{u^{\text{-}\!1_{\!\mathrm{v}}}},}\hspace{0.6cm}
 \xymatrix@C=-3pt@R=-3pt{a&&\ar[ll]_{f^{\text{-}\!1_{\mathrm{h}}}}c\\&\alpha^{\text{-}\!1}
 &\\b\ar[uu]^{u^{\text{-}\!1_{\!\mathrm{v}}}}&&\ar[ll]^{g^{\text{-}\!1_{\mathrm{h}}}}d\ar[uu]_{w^{\text{-}\!1_{\!\mathrm{v}}}},}
 $$
 is the inverse of $\alpha$ in the horizontal groupoid of squares,
 the second one denotes the inverse of $\alpha$ in the vertical
 groupoid of squares, and the third one is the square
 $(\alpha^{\text{-}1_{\mathrm{h}}})^{\text{-}1_{\!\mathrm{v}}}=(\alpha^{\text{-}1_{\!\mathrm{v}}})^{\text{-}1_{\mathrm{h}}}$,
 which is denoted simply by $\alpha^{\text{-}1}$.

 \vspace{0.2cm}
 The double groupoids we are interested in satisfy the condition below.

 \begin{description}
 \item[Filling condition]
  {\em Any filling problem $$\xymatrix@C=-1pt@R=-1pt{\cdot && \cdot\ar[ll]_g\\ &\scriptstyle{\exists ?}&\\ \cdot \ar@{.>}[uu] & &  \cdot \ar@{.>}[ll]\ar[uu]_{u}}$$
  has a solution; that is, for any horizontal morphism $g$ and any vertical morphism $u$ such that $\mathrm{s}^{\mathrm{h}}g=\mathrm{t}^{\hspace{-1pt}\mathrm{v}}u$, there is a square
  $\alpha$ with $\mathrm{s}^{\mathrm{h}}\alpha=u$ and $\mathrm{t}^{\hspace{-1pt}\mathrm{v}}\alpha=g$.}
 \end{description}

 As we recalled in the introduction, this filling condition on double groupoids is
  often satisfied for those double groupoids arising in algebraic topology.
 Further below, in Sections \ref{dgt} and \ref{lad}, we post two new homotopical double groupoid
 constructions that relevant to our  deliberations: one, $\dpi X$, for topological spaces $X$, and the other, $\adj K$, for bisimplicial
 sets $K$, both yielding double groupoids satisfying the filling condition.

 The remainder of this section is devoted to defining {\em homotopy groups}, $\pi_i(\mathcal{G},a)$, for double groupoids $\mathcal{G}$ satisfying the filling condition. The useful observation below is
 a direct consequence of  \cite[Lemma 1.12]{a-n-2}.

 \begin{lemma}\label{fc}
 A double groupoid $\mathcal{G}$ satisfies the filling condition if and only if any filling
 problem  such as the one below has a solution.
 $$\xymatrix@C=-1pt@R=-1pt{\cdot && \cdot\ar@{.>}[ll]\\ &\scriptstyle{\exists ?}&\\ \cdot \ar[uu]^w & &  \cdot \ar[ll]^f\ar@{.>}[uu]_{\ \textstyle{,}}}\hspace{0.4cm} \xymatrix@C=-1pt@R=-1pt{\cdot && \cdot\ar@{.>}[ll]\\ &\scriptstyle{\exists ?}&\\ \cdot \ar@{.>}[uu] & &  \cdot \ar[ll]^f\ar[uu]_{u\ \textstyle{,}}}\hspace{0.4cm} \xymatrix@C=-1pt@R=-1pt{\cdot && \cdot\ar[ll]_g\\ &\scriptstyle{\exists ?}&\\ \cdot \ar[uu]^w & &  \cdot \ar@{.>}[ll]\ar@{.>}[uu]_{\ \textstyle{,}}}$$
  \end{lemma}

Hereafter, we assume $\mathcal{G}$ is a double groupoid satisfying the filling condition.

\vspace{0.2cm}
\subsection{The pointed sets $\pi_0(\mathcal{G},a)$.}\label{s21}~

We state that two objects $a,b$ of $\mathcal{G}$ are {\em connected} whenever there is a pair of
morphisms $(g,u)$ in $\mathcal{G}$ of the form $$\xymatrix@R=8pt@C=10pt{b &\cdot \ar[l]_g \\ & a \ar[u]_{u\ \textstyle{,}} }$$ that is, where $g$ is a horizontal morphism and $u$ a vertical morphism such that $\mathrm{s}^{\mathrm{h}}g=\mathrm{t}^{\mathrm{v}}u$,  $\mathrm{t}^{\mathrm{h}}g=b$, and $\mathrm{s}^{\mathrm{v}}u=a$. Because of the filling condition, this is
equivalent to saying that there is a square in $\mathcal{G}$ of the form
$$\xymatrix@C=-1pt@R=-1pt{b && \cdot\ar[ll]_g\\ &\alpha&\\ \cdot \ar[uu]^w & & a \ar[ll]^f\ar[uu]_{u\ \textstyle{,}}}$$
and it is also equivalent to saying that there is a pair of matching morphisms $(w,f)$ as $$\xymatrix@R=10pt@C=10pt{b & \\ \cdot \ar[u]^-w & a. \ar[l]_-f }$$

If $a$ and $b$ are recognized as being connected by means of the pair of morphisms $(g,u)$ as above, then the pair
$(u^{\text{-}1_\mathrm{v}},g^{\text{-}1_\mathrm{h}})$ shows that $b$ is connected to $a$. Hence,  being connected is a
symmetric relation on the set of objects of $\mathcal{G}$. This relation is clearly reflexive thanks to the identity
morphisms $(\mathrm{I}\mathrm{^h}a,\mathrm{I}\mathrm{^v}a)$, and it is also transitive. Suppose $a$ is connected with $b$,
which itself is connected with another object $c$. Then, we have morphisms $u, f, v, g$ as in the diagram
$$
\xymatrix@C=12pt@R=12pt{ c&\ar[l]_g\cdot&\ar@{.>}[l]_{g'}\cdot\\&b\ar[u]^{v}\ar@{}[ur]|-{\textstyle \beta}&\cdot\ar[l]^f\ar@{.>}[u]_{u'}\\&&a\ar[u]_u}
$$
where $\beta$ is  any square with $\mathrm{t}^{\mathrm{h}}\beta=v$ and $\mathrm{s}^{\mathrm{v}}\beta=f$, and the dotted $g'$ and $u'$ are the other sides of $\beta$. Consequently, on considering the pair of composites $(g\circ_\mathrm{h}g',u'\circ_\mathrm{v}u)$, we see that $a$
and $c$ are connected.

Therefore,  being  connected establishes an equivalence relation on  the objects of the double groupoid and,
associated to $\mathcal{G}$, we take

\begin{equation}  \mbox{$\pi_0\mathcal{G}=$ {\em the set of connected classes of objects of }
$\mathcal{G}$,}\end{equation} and  we write $\pi_0(\mathcal{G},a)$
for the set $\pi_0\mathcal{G}$ pointed with the class $[a]$ of an
object $a$ of $\mathcal{G}$.

\vspace{0.2cm}
\subsection{The groups $\pi_1(\mathcal{G},a)$}~

Let $a$ be any given object of $\mathcal{G}$, and let
\begin{equation}
\mathcal{G}(a)=\left \{\parbox{57pt}{\xymatrix@R=10pt@C=10pt{a &x \ar[l]_g \\ & a  \ar[u]_{u}}}\right\}
\end{equation}
be the set of all pairs of morphisms $(g,u)$, where $g$ is a horizontal morphism and $u$ a vertical morphism in
$\mathcal{G}$ such that  $\mathrm{t}^{\mathrm{h}}g=a=\mathrm{s}^{\mathrm{v}}u$ and $\mathrm{s}^{\mathrm{h}}g=
\mathrm{t}^{\mathrm{v}}u$.

Define a relation $\sim$ on $\mathcal{G}(a)$ by the rule $(g,u)\sim (g',u')$ if and only if there are two squares
$\alpha$ and $\alpha' $ in $ \mathcal{G}$ of the form $$\xymatrix@C=10pt@R=10pt
        {           a  & \cdot\ar[l]_g      \\
                \cdot\ar[u]^w \ar@{}[ru]|{\textstyle \alpha}& a\ar[l]^f\ar[u]_u }\hspace{0.5cm}
                \xymatrix@C=10pt@R=10pt
        {           a  & \cdot\ar[l]_{g'}      \\
                \cdot\ar[u]^w \ar@{}[ru]|(0.55){\textstyle \alpha'}& a\ar[l]^f\ar[u]_{u'} }$$
that is, such that $\mathrm{t}^{\mathrm{h}}\alpha=\mathrm{t}^{\mathrm{h}}\alpha',\
\mathrm{s}^{\mathrm{v}}\alpha=\mathrm{s}^{\mathrm{v}}\alpha',\ \mathrm{s}^{\mathrm{h}}\alpha=u,\
\mathrm{s}^{\mathrm{h}}\alpha'=u',\ \mathrm{t}^{\mathrm{v}}\alpha=g$, and $\mathrm{t}^{\mathrm{v}}\alpha'=g'$.

\begin{lemma}
The relation $\sim$ is an equivalence.
\end{lemma}
\begin{proof}
Since $\mathcal{G}$ satisfies the filling condition, the relation is clearly reflexive, and it is obviously  symmetric.
To prove transitivity, suppose $(g,u)\sim(g',u')\sim(g'',u'')$, so that there are squares $\alpha,\alpha',\beta$ and
$\beta'$ as below.
$$\xymatrix@C=10pt@R=10pt
        {           a  & \cdot\ar[l]_g      \\
                \cdot\ar[u]^w \ar@{}[ru]|{\textstyle \alpha}& a\ar[l]^f\ar[u]_u }\hspace{0.5cm}
                \xymatrix@C=10pt@R=10pt
        {           a  & \cdot\ar[l]_{g'}      \\
                \cdot\ar[u]^w \ar@{}[ru]|(0.55){\textstyle \alpha'}& a\ar[l]^f\ar[u]_{u'} }\hspace{0.5cm}
\xymatrix@C=10pt@R=10pt
        {           a  & \cdot\ar[l]_{g'}      \\
                \cdot\ar[u]^{w'} \ar@{}[ru]|(0.55){\textstyle \beta}& a\ar[l]^{f'}\ar[u]_{u'} }\hspace{0.5cm}
                \xymatrix@C=10pt@R=10pt
        {           a  & \cdot\ar[l]_{g''}      \\
                \cdot\ar[u]^{w'} \ar@{}[ru]|(0.55){\textstyle \beta'}& a\ar[l]^{f'}\ar[u]_{u''} }$$
Then, we have the horizontally composable squares
$$
\xymatrix@C=20pt@R=14pt{a&\ar[l]_{g'}\cdot&a\ar[l]_-{{g'}^{\text{-}1_\mathrm{h}}}&\cdot\ar[l]_{g}\\
\cdot\ar[u]^{w'}\ar@{}[ru]|{\textstyle \beta}&a\ar[u] \ar[l]^{f'}\ar@{}[ru]|(0.55){\textstyle
{\alpha'}^{\text{-}1_\mathrm{h}}}&\cdot\ar[l]^{f^{\text{-}1_\mathrm{h}}}\ar[u]\ar@{}[ur]|{\textstyle
\alpha}&a\ar[l]^f\ar[u]_{u} }
$$
whose composition $\beta\!\circ_\mathrm{h}\!\alpha'^{\text{-}1_{\mathrm{h}}}\!\circ_\mathrm{h}\!\alpha$ and $\beta'$
show that $(g,u)\sim(g'',u'')$.
\end{proof}

We write $[g,u]$ for the $\sim$-equivalence class of $(g,u)\in\mathcal{G}(a)$. Now we define a product on
\begin{equation}
\pi_1(\mathcal{G},a)\!:=\mathcal{G}(a)\diagup\!\! \sim
\end{equation}
as follows: given $[g_1,u_1],\ [g_2,u_2]\in \pi_1(\mathcal{G},a)$, by the filling condition on $\mathcal{G}$, we can
choose a square $\gamma$ with $\mathrm{s}^\mathrm{v}\gamma=g_2$ and $\mathrm{t}^\mathrm{h}\gamma=u_1$ so that we have
a configuration in $\mathcal{G}$ of the form
$$
\xymatrix@C=12pt@R=12pt{ a&\ar[l]_{g_1}\cdot&\ar[l]_{g}\cdot\\&a\ar[u]\ar@{}@<-2pt>[u]^-{u_1}\ar@{}[ur]|{\textstyle
\gamma}&\cdot\ar[l]^{g_2}\ar[u]_{u}\\&&a\ar[u]_{u_2}}
$$
where $g=\mathrm{t}^\mathrm{v}\gamma$ and $u=\mathrm{s}^\mathrm{h}\gamma$. Then we define
\begin{equation}
[g_1,u_1]\circ [g_2,u_2]=[g_1\circ_{\mathrm{h}}g,u\circ_{\mathrm{v}}u_2]
\end{equation}
\begin{lemma}
The product is well defined.
\end{lemma}
\begin{proof}
Let $[g_1,u_1]=[g_1',u_1'],\ [g_2,u_2]=[g_2',u_2']$ be elements of $\pi_1(\mathcal{G},a)$. Then, there are squares
$$\xymatrix@C=10pt@R=10pt
        {           a  & \cdot\ar[l]_{g_1}      \\
                \cdot\ar[u]^{w_1} \ar@{}[ru]|{\textstyle \alpha}& a\ar[l]^{f_1}\ar[u]_{u_1} }\hspace{0.5cm}
                \xymatrix@C=10pt@R=10pt
        {           a  & \cdot\ar[l]_{g'_1}      \\
                \cdot\ar[u]^{w_1} \ar@{}[ru]|(0.55){\textstyle \alpha'}& a\ar[l]^{f_1}\ar[u]_{u'_1} }\hspace{0.5cm}
\xymatrix@C=10pt@R=10pt
        {           a  & \cdot\ar[l]_{g_2}      \\
                \cdot\ar[u]^{w_2} \ar@{}[ru]|(0.55){\textstyle \beta}& a\ar[l]^{f_2}\ar[u]_{u_2} }\hspace{0.5cm}
                \xymatrix@C=10pt@R=10pt
        {           a  & \cdot\ar[l]_{g'_2}      \\
                \cdot\ar[u]^{w_2} \ar@{}[ru]|(0.55){\textstyle \beta'}& a\ar[l]^{f_2}\ar[u]_{u'_2} }$$
and choosing squares $\gamma$ and $\gamma'$ as in
$$\xymatrix@C=12pt@R=12pt
        {           \cdot  & \cdot\ar[l]_g      \\
                a\ar[u]^{u_1} \ar@{}[ru]|{\textstyle \gamma}& \cdot\ar[l]^{g_2}\ar[u]_u }\hspace{0.5cm}
                \xymatrix@C=12pt@R=12pt
        {           \cdot & \cdot\ar[l]_{g'}      \\
                a\ar[u]^{u'_1} \ar@{}[ru]|(0.55){\textstyle \gamma'}& \cdot\ar[l]^{g'_2}\ar[u]_{u'} }$$
we have $[g_1,u_1]\circ [g_2,u_2]=[g_1\circ_\mathrm{h}g,u\circ_\mathrm{v}u_2]$ and
$[g_1',u_1']\circ[g_2',u_2']=[g_1'\circ_\mathrm{h}g',u'\circ_\mathrm{v}u_2']$. Now, letting  $\theta$ be any square
with $\mathrm{t}^\mathrm{v}\theta=f_1$ and $\mathrm{s}^\mathrm{h}\theta=w_2$,
 we have squares as in
$$\xymatrix@C=-1pt@R=-1pt{a&&\ar[ll]_{g_1}\cdot&&\cdot\ar[ll]_g\\
&\alpha& &\gamma&\\
\ar[uu]^{w_1}\cdot&&\ar[ll]a\ar[uu]&&\ar[uu]_u\cdot\ar[ll] \\
&\theta& &\beta& \\
\ar[uu]\cdot&&\ar[ll]\cdot\ar[uu]&&\ar[ll]^{f_2}a\ar[uu]_{u_2}&&}\hspace{0.4cm}
\xymatrix@C=-1pt@R=-2pt{a&&\ar[ll]_{g_1'}\cdot&&\cdot\ar[ll]_{g'}\\
&\alpha'& &\gamma'&\\
\ar[uu]^{w_1}\cdot&&\ar[ll]a\ar[uu]&&\ar[uu]_{u'}\cdot\ar[ll] \\
&\theta& &\beta'& \\
\ar[uu]\cdot&&\ar[ll]\cdot\ar[uu]&&\ar[ll]^{f_2}a\ar[uu]_{u_2'}}$$ whose corresponding composites
$(\alpha\circ_\mathrm{h}\gamma)\circ_\mathrm{v}(\theta\circ_\mathrm{h}\beta)$ and
$(\alpha'\circ_\mathrm{h}\gamma')\circ_\mathrm{v}(\theta\circ_\mathrm{h}\beta')$ show that $[g_1\circ_
\mathrm{h}g,u\circ_\mathrm{v}u_2]=[g_1'\circ_\mathrm{h}g',u'\circ_\mathrm{v}u_2']$, as required.
\end{proof}
\begin{lemma}
The given multiplication turns $\pi_1(\mathcal{G},a)$ into a group.
\end{lemma}
\begin{proof}
To see the associativity, let $[g_1,u_1],\ [g_2,u_2],\ [g_3,u_3]\in\pi_1(\mathcal{G},a)$, and choose $\gamma,\gamma'$
and $\gamma''$ any three squares as in the diagram (\ref{eas}) below. Then we have $$([g_1,u_1] \circ [g_2,u_2])\circ
[g_3,u_3]=[g_1\circ_\mathrm{h}g\circ_\mathrm{h}g',u\circ_\mathrm{v}u'\circ_\mathrm{v}u_3]=[g_1,u_1]\circ
([g_2,u_2]\circ[g_3,u_3]).$$
\begin{equation}\label{eas}
\xymatrix@R=12pt@C=12pt{a&\ar[l]_{g_1}\cdot&\cdot\ar[l]_{g}&\cdot\ar[l]_{g'}\\
&a\ar@{}[u]<-2pt>^(0.4){u_1}\ar[u]\ar@{}[ru]|{\textstyle \gamma}&\ar@{}[ru]|(0.55){\textstyle
\gamma'}\cdot\ar[l]^(0.4){g_2}\ar[u]&\cdot\ar[l]\ar[u]_u\\
&&a\ar@{}[u]<-2pt>^(0.3){u_2}\ar[u]\ar@{}[ru]|(0.55){\textstyle \gamma''}&\cdot\ar[l]^(0.4){g_3}\ar[u]_{u'}\\
&&&a\ar[u]_{u_3}}\end{equation}

The identity of $\pi_1(\mathcal{G},a)$ is $[\mathrm{I}^\mathrm{h}a,\mathrm{I}^\mathrm{v}a]$. In effect, if
$[g,u]\in\pi_1(\mathcal{G},a)$, then the diagrams
$$
\xymatrix@C=12pt@R=12pt{
a&\ar[l]_{g}x&\ar@<-0.1ex>@{-}[l]\ar@<0.1ex>@{-}[l]x\\&\ar[u]\ar@{}@<-2pt>[u]^-{u}\ar@{}[ur]|{\textstyle
\mathrm{I}^{\mathrm{h}}u}a&a\ar@<-0.1ex>@{-}[l]\ar@<0.1ex>@{-}[l]\ar[u]_{u}\\&&a\ar@<-0.1ex>@{-}[u]\ar@<0.1ex>@{-}[u]
}\hspace{0.4cm}\xymatrix@C=12pt@R=12pt{ a&a \ar@<-0.1ex>@{-}[l]\ar@<0.1ex>@{-}[l]&\ar[l]_g x\\&\ar@<-0.1ex>@{-}[u]
\ar@<0.1ex>@{-}[u]\ar@{}[ur]|{\textstyle
\mathrm{I}^{\mathrm{v}}g}a&x\ar@<-0.1ex>@{-}[u]\ar@<0.1ex>@{-}[u]\ar[l]^g\\&&a\ar[u]_u }
$$
show that
$$[g,u]\circ[\mathrm{I}^\mathrm{h}a,\mathrm{I}^\mathrm{v}a]=[g\circ_\mathrm{h}\mathrm{I}^\mathrm{h}x,u\circ_\mathrm{v}
\mathrm{I}\mathrm{^v}a]=[g,u]=[\mathrm{I}^\mathrm{h}a\circ_\mathrm{h}g,
\mathrm{I}\mathrm{^v}x\circ_\mathrm{v}u]=[\mathrm{I}^\mathrm{h}a,\mathrm{I}^\mathrm{v}a]\circ[g,u].$$

Finally, to see the existence of inverses, let $[g,u]\in\pi_1(\mathcal{G},a)$. By choosing any square $\alpha$ with
$\mathrm{t}^\mathrm{h}\alpha=u^{\text{-}1_\mathrm{v}}$ and $\mathrm{s}^\mathrm{v}\alpha=g^{\text{-}1_\mathrm{h}}$, that
is, of the form $$\xymatrix@R=10pt@C=10pt{a & \cdot\ar[l]_{f}
\\ \ar[u]^{u^{\text{-}1_\mathrm{v}}} \ar@{}[ru]|{\textstyle \alpha}\cdot &a  \ar[l]^{g^{\text{-}1_\mathrm{h}}}
 \ar[u]_{v}}$$ we find $[f,v]\!:=[\mathrm{t}^\mathrm{v}\alpha,\mathrm{s}^\mathrm{h}\alpha]\in\pi_1(\mathcal{G},a)$.
 Since the diagrams
$$
\xymatrix@C=12pt@R=12pt{
a&\ar[l]_{g}\cdot&\ar[l]_{g^{\text{-}1_\mathrm{h}}}a\\&\ar[u]\ar[u]^-{u}\ar@{}[ur]|(0.55){\textstyle
\alpha^{\text{-}1_\mathrm{v}}}a&\cdot\ar[l]^f\ar[u]_{v^{\text{-}1_\mathrm{v}}}\\&&a\ar[u]_v}\hspace{0.4cm}
\xymatrix@C=12pt@R=12pt{
a&\ar[l]_{f}\cdot&\ar[l]_{f^{\text{-}1_\mathrm{h}}}a\\&\ar[u]\ar[u]^-{v}\ar@{}[ur]|(0.55){\textstyle
\alpha^{\text{-}1_\mathrm{h}}}a&\cdot\ar[l]^g\ar[u]_{u^{\text{-}1_\mathrm{v}}}\\&&a\ar[u]_u}
$$
show that $[g,u]\circ [f,v]=[\mathrm{I}^\mathrm{h}a,\mathrm{I}^\mathrm{v}a]=[f,v]\circ [g,u]$, we have
$[g,u]^{\text{-}1}=[f,v]$.
\end{proof}
\vspace{0.2cm}
\subsection{The abelian groups $\pi_i(\mathcal{G},a)$, $i\geq 2$.}~

These are easier to define than the previous ones. For $i=2$, as in \cite[Section 2]{b-s}, we take
\begin{equation}
\pi_2(\mathcal{G},a)=\left\{\parbox{55pt}{\xymatrix@R=-1pt@C=-1pt{a& & a\ar@<-0.1ex>@{-}[ll]\ar@<0.1ex>@{-}[ll]\\ &
\alpha & \\ a\ar@<-0.1ex>@{-}[uu]\ar@<0.1ex>@{-}[uu] & & a\ar@<-0.1ex>@{-}[ll]\ar@<0.1ex>@{-}[ll]
\ar@<-0.1ex>@{-}[uu]\ar@<0.1ex>@{-}[uu]\\}}\right\}
\end{equation}
the set of all squares $\alpha\in\mathcal{G}$ whose boundary edges are
$\mathrm{s}^\mathrm{h}\alpha=\mathrm{t}^\mathrm{h}\alpha=\mathrm{I}^\mathrm{v}a$ and
$\mathrm{s}^\mathrm{v}\alpha=\mathrm{t}^\mathrm{v}\alpha=\mathrm{I}^\mathrm{h}a$.

By the general Eckman-Hilton argument, it is a consequence of the interchange law that, on $\pi_2(\mathcal{G},a)$,
operations $\circ_\mathrm{h}$ and $\circ_\mathrm{v}$ coincide and are commutative. In effect, for
$\alpha,\beta\in\pi_2(\mathcal{G},a)$,
\begin{equation*}
\begin{split}
\alpha\circ_\mathrm{h}\beta& =(\alpha\circ_\mathrm{v}\mathrm{I}a)\circ_\mathrm{h}(\mathrm{I}a\circ_\mathrm{v}\beta)=
(\alpha\circ_\mathrm{h}\mathrm{I}a)\circ_\mathrm{v}(\mathrm{I}a\circ_\mathrm{h}\beta)\\
& =\alpha\circ_\mathrm{v}\beta\\
& =(\mathrm{I}a\circ_\mathrm{h}\alpha)\circ_\mathrm{v}(\beta\circ_\mathrm{h}\mathrm{I}a)=
(\mathrm{I}a\circ_\mathrm{v}\beta)\circ_\mathrm{h}(\alpha\circ_\mathrm{v}\mathrm{I}a)\\
&=\beta\circ_\mathrm{h}\alpha.
\end{split}
\end{equation*}

Therefore, $\pi_2(\mathcal{G},a)$ is an abelian group with product
\begin{equation}
\alpha\circ_\mathrm{h}\beta=\alpha\circ_\mathrm{v}\beta\ ,
\end{equation}
identity  $\mathrm{I}a=\mathrm{I}^\mathrm{v}\mathrm{I}^\mathrm{h}a$, and inverses
$\alpha^{\text{-}1_\mathrm{h}}=\alpha^{\text{-}1_\mathrm{v}}$.

The higher homotopy groups of the double groupoid are defined to be trivial, that is,
\begin{equation}
\pi_i(\mathcal{G},a)=0\quad \text{if}\quad i\geq 3.
\end{equation}

\subsection{Weak equivalences}\label{we}~

A {\em double functor} $F:\mathcal{G}\to \mathcal{G'}$ between double categories takes objects, horizontal and vertical
morphisms, and squares in $\mathcal{G}$ to objects, horizontal and vertical morphisms, and squares in $\mathcal{G'}$,
respectively, in such a way that all the structure categories are preserved.

Clearly, each double functor $F:\mathcal{G}\to \mathcal{G'}$,
between double groupoids satisfying  the filling condition, induces
maps (group homomorphisms if $i>0$)
$$\pi_iF:\pi_i(\mathcal{G},a)\to \pi_i(\mathcal{G'},Fa)$$  for $i\geq 0$ and $a$ any object of
$\mathcal{G}$.  Call such a double functor a {\em weak equivalence}
if it induces isomorphisms $\pi_iF$ for all integers $i\geq 0$.

\section{A homotopy double groupoid for topological spaces.}\label{dgt}

Our aim here is to provide a new construction of a double groupoid for a topological space that, as we will see
later, captures the  homotopy 2-type of the space. For any given space $X$, the construction of this
{\em homotopy double groupoid}, denoted by $\dpi X$, is as follows:

The objects in $\dpi X$ are the paths in $X$, that is, the continuous maps ${u\!:\!I\!=\![0,1]\!\to\! X}$.

The groupoid of horizontal morphisms in $\dpi X$ is the category
with a unique morphism between  each pair $(u',u)$ of paths in $X$
such that $u'(1)=u(1)$, and, similarly, the  groupoid of vertical
morphisms in $\dpi X$ is the category having a unique morphism
 between each pair  $(v,u)$ of paths in $X$ such that $v(0)=u(0)$.

A square in $\dpi X$, $[\alpha]$, with a boundary as in
\begin{equation}\label{[alfa]}
\parbox{90pt}{\xymatrix@C=-4pt@R=-4pt{v'&&\ar[ll] v\\&[\alpha]&\\u'\ar[uu] &&\ar[ll] u\ar[uu]}}
\end{equation}
is the equivalence class, $[\alpha]$, of a map
$\alpha\!:\!I^2\!\to\! X$ whose effect on the boundary
$\partial(I^2)$ is such that $\alpha(x,0)=u(x),\ \alpha(0,y)=v(y),\
\alpha(1,1-y)=u'(y)$, and $\alpha(1-x,1)=v'(x)$, for $x,y\in I$. We
 call such an application a ``square in X" and draw it as
\begin{equation}\label{alfa}\xymatrix@C=0pt@R=0pt{\cdot&&\ar[ll]_{v'} \ar[dd]^(.4){u'}\cdot\\&
\alpha&\\ \cdot\ar[uu]^v \ar[rr]_u&&\cdot}\end{equation}
Two such mappings $\alpha, \alpha'$ are equivalent, and then represent the same square in $\dpi X$, whenever they are related by a homotopy relative to the sides of the square,  that is, if there exists a continuous map $H:I^2\times I\to X$ such that  $H(x,y,0)=\alpha(x,y),\ H(x,y,1)=\alpha'(x,y),\ H(x,0,t)=u(x),\ H(0,y,t)=v(y),\ H(x,1,t)=v'(1-x)$ and $H(1,y,t)=u'(1-y)$, for $x,y,t\in I$.

Given the squares in $\dpi X$
$$\xymatrix@C=-4pt@R=-4pt
        {           &       & w'            &       & w\ar[ll]          \\
                    &       &               & [\beta]   &               \\
            v''     &       & v'\ar[ll]\ar[uu]  &       &v\ar[ll]\ar[uu]    \\
                    & [\alpha']&                &[\alpha]   &               \\
            u''\ar[uu]  &       & u'\ar[ll]\ar[uu]  &       & u\ar[ll]\ar[uu]   \\}$$
the corresponding composite squares
\begin{center}
$\xymatrix@C=-6pt@R=1pt
        {           v''     &                       &   v\ar[ll]        \\
                    &  [\alpha']\!\circ_{\mathrm{h}}\![\alpha] &               \\
                u''\ar[uu] &                    & u\ar[ll]\ar[uu] &    \\}$
$\xymatrix@C=-6pt@R=1pt
        {           w'  &                       &   w\ar[ll]        \\
                    &  [\beta]\!\circ_{\mathrm{v}}\![\alpha]   &               \\
                u'\ar[uu] &                 & u\ar[ll]\ar[uu]   \\}$
\end{center}
are defined to be those represented by the squares in $X$
$$
\xymatrix@C=10pt@R=10pt{\cdot&&\cdot\ar[ll]_{v''}\ar[dd]^(.4){u''}\\ &\cdot\ar[ul]\ar@<3pt>@{}[ul]_(.4){v'}\ar[dr]\ar@<-3pt>@{}[dr]^(.4){u'}& \\
\cdot\ar[uu]^v\ar[rr]_u \ar@{}[ru]|(.6){\textstyle \alpha}\ar@{}[rruu]|(.75){\textstyle \alpha'}&&\cdot}
\hspace{0.1cm}
\xymatrix@C=10pt@R=10pt{\cdot\ar@{}[rd]|(.6){\textstyle \beta}\ar@{}[rrdd]|(.7){\textstyle \alpha}&&\cdot\ar[ll]_{w'}\ar[dd]^{u'}\ar[dl]\ar@<5pt>@{}[dl]_(.4){v'}\\ &\cdot& \\
\cdot\ar[uu]^w\ar[rr]_u \ar[ur]\ar@<-4pt>@{}[ur]^(.5){v}&&\cdot}
$$
obtained, respectively, by pasting $\alpha'$ with $\alpha$, and $\beta$ with $\alpha$, along their common pair of sides. That is,
\begin{equation}
[\alpha']\circ_{\mathrm{h}}[\alpha]=[\alpha'\circ_{\mathrm{h}}\alpha],
\ \ \ \
[\beta]\circ_{\mathrm{v}}[\alpha]=[\beta\circ_{\mathrm{v}}\alpha]
\end{equation}
where
$(\alpha'\circ_{\mathrm{h}}\alpha)(x,y)=\left\{\begin{array}{lll}
\alpha(2x,x+y)&\text{if} &    x\leq y,\ x+y\leq 1,\\
\alpha(x+y,2y)&\text{if}&  x\geq y,\ x+y\leq 1, \\
\alpha'(x+y-1,2y-1) &\text{if}&  x\leq y,\ x+y\geq 1, \\
\alpha'(2x-1,x+y-1)&\text{if}& x\geq y,\ x+y\geq 1,\\
\end{array}\right.$ \vspace{0.2cm}\newline
and $(\beta\circ_{\mathrm{v}}\alpha)(x,y)=\left\{\begin{array}{lll}
\alpha(2x-1,1-x+y) &\text{if}& x\geq y,\ x+y\geq 1,\\
\alpha(x-y,2y)&\text{if}& x\geq y,\ x+y\leq 1, \\
\beta(1+x-y,2y-1)&\text{if}&x\leq y,\ x+y\geq 1, \\
\beta(2x,y-x)&\text{if}&x\leq y,\ x+y\leq 1. \\
\end{array}\right.$\vspace{0.3cm}

It is not hard to see that both the horizontal and vertical compositions
of squares in $\dpi X$ are well defined. For example, to prove that
$[\alpha]=[\alpha_1]$ and $[\alpha']=[\alpha_1']$ imply
$[\alpha'\circ_{\mathrm{h}}\alpha]=[\alpha_1'\circ_{\mathrm{h}}\alpha_1]$,
let $H,H':I^2\times I\to X$ be homotopies ($rel\ \partial (I^2)$)
from $\alpha$ to $\alpha_1$ and from $\alpha'$ to $\alpha_1'$
respectively. Then, a homotopy $F:I^2\times I\to X$ is defined by
$$F(x,y,t)=\left\{\begin{array}{lll}
H(2x,x+y,t)&\text{if}&x\leq y,\ x+y\leq 1,\\
H(x+y,2y,t)&\text{if}& x\geq y,\ x+y\leq 1, \\
H'(x+y-1,2y-1,t)&\text{if}& x\leq y,\ x+y\geq 1 ,\\
H'(2x-1,x+y-1,t) &\text{if}&x\geq y,\ x+y\geq 1,\\
\end{array}\right.$$
showing that $\alpha'\circ_{\mathrm{h}}\alpha$ and
$\alpha_1'\circ_{\mathrm{h}}\alpha_1$ represent the same square in
$\dpi X$.

The horizontal identity square on a vertical morphism $(v,u)$ is
$$
\xymatrix@C=-3pt@R=-3pt{&v&&\ar@<-0.1ex>@{-}[ll]\ar@<0.1ex>@{-}[ll]
v\\\mathrm{I}^{\mathrm{h}}(v,u)=&&[e^{\mathrm{h}}]&
\\&u\ar[uu] &&\ar@<-0.1ex>@{-}[ll]\ar@<0.1ex>@{-}[ll] u\ar[uu]
&}$$
 where $e^{\mathrm{h}}(x,y)=\left\{ \!\begin{array}{lll}
v(y-x)&\text{if}&x\leq y,\\
u(x-y)&\text{if}&x\geq y,\\
\end{array}\right.$
whereas, for any horizontal morphism $(u',u)$, its corresponding
vertical identity square is
$${\xymatrix@C=-3pt@R=-3pt{&u'&&\ar[ll] u\\{\mathrm{I}}^{\mathrm{v}}(u',u)=&
&[e^{\mathrm{v}}]&\\&u'\ar@<-0.1ex>@{-}[uu]\ar@<0.1ex>@{-}[uu]& &\ar[ll] u\ar@<-0.1ex>@{-}[uu]\ar@<0.1ex>@{-}[uu] &}}$$
where $e^\mathrm{v}(x,y)=\left\{\! \begin{array}{lll}
u(x+y)&\text{if}&x+y\leq 1,\\
u'(2-x-y)&\text{if}&x+y\geq 1.\\
\end{array}\right.$

\begin{theorem}
$\dpi X$ is a double groupoid satisfying the filling condition.
\end{theorem}
\begin{proof}
  The horizontal composition of squares in $\dpi X$ is associative since, for any three composable squares, say \parbox{95pt}{$\xymatrix@C=-2pt@R=-2pt{\cdot&&\ar[ll]\cdot  & &\cdot \ar[ll] & & \cdot\ar[ll]\\
                                    &[\alpha'']& & [\alpha'] & & [\alpha]\\
                                    \cdot \ar[uu] &&\cdot\ar[ll] \ar[uu]  &&\cdot\ar[ll] \ar[uu] &&\cdot\ar[ll] \ar[uu]}$}, a relative homotopy $(\alpha''\circ_{\mathrm{h}}\alpha')\circ_{\mathrm{h}}\alpha \overset{H}\to \alpha''\circ_{\mathrm{h}}(\alpha'\circ_{\mathrm{h}}\alpha)$ is given by the formula \newline

\noindent $H(x,y,t)=$
$$\left\{\begin{array}{lll}
\alpha(\frac{4x}{2-t},\frac{(2+t)x+(2-t)y}{2-t})&  \text{ if} &  {\scriptstyle x\leq y,\ (2-t)(1-y)\geq (2+t)x,}\\[5pt]
\alpha(\frac{(2-t)x+(2+t)y}{2-t},\frac{4y}{2-t})&    \text{ if}  &  {\scriptstyle x\geq y,\ (2-t)(1-x)\geq (2+t)y,} \\[5pt]
\alpha'({\scriptstyle t(1+x-y)+2(x+y-1),x+3y-2+t(1+x-y)} ) &  \text{ if} &  {\scriptstyle x\leq y,\,(2-t)(1-y)\leq  (2+t)x,\,(1+t)x\leq(3-t)(1-y),}\\[5pt]
\alpha'({\scriptstyle 3x+y-2+t(1-x+y),t(1-x+y)+2(x+y-1)})  &  \text{ if}  & {\scriptstyle  x\geq y,\,
(2-t)(1-x)\leq (2+t)y,\,(1+t)y\leq (3-t)(1-x),}\\[5pt]
\alpha''(\frac{x+3y-3+t(1+x-y)}{1+t},\frac{t-3+4y}{1+t})& \text{ if} &{\scriptstyle x\leq y,\ (1+t)x\geq (1-y)(3-t),}\\[5pt]
\alpha''(\frac{t-3+4x}{1+t},\frac{3x+y-3+t(1-x+y)}{1+t})& \text{ if} & {\scriptstyle x\geq y,\ (1+t)y\geq (3-t)(1-x)}.\\
\end{array}\right.$$

And, similarly, we prove the associativity for the vertical
composition of squares in $\dpi X$. For identities, let $[\alpha]$
be any square in $\dpi X$ as in (\ref{[alfa]}). Then, a relative
homotopy between $\alpha$ and
$\alpha\circ_{\mathrm{h}}\!e^\mathrm{h}$ is given by the map
$H:I^2\times I\to X$ defined by
$$H(x,y,t)=\left\{\begin{array}{lll}
v({\scriptstyle {y-x}}) & \text{if} &{\scriptstyle x\leq y,\ x\leq \frac{1}{2}(1-t)(1+x-y)},\\[5pt]
u({\scriptstyle x-y}) & \text{if} &{\scriptstyle x\geq y,\ x\leq \frac{1}{2}(1-t)(1+x-y)},\\[5pt]
\alpha(\frac{x+y-1+t(1+x-y)}{1+t},\frac{2y+t-1}{1+t}) & \text{if} &{\scriptstyle \frac{1}{2}(1-t)(1+x-y)\leq x\leq y},\\[5pt]
\alpha(\frac{2x+t-1}{1+t},\frac{x+y-1+t(1-x+y)}{1+t}) & \text{if} &{\scriptstyle \frac{1}{2}(1-t)(1-x-y)\leq y\leq x}.\\
\end{array}\right.$$
 Therefore, $[\alpha ] \circ_{\mathrm{h}} \mathrm{I}^{\mathrm{h}}(v,u)=[\alpha]$; and similarly we prove the remaining
needed equalities: $ [\alpha]=\mathrm{I}^{\mathrm{h}} \circ_{\mathrm{h}} [\alpha]= [\alpha ] \circ_{\mathrm{v}}
\mathrm{I}^{\mathrm{v}}= \mathrm{I}^{\mathrm{v}} \circ_{\mathrm{v}} [\alpha]$.

Let us now describe inverse squares in $\dpi X$. For any given square $[\alpha]$ as in (\ref{[alfa]}), its respective horizontal and vertical inverses
\begin{center}
$\xymatrix@C=-5pt@R=-2pt
        {           v   &                       &   v'\ar[ll]       \\
                    &  [\alpha]^{\text{-}1_{\mathrm{h}}}           &               \\
                u\ar[uu] &                  & u'\ar[ll]\ar[uu] &,   \\}$
$\xymatrix@C=-5pt@R=-2pt
        {           u'  &                       &   u\ar[ll]        \\
                    &  [\alpha]^{\text{-}1_{\!\mathrm{v}}}           &               \\
                v'\ar[uu] &                 & v\ar[ll]\ar[uu]   \\}$
\end{center}
are represented by the squares in $X$,
$\alpha^{\text{-}1_{\mathrm{h}}}, \alpha^{\text{-}1_v}:I^2\to X$,
defined respectively by the formulas
$$\alpha^{\text{-}1_{\mathrm{h}}}(x,y)=\alpha(1-y,1-x),\hspace{0.5cm}
\alpha^{\text{-}1_{\!\mathrm{v}}}(x,y)=\alpha(y,x).$$

The equality
$[\alpha^{\text{-}1_{\mathrm{h}}}]\circ_{\mathrm{h}}[\alpha]=\mathrm{I}^{\mathrm{h}}(v,u)$
holds, thanks to the homotopy $H:I^2\times I:\to X$ defined by
$$H(x,y,t)=\left\{\begin{array}{lll}
\alpha({\scriptstyle 2x(1-t),(1-2t)x+y}) & \text{if} &{\scriptstyle x\leq y,\ x+y\leq 1},\\[5pt]
\alpha({\scriptstyle x+(1-2t)y,2y(1-t)}) & \text{if} &{\scriptstyle x\geq y,\ x+y\leq 1},\\[5pt]
\alpha({\scriptstyle 2(ty-t-y+1),2(ty-t+1)-x-y}) & \text{if} &{\scriptstyle x\leq y,\ x+y\geq 1},\\[5pt]
\alpha({\scriptstyle (2x-2)t+2-x-y,(2x-2)t+2-2x }) & \text{if} &{\scriptstyle x\geq y,\ x+y\geq 1}.\\
\end{array}\right.$$
And, similarly, one sees the remaining equalities
$[\alpha]\circ_{\mathrm{h}}[\alpha]^{\text{-}1_{\mathrm{h}}}=\mathrm{I}^{\mathrm{h}}$,
$[\alpha]\circ_{\mathrm{v}}[\alpha]^{\text{-}1_{\!\mathrm{v}}}=\mathrm{I}^{\mathrm{v}}$
and
$[\alpha]^{\text{-}1_{\!\mathrm{v}}}\circ_{\mathrm{v}}[\alpha]=\mathrm{I}^{\mathrm{v}}$.

By construction of $\dpi X$, conditions (i) and (ii) in {\bf Axiom
1} are clearly satisfied. For (iii) in {\bf Axiom 1}, we need to
prove that, for any path $u:I\to X$, the equality
$\mathrm{I}^{\mathrm{h}}(u,u)=\mathrm{I}^{\mathrm{v}}(u,u)$ holds.
But this follows from the relative homotopy $H:e^{\mathrm{h}}\to
e^{\mathrm{v}}$ defined by
$$H(x,y,t)=\left\{\begin{array}{lll}
u({\scriptstyle y-x}) & \text{if} &{\scriptstyle x\leq y,\ (1-t)(1-y)\geq (1+t)x},\\[5pt]
u({\scriptstyle 2y-1+t(1+x-y)}) &\text{if}  &{\scriptstyle x\leq y,\ x+y\leq 1,\ (1-t)(1-y)\leq (1+t)x},\\[5pt]
u({\scriptstyle x-y})  & \text{if}  &{\scriptstyle   x\geq y,\ (1-t)(1-x)\geq (1+t)y},\\[5pt]
u({\scriptstyle 2x-1+t(1-x+y)}) & \text{if}  &{\scriptstyle x\geq y,\ x+y\leq 1,\ (1-t)(1-x)\leq (1+t)y},\\[5pt]
u({\scriptstyle 1-2x+t(1+x-y)}) & \text{if}  &{\scriptstyle x\leq y,\ x+y\geq 1,\ (1+t)(1-y)\geq (1-t)x},\\[5pt]
u({\scriptstyle y-x}) & \text{if}  &{\scriptstyle x\leq y,\ x+y\geq 1,\ (1+t)(1-y)\leq (1-t)x},\\[5pt]
u({\scriptstyle 1-2y+t(1-x+y)}) & \text{if}  &{\scriptstyle x\geq y,\ x+y\geq 1,\ (1+t)(1-x)\geq (1-t)y},\\[5pt]
u({\scriptstyle x-y}) & \text{if}  &{\scriptstyle x\geq y,\  (1+t)(1-x)\leq (1-t)y}.\\
\end{array}\right.$$

The given definition of how squares in $\dpi X$ compose makes
the conditions (i) and (ii) in {\bf Axiom 2} clear, and  the
remaining condition (iii) holds since, for any three paths $u,v,w:I\to
X$  with $u(1)=v(1)=w(1)$,  there is a relative homotopy between
$e^{\mathrm{v}}{(w,v)}\circ_{\mathrm{h}} e^{\mathrm{v}}{(v,u)}$ and
$e^{\mathrm{v}}{(w,u)}$, defined by

$$H(x,y,t)=\left\{\begin{array}{lll}
u({\scriptstyle y-x+\frac{4x}{1+t}}) & \text{if}  &{\scriptstyle x\leq y,\ (1+t)(1-y)\geq (3-t)x},\\[5pt]
v({\scriptstyle 2-3x-y+t(1+x-y)}) & \text{if}  &{\scriptstyle x\leq y,\ x+y\leq 1,\ (1+t)(1-y)\leq (3-t)x},\\[5pt]
u({\scriptstyle x-y+\frac{4y}{1+t}})  & \text{if}  &{\scriptstyle   x\geq y,\ (1+t)(1-x)\geq (3-t)y},\\[5pt]
v({\scriptstyle 2-x-3y+t(1-x+y)}) & \text{if}  &{\scriptstyle x\geq y,\ x+y\leq 1,\ (1+t)(1-x)\leq (3-t)y},\\[5pt]
v({\scriptstyle x+3y-2+t(1+x-y)}) & \text{if}  &{\scriptstyle x\leq y,\ x+y\geq 1,\ (3-t)(1-y)\geq (1+t)x},\\[5pt]
w({\scriptstyle y-x+\frac{4(y-1)}{1+t}}) & \text{if} &{\scriptstyle x\leq y,\  (3-t)(1-y)\leq (1+t)x},\\[5pt]
v({\scriptstyle 3x+y-2+t(1-x+y)}) & \text{if}  &{\scriptstyle x\geq y,\ x+y\geq 1,\ (3-t)(1-x)\geq (1+t)y},\\[5pt]
w({\scriptstyle x-y+\frac{4(1-x)}{1+t}}) & \text{if}  &{\scriptstyle x\geq y,\  (3-t)(1-x)\leq (1+t)y}.\\
\end{array}\right.$$
And, similarly,  one proves the equality $\mathrm{I}^{\mathrm{h}}(w,u)
\circ_{\mathrm{v}} \mathrm{I}^{\mathrm{h}}(w,u)=
\mathrm{I}^{\mathrm{h}}(w,u)$ , for any three paths in $X$,
$u,v,w:I\to X$ with $u(0)=v(0)=w(0)$.

Then, it only remains to prove the interchange law in {\bf Axiom 3}.
To do so, let
$$\xymatrix@C=-3pt@R=-3pt
        {   w''         &       & w'\ar[ll]         &       & w\ar[ll]          \\
                    & [\delta]  &               & [\beta]   &               \\
            v''\ar[uu]      &       & v'\ar[ll]\ar[uu]  &       &v\ar[ll]\ar[uu]    \\
                    &[\gamma]&              &[\alpha]   &               \\
            u''\ar[uu]  &       & u'\ar[ll]\ar[uu]  &       & u\ar[ll]\ar[uu]   \\}$$
be squares in $\dpi X$. Then, the required equality follows from the
existence of the relative homotopy
$(\delta\circ_{\mathrm{h}}\beta)\circ_{\mathrm{v}}(\gamma\circ_{\mathrm{h}}\alpha)\to
(\delta\circ_{\mathrm{v}}\gamma)\circ_{\mathrm{h}}(\beta\circ_{\mathrm{v}}\alpha)$
defined by the map $H:I^2\times I\to X$ such that

\hspace{1cm}$H(x,y,t)=$

\noindent $\begin{array}{ll}
\parbox{15pt}{\xy*{1}*\cir<7pt>{}\endxy}\ \alpha({\scriptstyle x+y-2ty,4y})&
  {\scriptstyle 1-x+2ty\geq 5y,\ } {\scriptstyle  x-3y\geq 2ty},\\[3pt]
\parbox{15pt}{\xy*{2}*\cir<7pt>{}\endxy}\ \alpha(\frac{2(x-y)}{1+t},\frac{2(x-tx+y+3ty)}{2+t-t^2})&
{\scriptstyle 2+t-t^2-6x+4tx+2y\geq 8ty,\ } {\scriptstyle (3+2t)y\geq x \geq y},\\[3pt]
\parbox{15pt}{\xy*{3}*\cir<7pt>{}\endxy}\ \alpha(\frac{t^2-2t+2x-2y+4ty}{1-2t+2t^2},\frac{3t^2-t(1+4x)+2(x+y)}{2-4t+4t^2} ) &
{\scriptstyle t^2+t(4x-3)\geq2(x+y-1),\ }{\scriptstyle t^2-2x+6y\geq t(4x+8y-3), }\\[-5pt]
                                            &{\scriptstyle t^2+6x+t(8y-4x-1)\geq
                                            2(1+y)},\\[-3pt]
\parbox{15pt}{\xy*{4}*\cir<7pt>{}\endxy}\ \alpha(\frac{t-2(x+y)}{t-2},\frac{2t(x+3y-1)-8y}{t^2-t-2})  &
{\scriptstyle  1\geq x+y,\ }{\scriptstyle x-1\geq
(2t-5)y,\,}{\scriptstyle 2x-t^2-6y\geq t(3-4x-8y)},\\ [3pt]
\parbox{15pt}{\xy*{5}*\cir<7pt>{}\endxy}\ v'({\scriptstyle 3-t-4x})   &   {\scriptstyle  x\geq y,\ }{\scriptstyle 1\geq x+y,\ }{\scriptstyle 2(x+y-1)\geq t^2+t(4x-3)},\\[3pt]
\parbox{15pt}{\xy*{6}*\cir<7pt>{}\endxy}\ \gamma( {\scriptstyle 4x-3,\,x+y-1-2t(x-1)})& {\scriptstyle 5x+y-5\geq 2t(x-1),}{\scriptstyle 2t(x-1)+3x\geq y+2},\\[3pt]
\parbox{15pt}{\xy*{7}*\cir<7pt>{}\endxy}\ \gamma(\frac{6+t^2-8x+t(6x+2y-7)}{t^2-t-2},\frac{2(x+y-1)}{2-t})&  {\scriptstyle x+y\geq 1,\,}{\scriptstyle 5+2t(x-1)\geq 5x+y,}\\[-5pt]
                                                        & {\scriptstyle 9t+6x\geq 4+t^2+8tx+2y+4ty},\\[-3pt]
\parbox{15pt}{\xy*{8}*\cir<7pt>{}\endxy}\ \gamma(\frac{t+t^2-4ty+2(x+y-1)}{2-4t+4t^2},\frac{1+t^2+4t(x-1)-2x+2y}{1-2t+2t^2})&  {\scriptstyle  t+t^2+2(x+y-1)\geq 4ty,\,}{\scriptstyle t^2+2(1+x-3y)\geq t(8x-y-3),}\\[-5pt]
                                                        &{\scriptstyle 4+t^2-6x+2y\geq t(9-8x-4y)},\\[-3pt]
\parbox{15pt}{\xy*{9}*\cir<7pt>{}\endxy}\ \gamma(\frac{t^2-2(x+y-1)+t(3-6x+2y)}{t^2-t-2},\frac{1+t-2x+2y}{1+t})& {\scriptstyle  8tx+6y-4ty\geq 2+3t+t^2+2x,\,}{\scriptstyle x\geq y,\ }{\scriptstyle 2+y\geq 2t(x-1)+3x},\\[3pt]
\parbox{15pt}{\xy*{10}*\cir<7pt>{}\endxy}\ v'({\scriptstyle 4y-1-t}) &  {\scriptstyle  x\geq y,\ }{\scriptstyle x+y\geq 1,\,}{\scriptstyle 2+4ty\geq t+t^2+2x+2y},\\[3pt]
\parbox{15pt}{\xy*{11}*\cir<7pt>{}\endxy}\ \beta({\scriptstyle 4x,x+y-2tx}) &  {\scriptstyle 1+2tx\geq y+5x,\ }{\scriptstyle y\geq 3x+2tx},\\[3pt]
\parbox{15pt}{\xy*{12}*\cir<7pt>{}\endxy}\ \beta(\frac{2t(y+3x-1)-8x}{t^2-t-2},\frac{t-2(x+y)}{t-2})&  {\scriptstyle  1\geq x+y,\ }{\scriptstyle y+(5-2t)x\geq 1,\,}{\scriptstyle 2y+t(3-4y-8x)-6x\geq t^2},\\[3pt]
\parbox{15pt}{\xy*{13}*\cir<7pt>{}\endxy}\ \beta(\frac{3t^2-t(1+4y)+2(x+y)}{2-4t+4t^2},\frac{t^2-2t+2y-2x+4tx}{1-2t+2t^2})&  {\scriptstyle  t^2+t(4y-3)\geq 2(x+y-1),\,}{\scriptstyle t^2+t(3-4y-8x)+6x\geq 2y,}\\[-5pt]
                                                        & {\scriptstyle t^2+6y-2(1+x)\geq t(1+4y-8x)},\\[-3pt]
\parbox{15pt}{\xy*{14}*\cir<7pt>{}\endxy}\ \beta(\frac{2(y-ty+x+3tx)}{2+t-t^2},\frac{2(y-x)}{1+t})& {\scriptstyle  2+t+4ty+2x\geq t^2+6y+8tx,\,} {\scriptstyle (3+2t)x\geq y\geq x },\\[3pt]
\parbox{15pt}{\xy*{15}*\cir<7pt>{}\endxy}\ v'({\scriptstyle 3-t-4y}) & {\scriptstyle  y\geq x,\ } {\scriptstyle 1\geq x+y,\ }{\scriptstyle 2(x+y-1)\geq t^2+t(3-4y)},\\[3pt]
\parbox{15pt}{\xy*{16}*\cir<7pt>{}\endxy}\ \delta({\scriptstyle  2t(1-y)+x+y-1,4y-3}) &  {\scriptstyle  5y+x\geq 5+2t(y-1),\ }{\scriptstyle 2t(y-1)+3y\geq x+2},\\[3pt]
\parbox{15pt}{\xy*{17}*\cir<7pt>{}\endxy}\ \delta(\frac{2(x+y-1)}{2-t},\frac{6+t^2-8y+t(6y+2x-7)}{t^2-t-2})&
{\scriptstyle x+y\geq 1,\,}
{\scriptstyle 9t+6y-8ty-2x-4tx\geq 4+t^2,}\\[-5pt] &{\scriptstyle 5+2t(y-1)\geq 5y+x},\\[-3pt]
\parbox{15pt}{\xy*{18}*\cir<7pt>{}\endxy}\ \delta(\frac{1+t^2+4t(y-1)-2y+2x}{1-2t+2t^2},
\frac{t+t^2-4tx+2(x+y-1)}{2-4t+4t^2})& {\scriptstyle  t+t^2-4tx\geq 2(1-x-y),\,}{\scriptstyle t^2+2(1+y-3x)
\geq t(8y-4x-3),}\\[-4pt]                                                                                                        &{\scriptstyle 4+t^2-6y+2x\geq t(9-8y-4x)},\\[3pt]
\parbox{15pt}{\xy*{19}*\cir<7pt>{}\endxy}\
\delta(\frac{1+t-2y+2x}{1+t},\frac{t^2-2(x+y-1)+t(3-6y+2x)}{t^2-2-t})&
{\scriptstyle  8ty+6x-4tx\geq 2+3t+t^2+2y,\,}{\scriptstyle y\geq
x,\, }{\scriptstyle 2-3y+x\geq 2t(y-1)},\\ [3pt]
\parbox{15pt}{\xy*{20}*\cir<7pt>{}\endxy}\  v'({\scriptstyle
4y-1-t}) &  {\scriptstyle  y\geq x,\ }{\scriptstyle x+y\geq
1,\,}{\scriptstyle 2+4tx\geq t+t^2+2y+2x},
\end{array}$

\noindent where  parts \parbox{15pt}{\xy*{n}*\cir<6pt>{}\endxy} in the
homotopy $H(x,y,t)$ above correspond to the  areas with $(x,y)\in
I^2$ shown in Figure 1 below.
\begin{figure}[h!]
\label{figura}\hfill \xy
@={(0,0),(50,0),(0,50),(50,50),(0,0),(25,8.33),(16.66,16.66),
(33.33,16.66),(25,8.33),(50,0),(50,50),(41.66,25),(33.33,16.66),
(33.33,33.33),(41.66,25),(50,0),(16.66,33.33),(16.66,16.66),(8.33,25),(0,0),
(0,50),(8.33,25),(16.66,33.33),(33.33,33.33),(25,41.66),(50,50),(0,50),
(25,41.66),(16.66,33.33),(16.66,16.66)}, s0="pre"
@@{;"pre";**@{-}="pre"} @={(4,25)} @@{*{\scriptstyle
11}*\cir<5pt>{}} @={(13,25)} @@{*{\scriptstyle 13}*\cir<5pt>{}}
@={(20,25)} @@{*{\scriptstyle 15}*\cir<5pt>{}} @={(30,25)}
@@{*{\scriptstyle 10}*\cir<5pt>{}} @={(37,25)} @@{*{\scriptstyle
8}*\cir<5pt>{}} @={(46,25)} @@{*{\scriptstyle 6}*\cir<5pt>{}}
@={(25,4)} @@{*{\scriptstyle 1}*\cir<5pt>{}} @={(25,13)}
@@{*{\scriptstyle 3}*\cir<5pt>{}} @={(25,20)} @@{*{\scriptstyle
5}*\cir<5pt>{}} @={(25,30)} @@{*{\scriptstyle 20}*\cir<5pt>{}}
@={(25,37)} @@{*{\scriptstyle 18}*\cir<5pt>{}} @={(25,46)}
@@{*{\scriptstyle 16}*\cir<5pt>{}} @={(16.66,10)} @@{*{\scriptstyle
2}*\cir<5pt>{}} @={(10,16.66)} @@{*{\scriptstyle 14}*\cir<5pt>{}}
@={(33.33,10)} @@{*{\scriptstyle 4}*\cir<5pt>{}} @={(40,16.66)}
@@{*{\scriptstyle 7}*\cir<5pt>{}} @={(33.33,40)} @@{*{\scriptstyle
19}*\cir<5pt>{}} @={(40,33.33)} @@{*{\scriptstyle 9}*\cir<5pt>{}}
@={(16.66,40)} @@{*{\scriptstyle 17}*\cir<5pt>{}} @={(10,33.33)}
@@{*{\scriptstyle 12}*\cir<5pt>{}}
\endxy\hfill\
\caption{ }
\end{figure}

Finally, we observe that $\dpi X$ satisfies the filling condition.
Suppose a configuration of morphisms in $\dpi X$
$$\xymatrix@R=10pt@C=10pt{v' &v \ar[l] \\ & u \ar[u]}$$
is given. This means we have paths $u,v,v':I\to X$ with $u(0)=v(0)$ and $v(1)=v'(1)$. Since the inclusion $\partial I\hookrightarrow I$ is a cofibration, the map $f:(\{0\}\times I)\cup (I\times\partial I)\to X$ with $f(0,t)=v(t)$, $f(t,0)=u(t)$ and $f(t,1)=v'(1-t)$ for $0\leq t\leq 1$, has an extension to a map $\alpha:I\times I\to X$, which precisely represents a square in $\dpi X$ of the form
$$\xymatrix@C=-2pt@R=-2pt{v'&&\ar[ll] v\\&[\alpha]&\\u'\ar[uu] &&\ar[ll] u\ar[uu]}$$
where $u':I\to X$ is the path $u'(t)=\alpha(1,1-t)$. Hence, $\dpi X$ verifies the filling condition.
\end{proof}

In the previous Section \ref{h-g} we  introduced homotopy groups for double groupoids satisfying the filling
condition. The next proposition  provides greater specifics on the relationship between the homotopy groups of the associated homotopy double
groupoid $\dpi X$ to a topological space $X$ and the corresponding for $X$.

\begin{theorem}\label{ps}
For any space $X$, any path $u:I\to X$, and $0\leq i\leq 2$, there
is an isomorphism $$\pi_i(\dpi X,u)\cong\pi_i(X,u(0)).$$
\end{theorem}
\begin{proof}
For any two points $x,y\in X$,  the constant paths $c_x$ and $c_y$ are in the same connected component  of $\dpi X$
if and only if there is a pair of morphisms in $\dpi X$ of the form
$$\xymatrix@R=8pt@C=7pt{c_y &u \ar[l] \\ & c_x \ar[u]}$$ or, equivalently, if and only if there is a
path $u:I\to X$ in $X$ such that $u(1)=y$ and $u(0)=x$. Then, we have an injective map
$$\pi_0X\to \pi_0\dpi X,\ \quad [x]\mapsto [c_x],$$
which is also surjective since, for any path $u$ in $X$, we have a
vertical morphism $u\gets c_{u(0)}$ in  $\dpi X$; whence the
announced bijection $\pi_0 X\cong \pi_0\dpi X$.

Next, we prove that there is an isomorphism $\pi_1(\dpi X,u)\cong\pi_1(X,u(0))$ for any given path $u:I\to X$. To do
 so, we shall use the fundamental groupoid $\Pi X$ of the space $X$; that is, the groupoid whose objects are the points
of $X$ and whose morphisms are the (relative to $\partial I$) homotopy classes $[v]$ of paths $v:I\to X$. Simply by
checking the construction, we see that an element $[(u,v),(v,u)]\in\pi_1(\dpi X,u)$ is determined by a path $v:I\to
X$, with $v(0)=u(0)$ and $v(1)=u(1)$. Moreover, for any other such $v':I\to X$, it holds that
$[(u,v),(v,u)]=[(u,v'),(v',u)]$ in $\pi_1(\dpi X,u)$ if and only if there are squares in $\dpi X$ of the
form $$\xymatrix@C=-3pt@R=-2pt
        {           u   &                       &   v\ar[ll]        \\
                    &  [\alpha]                 &               \\
                w\ar[uu] &                  & u\ar[ll]\ar[uu] &   \\}
\hspace{0.6cm}\xymatrix@C=-3pt@R=-2pt
        {           u   &                       &   v'\ar[ll]       \\
                    &  [\alpha']                &               \\
                w\ar[uu] &                  & u\ar[ll]\ar[uu]   \\}$$
or, equivalently, if and only if there are squares in $X$, $\alpha,\alpha':I^2\to X$ with boundaries as in
\begin{center}
$\xymatrix@C=0pt@R=0pt{\cdot&&\cdot\ar[ll]_{u} \ar[dd]^{w}\\&\alpha&\\ \cdot\ar[uu]^{v} \ar[rr]_{u}&&\cdot &}$
$\xymatrix@C=-1pt@R=-2pt{\cdot&&\cdot\ar[ll]_{u} \ar[dd]^{w}\\&\alpha'&\\ \cdot\ar[uu]^{v'} \ar[rr]_u&&\cdot}$
\end{center}

Since this last condition simply means that, in the fundamental groupoid $\Pi X$, the equality $[v]=[v']$ holds, we
conclude with bijections
$$\xymatrix@R=5pt@C=-4pt{\pi_1(\dpi X,u) & \hspace{-6pt}\cong & \mbox{Hom}_{\Pi X}(u(0),v(1)) & \cong & \pi_1(X,u(0))\\
 [(u,v),(v,u)]\ar@{|->}[rr] &     & [v]\ar@{|->}[rr]       &    & [u]^{\text{-}1}\!\circ[v]}$$

To see that the composite bijection $\phi:[(u,v),(v,u)]\mapsto [u]^{\text{-}1}\!\circ[v]$ is actually an isomorphism,
let $v_1,v_2:I\to X$ be paths in $X$, both from $u(0)$ to $u(1)$. Then, $[(u,v_1),(v_1,u)]\circ
[(u,v_2),(v_2,u)]=[(u,v),(v,u)]$, where $v$ occurs in a configuration such as
$$\xymatrix@C=10pt@R=10pt{u   &v_1\ar[l] &v\ar[l]\\
                             & u \ar[u]\ar@{}[ru]|{\textstyle [\gamma]}&v_2 \ar[u] \ar[l] \\
                             &       &u\ar[u]}$$
for some (any) square $\gamma:I^2\to X$ in $X$ with boundary as
below. $$\xymatrix@R=12pt@C=12pt{\cdot  & \cdot\ar[l]_{v_1}\ar[d]^u
\\ \cdot\ar[u]^{v}\ar[r]_{v_2}\ar@{}[ru]|{\textstyle \gamma}  &\cdot
}$$

It follows that, in $\Pi X$, $[v]=[v_1]\circ [u]^{\text{-}1}\circ [v_2]$ and therefore
\begin{equation*}
\begin{split}
\phi[(u,v_1),(v_1,u)]\circ \phi [(u,v_2),(v_2,u)] & =[u]^{\text{-}1}\circ[v_1]\circ[u]^{\text{-}1}\circ
[v_2]=[u]^{\text{-}1}\circ [v]\\
                                    & =\phi([(u,v_1),(v_1,u)]\circ [(u,v_2),(v_2,u)]).\\
\end{split}
\end{equation*}

Finally, we consider the case $i=2$. Let $u:I\to X$ be any path with
$u(0)=x$. Then, the mapping  $[\alpha]\mapsto
\mathrm{I}^\mathrm{h}(c_x,u)\circ_\mathrm{v}[\alpha]\circ_\mathrm{v}\mathrm{I}^\mathrm{h}(u,c_x)$,
which carries a square $[\alpha]\in\pi_2(\dpi X,u)$, to the
composite of
$$\xymatrix@R=-4pt@C=-4pt{c_x& & c_x\ar@<-0.1ex>@{-}[ll]\ar@<0.1ex>@{-}[ll]\\ & [e^\mathrm{h}] &\\
                         u\ar[uu]& &u\ar[uu]\ar@<-0.1ex>@{-}[ll]\ar@<0.1ex>@{-}[ll]\\
                          & [\alpha] & \\ u\ar@<-0.1ex>@{-}[uu]\ar@<0.1ex>@{-}[uu]& & u\ar@<-0.1ex>@{-}[ll]
                          \ar@<0.1ex>@{-}[ll]
                          \ar@<-0.1ex>@{-}[uu]\ar@<0.1ex>@{-}[uu]\\
                        & [e^\mathrm{h}]& \\c_x\ar[uu] & &c_x\ar@<-0.1ex>@{-}[ll]\ar@<0.1ex>@{-}[ll]\ar[uu]}$$
establishes an isomorphism $\pi_2(\dpi X,u)\cong \pi_2(\dpi X,c_x)$. Now, it is clear that both $\pi_2(\dpi X,c_x)$ and
$\pi_2(X,x)$ are the same abelian group of relative to $\partial I^2$ homotopy classes of maps $I^2\to X$ which are
constant $x$ along the four sides of the square.
\end{proof}

The construction of the double groupoid $\dpi X$ from a space $X$ is
easily seen to be functorial and, moreover, the isomorphisms in
Theorem \ref{ps} above  become natural. Then, we have the next corollary.
\begin{corollary}\label{we1}
A continuous map $f:X\to Y$ is a weak homotopy $2$-equivalence if and only if the induced double functor $\dpi f:\dpi X
\to \dpi Y$ is a weak equivalence.
\end{corollary}

\section{The geometric realization of a double groupoid.}

Hereafter, we shall regard each ordered set $[n]$ as the category with exactly one arrow
$j\to i$ when $0\leq i\leq j\leq n$. Then, a non-decreasing map $[n]\to [m]$ is the same
as a functor.

The geometric realization, or classifying space, of a category
$\mathcal{C}$, \cite{qui2}, is $|\mathcal{C}|\!:=\!|\n \mathcal{C}|$,
the geometric realization of its nerve \cite{grothendieck}
$$
\n\mathcal{C}:\Delta^{\!o}\to \mathbf{Set},\hspace{0.5cm}[n]\mapsto
\mathrm{Func}([n],\mathcal{C}),
$$
that is, the simplicial set whose $n$-simplices are the functors $F:[n]\to\mathcal{C}$,
or tuples of arrows in $\mathcal{C}$
$$
F=\big(F_i \stackrel{\textstyle F_{i,j}}{\longleftarrow} F_j\big)_{^{0\leq i \leq j\leq
n}}
$$
such that $F_{i,j}\circ F_{j,k}=F_{i,k}$ and $F_{i,i}=\mathrm{I}_{F_i}$. If $\mathcal{G}$
is a double category, then its geometric realization, $|\mathcal{G}|$, is defined by
first taking the double nerve $\dn\mathcal{G}$, which is a bisimplicial set, and then
realizing to obtain a space
$$
|\mathcal{G}|:=|\dn\mathcal{G}|.
$$

To have a manageable description handle description for the bisimplices in
$\dn\mathcal{G}$, we can use the following construction: If
$\mathcal{A}$ and $\mathcal{B}$ are categories, let
$\mathcal{A}\otimes\mathcal{B}$ be the double category whose objects
are pairs $(a,b)$, where $a$ is an object of $\mathcal{A}$ and $b$
is an object of $\mathcal{B}$; horizontal morphisms are pairs
$(f,b):(a,b)\to(c,b)$, with $f:a\to c$ a morphism in $\mathcal{A}$;
vertical morphisms are pairs $(a,u):(a,b)\to (a,d)$ with $u:b\to d$
in $\mathcal{B}$; and a square in $\mathcal{A}\otimes\mathcal{B}$ is
given by each morphism $(f,u):(a,b)\to(c,d)$ in the product category
$\mathcal{A}\times\mathcal{B}$, by stating its boundary as in
$$
\xymatrix{(c,d)\ar @{}[rd]|{\textstyle (f,u)} & (a,d)\ar[l]_-{\textstyle (f,d)} \\
                                        (c,b)\ar[u]^-{\textstyle (c,u)} & (a,b)\ar[l]^-{\textstyle (f,b)}\ar[u]_-{\textstyle (a,u)}}
$$
Compositions in $\mathcal{A}\otimes\mathcal{B}$ are defined in the evident way.

Then, the double nerve $\dn\mathcal{G}$ of a double category $\mathcal{G}$ is the
bisimplicial set
$$
\dn\mathcal{G}:\Delta^{\!o}\!\times\!\Delta^{\!o}\to \mathbf{Set},\hspace{0.4cm}
([p],[q])\mapsto \mathrm{DFunc}([p]\otimes[q],\mathcal{G}),
$$
whose $(p,q)$-bisimplices are the double functors
$F:[p]\otimes[q]\to\mathcal{G}$ or configurations of squares in
$\mathcal{G}$ of the form
$$
\xymatrix{\ar@{-}@/_1pc/@<1.5pc>[d]&F_i^r\ar @{}[rd]|{\textstyle F_{i,j}^{r,s}}  &  F_j^r \ar[l]_-{\textstyle F_{i,j}^r}& \ar@{-}@/^0.9pc/@<-1.5pc>[d]^(0.7){\hspace{-5pt}\scriptsize \begin{array}{l}0\leq i\leq j\leq p\\[-1pt] 0\leq r\leq s\leq q\end{array},}\\ &
                                        F_i^s\ar[u]^-{\textstyle F_i^{r,s}}  & F_j^s \ar[l]^-{\textstyle  F_{i,j}^s} \ar[u]_-{\textstyle  F_j^{r,s}}&}
$$
 such that $F_{i,j}^{r,s}\circ_{\mathrm{h}} F^{r,s}_{j,k}=F^{r,s}_{i,k}$,
 $F^{r,s}_{i,j}\circ_{\mathrm{v}}F^{s,t}_{i,j}=F_{i,j}^{r,t}$, $F^{r,s}_{i,i}=\mathrm{I}^{\mathrm{h}}F_i^{r,s}$, and $F^{r,r}_{i,j}=\mathrm{I}^{\mathrm{v}}F_{i,j}^r$.

But note that the double category $[p]\otimes[q]$ is free on the bigraph
$$
\xymatrix@C=12pt@R=18pt{\ar@{-}@/_0.7pc/@<1.5pc>[d]&\scriptstyle{ (j-1,r-1)} & \scriptstyle{(j,r-1)}\ar[l] &\ar@{-}@/^0.7pc/@<-1.5pc>[d]^(0.7){\hspace{-5pt}\scriptsize \begin{array}{l}0\leq i\leq j\leq p\\[-1pt] 0\leq r\leq s\leq q\end{array},}\\ &\scriptstyle{(j-1,r)}\ar[u] & \scriptstyle{(j,r)}\ar[l]\ar[u]&}
$$
and therefore,  giving a double functor $F:[p]\otimes[q]\to \mathcal{G}$ as above is
equivalent to specifying the $p\times q$ configuration of squares in $\mathcal{G}$
$$
\xymatrix{\ar@{-}@/_1pc/@<1.5pc>[d]&F_{j\!-\!1}^{r\!-\!1}\ar @{}[rd]|{\textstyle F_{j\!-\!1,j}^{r\!-\!1,r}}  &  F_j^{r\!-\!1} \ar[l]_-{\textstyle F_{j\!-\!1,j}^{r\!-\!1}}& \ar@{-}@/^0.9pc/@<-1.5pc>[d]^(0.7){\hspace{-5pt}\scriptsize \begin{array}{l}1\leq j\leq p\\[-1pt] 1\leq r\leq q\end{array}.}\\ &
                                        F^r_{j\!-\!1}\ar[u]^-{\textstyle F_{j\!-\!1}^{r\!-\!1,r}}  & F_j^r \ar[l]^-{\textstyle  F_{
                                        j\!-\!1,j}^r} \ar[u]_-{\textstyle  F_j^{r\!-\!1,r}}&}
$$

Thus, each vertical simplicial set $\dn\mathcal{G}_{p,*}$ is the nerve of the
``vertical'' category having as objects strings of $p$-composable horizontal morphisms
$a_0\leftarrow a_1\leftarrow \cdots\leftarrow a_p$, whose arrows consist of $p$
horizontally composable squares as in
$$
\xymatrix@C=14pt@R=14pt{b_0 & b_1\ar[l] & \cdot\ar[l] & \cdot \ar @{.}[l] & b_p\ar[l] \\
                                        a_0\ar[u] & a_1\ar[u]\ar[l] & \cdot\ar[u]\ar[l] & \cdot\ar[u] \ar @{.}[l] & a_p\ar[u]\ar[l]}
$$
And, similarly, each horizontal simplicial set $\dn\mathcal{G}_{*,q}$ is the nerve of the
``horizontal'' category whose objects are the length $q$ sequences of composable vertical
morphisms of $\mathcal{G}$, with length $q$ sequences of vertically composable squares as
morphisms between them.

For instance, if $\mathcal{A}$ and $\mathcal{B}$ are categories, then
$\dn(\mathcal{A}\otimes\mathcal{B})=\n\mathcal{A}\otimes\n\mathcal{B}$. In particular,
$$
\dn([p]\otimes[q])=\Delta[p]\otimes\Delta[q]=\Delta[p,q],
$$
is the standard $(p,q)$-bisimplex.

It is a well-known fact that the nerve $\n\mathcal{C}$ of a category $\mathcal{C}$
satisfies the Kan extension condition if and only if $\mathcal{C}$ is a groupoid, and, in
such a case, every $(k,n)$-horn ${\Lambda^k[n]\to\n\mathcal{C}}$, for $n\geq 2$, has a
unique extension to an $n$-simplex of $\n\mathcal{C}$
$$
\xymatrix@C=14pt@R=16pt{\Lambda^k[n]\ar[r]\ar @{^{(}->}[d] & \n\mathcal{C} \\
\Delta[n]\ar @{.>}[ru]_{\exists !}}
$$
(see \cite[Propositions 2.2.3 and 2.2.4]{illusie}, for example). For double categories
$\mathcal{G}$, we have the following:

\begin{theorem}\label{thfc} Let $\mathcal{G}$ be a double category. The following statements are equivalent:
\begin{enumerate}
\item[(i)] $\mathcal{G}$ is a double groupoid satisfying the filling condition.
\item[(ii)] The bisimplicial set $\dn\mathcal{G}$ satisfies the extension condition.
\item[(iii)] The simplicial set $\dia\dn\mathcal{G}$ is a Kan complex.
\end{enumerate}
\end{theorem}
\begin{proof} {(i)$\Rightarrow$ (ii)} Since $\mathcal{G}$ is a double groupoid, all simplicial sets $\dn\mathcal{G}_{p,\ast}$ and $\dn\mathcal{G}_{\ast,q}$ are nerves of groupoids. Therefore, every extension problem of the form
$$
\begin{matrix}
\xymatrix@C=14pt@R=16pt{\Delta[p]\otimes\Lambda^{\!l}[q]\ar[r]\ar
@{^{(}->}[d] & \dn\mathcal{G}\\ \Delta[p,q]\ar @{.>}[ru]_{\exists
!}}
\end{matrix} \mbox{ or } \begin{array}{c}
\xymatrix@C=14pt@R=16pt{\Lambda^{\!k}[p]\otimes\Delta[q]\ar[r]\ar
@{^{(}->}[d] & \dn\mathcal{G}\\ \Delta[p,q]\ar @{.>}[ru]_{\exists
!}}
\end{array}
$$
has a solution and it is unique. Suppose then an extension problem of the form
\begin{equation}\label{eq}
\xymatrix@C=14pt@R=14pt{\Lambda^{\!k,l}[p,q]\ar[r]\ar @{^{(}->}[d] &\dn\mathcal{G} \\
\Delta[p,q]\ar @{.>}[ru]}
\end{equation}
If $p\geq 2$, then the restricted map
$\Lambda^{\!k}[p]\otimes\Delta[q]\hookrightarrow
\Lambda^{\!k,l}[p,q]\to \dn\mathcal{G}$ has a unique extension to a
bisimplex $\Delta[p,q]\to\dn\mathcal{G}$, which is a solution to
(\ref{eq}) (which in fact has a unique solution if $p\geq 2$ or
$q\geq 2$). Hence, we reduce the proof to the case in which $p=1=q$,
with the four possibilities $k=0,1$ and $l=0,1$. But any such
extension problem has a solution  thanks to Lemma \ref{fc}. For
example, let us discuss the case $k=0=l$: A bisimplicial map
$\xymatrix@C=40pt{\Lambda^{\!0,0}[1,1]\ar[r]^-{(-,w;-,g)} &
\dn\mathcal{G}}$ consists of two bisimplicial maps $w:\Delta[0,1]\to
\dn\mathcal{G}$ and $g:\Delta[1,0]\to \dn\mathcal{G}$, such that
$wd^1_v=gd^1_h$. That is, a vertical morphism $w$ of $\mathcal{G}$
and a horizontal morphism $g$ of $\mathcal{G}$, such that both have
the same target. By Lemma \ref{fc}, there is a square $\alpha$ in
$\mathcal{G}$ of the form
$$
\xymatrix@C=10pt@R=10pt{\cdot \ar @{}[dr]|{\textstyle \alpha }&
\cdot \ar[l]_-g \\ \cdot \ar[u]^-w & \cdot \ar @{.>}[l] \ar
@{.>}[u]}
$$
which defines a bisimplicial map $F:\Delta[1,1]\to\dn\mathcal{G}$ such that
$F_{0,1}^{0,1}=\alpha$. Then $Fd_h^1=w$, $Fd^1_v=g$, and the diagram below commutes, as
required.
$$
\xymatrix@C=35pt@R=15pt{\Lambda^{\!0,0}[1,1]\ar @{^{(}->}[d]\ar[r]^-{(-,w;-,g)} & \dn\mathcal{G} \\
                                \Delta[1,1]\ar[ru]_-F}$$

{(ii) $\Rightarrow$ (i)} The simplicial sets
$\dn\mathcal{G}_{0,\ast},\,\dn\mathcal{G}_{*,0},\,\dn\mathcal{G}_{1,\ast}$,
and $\dn\mathcal{G}_{\ast,1}$ are respectively the nerves of the
four component categories of the double category $\mathcal{G}$.
Since all these simplicial sets satisfy the Kan extension condition,
it follows that the four category structures involved are groupoids;
that is, $\mathcal{G}$ is a double groupoid. Furthermore, for any
given filling problem in $\mathcal{G}$,
$$\xymatrix@C=-1pt@R=-1pt{\cdot && \cdot\ar[ll]_g\\ &\scriptstyle{\exists ?}&\\
\cdot \ar@{.>}[uu] & &  \cdot \ar@{.>}[ll]\ar[uu]_{u}}$$ we can
solve the extension problem
$$
\xymatrix@C=35pt@R=15pt{\Lambda^{\!1,0}[1,1]\ar @{^{(}->}[d]\ar[r]^-{(u,-;-,g)} & \dn\mathcal{G} \\
                                \Delta[1,1]\ar@{.>}[ru]_-F}$$
and the square $F^{0,1}_{0,1}$ has $u$ as horizontal source and $g$ as vertical target.
Thus $\mathcal{G}$ satisfies the filling condition.

{(i) $\Rightarrow$ (iii)} The higher dimensional part of the proof is in the following
lemma, that we establish for  future reference.
\begin{lemma}\label{lus} If $\mathcal{G}$ is any double groupoid and $n$ is any integer such that
$n\geq 3$, then every extension problem
$$
\xymatrix@C=20pt@R=15pt{\Lambda^k[n]\ar[r]\ar @{^{(}->}[d] & \dia\dn\mathcal{G} \\
\Delta[n]\ar @{.>}[ru]}
$$
has a solution and it is unique.
\end{lemma}
\begin{proof} Let $F=(F_{i,j}^{r,s}):[n]\otimes[n]\to \mathcal{G}$ denote the double functor we are looking for solving the given extension problem. Recall that to give  such an $F$ is equivalent to specifying the $n\times n$ configuration of squares
 $$
\xymatrix{\ar@{-}@/_1pc/@<1.5pc>[d]&F_{j\!-\!1}^{r\!-\!1}\ar @{}[rd]|{\textstyle F_{j\!-\!1,j}^{r\!-\!1,r}}  &  F_j^{r\!-\!1} \ar[l]_-{\textstyle F_{j\!-\!1,j}^{r\!-\!1}}& \ar@{-}@/^0.9pc/@<-1.5pc>[d]^(0.7){\hspace{-5pt}\scriptsize \begin{array}{l}1\leq j\leq n\\[-1pt] 1\leq r\leq n\end{array}.}\\ &
                                        F^r_{j\!-\!1}\ar[u]^-{\textstyle F_{j\!-\!1}^{r\!-\!1,r}}  & F_j^r
                                        \ar[l]^-{\textstyle  F_{
                                        j\!-\!1,j}^r} \ar[u]_-{\textstyle  F_j^{r\!-\!1,r}}&}
$$
We claim that $F$ exists and, moreover, that it is completely showed  from any three of
its (known) faces $[n\!-\!1]\otimes[n\!-\!1]\stackrel{d^m\otimes d^m}{\longrightarrow}
[n]\otimes[n]\stackrel{F}{\longrightarrow}\mathcal{G},\,m\neq k$; therefore,  by the input data
$\Lambda^{\!k}[n]\to \dia\dn\mathcal{G}$. In effect, since each $m^{\text{ th}}$-face
consists of all squares $F_{i,j}^{r,s}$ such that $m\notin\{i,j,r,s\}$, once we have
selected any three integers $m,\,p,\,q$ with $m<p<q$ and $k\notin\{m,p,q\}$, we know
explicitly all squares $F^{r,s}_{i,j}$ except those in which $m,\,p$ and $q$ appear in
the labels, that is:  $F_{q,j}^{m,p}$, $F^{m,q}_{p,j}$, $F^{j,p}_{m,q}$, and so on. In the case
where $k\geq 3$, if we take $\{m,p,q\}=\{0,1,2\}$ then we have given all squares $F_{i,j}^{r,s}$, except those with $\{0,1,2\}\subseteq\{r,s,i,j\}$. In particular, we have all
$F^{r,r+1}_{i,i+1}$, except four of them, namely,
$F^{0,1}_{2,3},\,F^{0,1}_{1,2},\,F^{1,2}_{0,1}$, and $F^{2,3}_{0,1}$, which, however, are
uniquely determined by the equations
$$
F^{0,1}_{2,3}\circ_{\mathrm{v}} F^{1,2}_{2,3}=F^{0,2}_{2,3},\,\,
F^{2,3}_{0,1}\circ_{\mathrm{h}}F^{2,3}_{1,2}=F^{2,3}_{0,2},\,\,
F^{0,1}_{1,2}\circ_{\mathrm{h}}F^{0,1}_{2,3}=F^{0,1}_{1,3},\,\,
F^{1,2}_{0,1}\circ_{\mathrm{v}}F^{2,3}_{0,1}=F^{1,3}_{0,1},
$$
that is,
$F^{0,1}_{2,3}=F^{0,2}_{2,3}\circ_{\mathrm{v}}(F^{1,2}_{2,3})^{\text{-}1_{\!\mathrm{v}}}$,
and so on. The other possibilities for $k$ are discussed in a similar way:
If $k=2$, then we select $\{m,p,q\}=\{0,1,n\}$ and determine
 $F$ completely by taking into account the two equations $
F^{0,1}_{n\!-\!1,n}\circ_{\mathrm{v}}F^{1,2}_{n\!-\!1,n}=F^{0,2}_{n\!-\!1,n}$,
$
F^{n\!-\!1,n}_{0,1}\circ_{\mathrm{h}}F^{n\!-\!1,n}_{1,2}=F^{n\!-\!1,n}_{0,2}.
$

If $k=1$, then we take $\{m,p,q\}=\{0,2,3\}$ and find the unknown
squares $F^{2,3}_{0,1}$ and $F^{0,1}_{2,3}$ by the equations
$F^{1,2}_{0,1}\circ_{\mathrm{v}}F^{2,3}_{0,1}=F^{1,3}_{0,1}$ and
$F^{0,1}_{1,2}\circ_{\mathrm{h}}F^{0,1}_{2,3}=F^{0,1}_{1,3}$,
respectively.

Finally, in the case where $k=0$, we take
$\{m,p,q\}=\{n\!-\!2,n\!-\!1,n\}$ and we determine the non-given
four squares of the family $(F^{r,r+1}_{i,i+1})$, that is,
$F^{n\!-\!2,n\!-\!1}_{n\!-\!1,n},\,F^{n\!-\!1,n}_{n\!-\!2,n\!-\!1},\,
F^{n\!-\!3,n\!-\!2}_{n\!-\!1,n}$, and
$F^{n\!-\!1,n}_{n\!-\!3,n\!-\!2}$ by means of the four equations
$F^{n\!-\!3,n\!-\!2}_{n\!-\!3,n\!-\!1}\circ_{\mathrm{h}}F^{n\!-\!3,n\!-\!2}_{n\!-\!1,n}=
F^{n\!-\!3,n\!-\!2}_{n\!-\!3,n}$,
$F^{n\!-\!3,n\!-\!1}_{n\!-\!3,n\!-\!2}\circ_{\mathrm{v}}
F^{n\!-\!1,n}_{n\!-\!3,n\!-\!2}=F^{n\!-\!3,n}_{n\!-\!3,n\!-\!2}$,
$F^{n\!-\!3,n\!-\!2}_{n\!-\!1,n}\circ_{\mathrm{v}}
F^{n\!-\!2,n\!-\!1}_{n\!-\!1,n}=F^{n\!-\!3,n\!-\!1}_{n\!-\!1,n}$,
and
$F^{n\!-\!1,n}_{n\!-\!3,n\!-\!2}\circ_{\mathrm{h}}F^{n\!-\!1,n}_{n\!-\!2
,n\!-\!1}=F^{n\!-\!1,n}_{n\!-\!3,n\!-\!1}$.
This completes the proof of the lemma.
\end{proof}
We now return to the proof of (i) $\Rightarrow$ (iii) in Theorem \ref{thfc}. After Lemma
\ref{lus} above, it remains to prove that every extension problem
$$
\xymatrix@C=14pt@R=16pt{\Lambda^{\!k}[2]\ar[r]\ar
@{^{(}->}[d] & \dia\dn\mathcal{C} \\ \Delta[2]\ar
@{.>}[ru]_{\exists ?}}
$$
for $k=0,1,2$, has a solution. In the case where
$k=0$, the data for a simplicial map
$(-,\tau,\sigma):\Lambda^0[2]\to\dia\dn\mathcal{G}$
consists of a couple of squares in $\mathcal{G}$
of the form $$
\xymatrix@R=10pt@C=10pt{\scriptstyle{a}\ar @{}[rd]|{\textstyle \sigma} & \cdot \ar[l] \\
\cdot \ar[u] & \cdot \ar[u]\ar[l]}\hspace{1cm}
\xymatrix@R=10pt@C=10pt{\scriptstyle{a}\ar @{}[rd]|{\textstyle \tau}
& \cdot \ar[l] \\ \cdot \ar[u] & \cdot \ar[u]\ar[l]} $$ and an
extension solution
$\xymatrix@C=18pt{\Delta[2]\ar@{.>}[r]&\dia\dn\mathcal{G}}$ amounts to
a diagram of squares as in
$$
\xymatrix@C=12pt@R=12pt{\scriptstyle{a}\ar @{}[rd]|{\textstyle
\sigma} & \cdot\ar[l]\ar @{}[rd]|{\textstyle x}
& \cdot \ar @{.>}[l] \\
\cdot\ar[u]\ar @{}[rd]|{\textstyle y} &
\cdot\ar[l]\ar[u]\ar @{}[rd]|{\textstyle z}
& \cdot  \ar @{.>}[l]\ar @{.>}[u]\\
\cdot\ar @{.>}[u] & \cdot \ar @{.>}[l]\ar
@{.>}[u] & \cdot \ar @{.>}[l]\ar @{.>}[u]}
$$
such that $(\sigma
\circ_{\mathrm{h}}x)\circ_{\mathrm{v}}(y\circ_{\mathrm{h}}z)=\tau$.
To see that such squares $x$, $y$, and $z$ exist, we form the
configuration (actually, a 3-simplex of $\dia\dn\mathcal{G}$)
$$
\xymatrix@C=16pt@R=16pt{\scriptstyle{a} \ar @{}[rd]|{\textstyle
\sigma} & \cdot \ar[l]\ar
@{}[rd]|{\textstyle{\sigma}^{\text{-}1_{\mathrm{h}}}} & \cdot
\ar[l]
\ar @{}[rd]|{\textstyle{\alpha}^{\text{-}1_{\!\mathrm{v}}}} & \cdot \ar[l] \\
\cdot \ar[u]\ar
@{}[rd]|{\textstyle{\sigma}^{\text{-}1_{\!\mathrm{v}}}} & \cdot
\ar[u]\ar[l]\ar @{}[rd]|{\textstyle{\sigma}^{\text{-}1}} & \cdot
\ar[l]\ar[u]\ar @{}[rd]|{\textstyle \alpha} & \cdot \ar @{.>}[l]\ar @{.>}[u] \\
\cdot \ar[u]\ar
@{}[rd]|{\textstyle{\beta}^{\text{-}1_{\mathrm{h}}}}
 & \cdot \ar[u]\ar[l]\ar
@{}[rd]|{\textstyle \beta} &
\scriptstyle{a}\ar[l] \ar[u]
\ar @{}[rd]|{\textstyle \tau} & \cdot \ar[l]\ar @{.>}[u] \\
\cdot \ar@{.>}[u] & \cdot \ar @{.>}[l]\ar
@{.>}[u] & \cdot \ar @{.>}[l]\ar[u] & \cdot
\ar[u]\ar[l]}
$$
where $\alpha$ and $\beta$ are any found thanks  $\mathcal{G}$
satisfies the filling condition. Then, we take
$x=\sigma^{\text{-}1_{\mathrm{h}}}\circ_{\mathrm{h}}\alpha^{\text{-}
1_{\!\mathrm{v}}}$,
$y=\sigma^{\text{-}1_{\!\mathrm{v}}}\circ_{\mathrm{v}}
\beta^{\text{-}1_{\mathrm{h}}}$,
and
$z=(\sigma^{\text{-1}}\circ_{\mathrm{h}}\alpha)\circ_{\mathrm{v}}(\beta\circ_{\mathrm{h}}\tau)$.

The case in which $k=2$ is dual of the case $k=0$ above, and the
case when $k=1$ is easier: A simplicial map
$(\sigma,-,\tau):\Lambda^{\!1}[2]\to\dia\dn\mathcal{G}$ amounts to a
couple of squares in $\mathcal{G}$ of the form
$$
\xymatrix@R=10pt@C=10pt{\scriptstyle{a}\ar @{}[rd]|{\textstyle \sigma} & \cdot \ar[l] \\
\cdot \ar[u] & \cdot \ar[u]\ar[l]}\hspace{1cm}
\xymatrix@R=10pt@C=10pt{\cdot\ar @{}[rd]|{\textstyle \tau} & \cdot
\ar[l]
\\ \cdot \ar[u] &\scriptstyle{a} \ar[u]\ar[l]} $$
and an extension solution
$\xymatrix@C=18pt{\Delta[2]\ar@{.>}[r]&\dia\dn\mathcal{G}}$ is given
by any configuration of squares in $\mathcal{G}$ of the form
$$
\xymatrix@C=12pt@R=12pt{ \cdot \ar@{}[rd]|{\textstyle \tau} & \cdot\ar[l]\ar @{}[rd]|{\textstyle x}
 & \cdot\ar @{.>}[l] \\
\cdot\ar[u]\ar @{}[rd]|{\textstyle y} &
\scriptstyle{a}\ar[u]\ar[l]\ar @{}[rd]|{\textstyle \sigma} &
\cdot \ar[l]\ar @{.>}[u] \\
\cdot\ar @{.>}[u] & \cdot \ar @{.>}[l]\ar[u] & \cdot \ar[u]\ar[l]}
$$
Since $\mathcal{G}$ satisfies the filling condition (recall Lemma
\ref{fc}), it is clear that filling squares $x$ and $y$ as above
exist, and therefore the required extension map exists.

(iii) $\Rightarrow$ (i)  By \cite[Theorem 8]{c-r2}, all simplicial
sets $\dn\mathcal{G}_{p,\ast}$ and $\dn\mathcal{G}_{\ast,q}$ satisfy
the Kan extension condition. In particular, the nerves of the four
component categories of the double category $\mathcal{G}$, that is,
the simplicial sets
$\dn\mathcal{G}_{0,\ast}$, $\dn\mathcal{G}_{*,0}$,$\dn\mathcal{G}_{1,\ast}$,
and $\dn\mathcal{G}_{\ast,1}$ are  all  Kan complexes. By  \cite[Propositions 2.2.3 and
2.2.4]{illusie}, it follows that the four category structures
involved are groupoids, and so $\mathcal{G}$ is a double groupoid.

To see that $\mathcal{G}$ satisfies the filling condition, suppose that a
filling problem $$\xymatrix@C=-1pt@R=-1pt{\cdot && \cdot\ar[ll]_g\\
&\scriptstyle{\exists ?}&\\ \cdot \ar@{.>}[uu] & &  \cdot
\ar@{.>}[ll]\ar[uu]_{u}}$$ is given. Since the simplicial map
$\xymatrix@C=40pt{
\Lambda^{\!1}[2]\ar[r]^-{(\mathrm{I}\mathrm{^h}u,-,\mathrm{I}\mathrm{^v}\!g)}
& \dia\dn\mathcal{C}}$ has an extension to a 2-simplex
$\xymatrix@C=18pt{\Delta[2]\ar@{.>}[r]&\dia\dn\mathcal{G}}$, we
conclude the existence of a diagram of squares in $\mathcal{G}$ of
the form
$$
\xymatrix@C=16pt@R=16pt{ \cdot
\ar@{}[rd]|{{\textstyle\mathrm{I}^{\mathrm{v}}\!g}} & \cdot
\ar[l]_-g & \cdot\ar @{.>}[l] \\
\cdot \ar@<-0.1ex>@{-}[u]\ar@<0.1ex>@{-}[u]&
 \cdot\ar[l]\ar@{}@<-2pt>[l]^-g\ar@<-0.1ex>@{-}[u]\ar@<0.1ex>@{-}[u]\ar@{}[rd]|{\textstyle \mathrm{I}^{\mathrm{h}}u} &
  \cdot\ar @{.>}[u] \ar@<-0.1ex>@{-}[l]\ar@<0.1ex>@{-}[l]\\\ar@{}[ru]|(0.45){\textstyle \alpha}
 \cdot \ar @{.>}[u] & \cdot \ar[u]\ar@{}@<-3pt>[u]^-u\ar @{.>}[l] & \cdot \ar[u]_-u \ar@<-0.1ex>@{-}[l]\ar@<0.1ex>@{-}[l]}
$$
and then, particularly, the existence of a square $\alpha$ as is
required. \end{proof}

We now state our main result in this section.

\begin{theorem} \label{th0}Let $\mathcal{G}$ be a double groupoid satisfying the filling condition.
Then, for each object $a$ of $\mathcal{G}$,  there are natural
isomorphisms
\begin{equation}\label{i-i}
\pi_i(\mathcal{G},a)\cong \pi_i(|\mathcal{G}|,|a|),\,\, i\geq 0.
\end{equation}
\end{theorem}
\begin{proof} By taking into account Fact \ref{f1} (1), we shall identify the homotopy groups
of $|\mathcal{G}|$ with those of the Kan complex (by Theorem
\ref{thfc}) $\dia\dn\mathcal{G}$, which are defined, as we
noted in the preliminary Section 2, using only its simplicial
structure.

To compare the $\pi_0$ sets, observe that the 0-simplices
$a\in\dia\dn\mathcal{G}_0=\dn\mathcal{G}_{0,0}$ are precisely the
objects of $\mathcal{G}$. Furthermore, two 0-simplices $a,\,b$ are
in the same connected component of $\dia\dn\mathcal{G}$ if and only
if there is a square (i.e., a 1-simplex) of the form
$$
\xymatrix@C=12pt@R=12pt{b \ar@{}[rd]|{\textstyle \exists ?}& \cdot \ar@{.>}[l] \\
\cdot\ar@{.>}[u] & a,\ar@{.>}[l]\ar@{.>}[u]}
$$
that is, since $\mathcal{G}$ satisfies the filling condition, if and
only if $a$ and $b$ are connected in $\mathcal{G}$ (see Subsetion \ref{s21}).
Thus,  $\pi_0|\mathcal{G}|=\pi_0\mathcal{G}$.

We now compare the $\pi_1$ groups. An element
$[\alpha]\in\pi_1(|\mathcal{G}|,|a|)$ is the equivalence class of a
square $\alpha$ in $\mathcal{G}$ of the form
$$
\xymatrix@R=12pt@C=12pt{a\ar@{}[rd]|{\textstyle \alpha} & \cdot\ar[l]_-g \\
\cdot\ar[u] & a\ar[l]\ar[u]_(0.4)u}
$$
and $[\alpha]=[\alpha']$ if and only if there is a configuration of
squares in $\mathcal{G}$ of the form
$$
\xymatrix@C=12pt@R=14pt{ a\ar@{}[rd]|{\textstyle \alpha} &
\cdot\ar@{}[rd]|{\textstyle x}\ar[l]_-g &
 \cdot\ar[l]_-{g'} \\
\cdot\ar@{}[rd]|{\textstyle y}\ar[u] & a\ar@{}[rd]|{\textstyle
\mathrm{I}a}\ar[u]\ar@{}@<2pt>[u]_(0.4)u\ar[l] &
a\ar[u]_(0.4){u'}\ar@<-0.1ex>@{-}[l]\ar@<0.1ex>@{-}[l] \\
                                        \cdot\ar[u] & a\ar[l]\ar@<-0.1ex>@{-}[u]\ar@<0.1ex>@{-}[u]
                                        & a\ar@<-0.1ex>@{-}[l]
\ar@<0.1ex>@{-}[l]\ar@<-0.1ex>@{-}[l]\ar@<0.1ex>@{-}[u]\ar@<-0.1ex>@{-}[u]}
$$
such that $(\alpha\circ_{\mathrm{h}}x)\circ_\mathrm{v}y=\alpha'$. By
recalling now the definition of the homotopy group $\pi_1
(\mathcal{G} , a)$, we observe that, if $[\alpha]=[\alpha']$ in
$\pi_1(|\mathcal{G}|, |a|)$, then, by the existence of the squares
$\alpha$ and $\alpha\circ_\mathrm{h}x$, we have
$[g,u]=[g\circ_{\mathrm{h}} g',u']$ in $\pi_1(\mathcal{G},a)$; that
is,$[\mathrm{t}^\mathrm{v} \alpha, \mathrm{s}^\mathrm{h}
\alpha]=[\mathrm{t}^\mathrm{v} \alpha', \mathrm{s}^\mathrm{h}
\alpha']$. It follows that there is a well-defined map
$$
\begin{array}{rcl}
 \Phi:\pi_1(|\mathcal{G}|,|a|) & \longrightarrow & \pi_1(\mathcal{G},a). \\
 \left[ \alpha \right] & \longmapsto & [g,u]\!=\![\mathrm{t}^\mathrm{v} \alpha, \mathrm{s}^\mathrm{h} \alpha]
\end{array}
$$
This map is actually a group homomorphism. To see that, let
$$
\xymatrix@C=12pt@R=12pt{a\ar@{}[rd]|{\textstyle\alpha_1} & \cdot\ar[l]_-{g_1} \\
\cdot\ar[u] & a\ar[l]\ar[u]_-{u_1}} \hspace{6pt}
\xymatrix@C=12pt@R=12pt{a\ar@{}[rd]|{\textstyle \alpha_2} & \cdot\ar[l]_-{g_2} \\
\cdot\ar[u] & a\ar[l]\ar[u]_-{u_2}}
$$
be squares representing elements $[\alpha_1],\,[\alpha_2]\in
\pi_1(|\mathcal{G}|,|a|)$. Then, its product in the homotopy group
$\pi_1(|\mathcal{G}|,|a|)$ is
$[\alpha_1]\circ[\alpha_2]=[(\alpha_1\circ_\mathrm{h}
\beta)\circ_\mathrm{v}( \gamma\circ_\mathrm{h}\alpha_2)]$, where
$\beta$ and $\gamma$ are any squares in $\mathcal{G}$ defining a
configuration of the form (i.e., a 2-simplex of
$\dia\dn\mathcal{G}$)
$$
\xymatrix@C=12pt@R=12pt{ a \ar@{}[rd]|{\textstyle\alpha_1} & \cdot
\ar[l]_-{g_1} \ar@{}[rd]|{\textstyle \beta} &
\cdot \ar[l]_-g \\
\cdot \ar@{}[rd]|{\textstyle \gamma}\ar[u] & a\ar@{}[rd]|{\textstyle\alpha_2}\ar[u]\ar[l] & \cdot\ar[l]\ar[u]_-u\\
\cdot\ar[u] & \cdot\ar[l]\ar[u] &  a\ar[l]\ar[u]_{u_2}}
$$
Hence,
$$
\Phi([\alpha_1]\circ[\alpha_2])=[g_1\circ_h g,u\circ_v
u_2]=[g_1,u_1]\circ[g_2,u_2]= \Phi([\alpha_1])\circ\Phi([\alpha_2]),
$$
and therefore $\Phi$ is a homomorphism.

From the filling condition on $\mathcal{G}$, it follows that $\Phi$
is a surjective map. To prove that it is also injective,  suppose
$\Phi[\alpha_1]=\Phi[\alpha_2]$, where $[\alpha_1],\,[\alpha_2]\in
\pi_1(|\mathcal{G}|,a)$ are as above. This means that there are
squares in $\mathcal{G}$, say $x_1$ and $x_2$, of the form
$$
\xymatrix@C=12pt@R=12pt{a\ar@{}[rd]|{\textstyle x_1} & \cdot\ar[l]_-{g_1} \\
\cdot\ar[u]^-w & a\ar[l]^-f\ar[u]_-{u_1}} \hspace{8pt}
\xymatrix@C=12pt@R=12pt{a\ar@{}[rd]|{\textstyle x_2} & \cdot\ar[l]_-{g_2} \\
\cdot\ar[u]^-w & a\ar[l]^-f\ar[u]_-{u_2}}
$$
with which we can form the following three 2-simplices of
$\dia\dn\mathcal{G}$
$$
\xymatrix@C=35pt@R=16pt{\cdot\ar@{}[rd]|{\textstyle x_1} &
\cdot\ar@{}[rd]|{\textstyle\mathrm{I}^\mathrm{h}u_1}\ar[l] & \cdot
\ar@<0.1ex>@{-}[l]\ar@<-0.1ex>@{-}[l] \\
\cdot\ar@{}[rd]|{\textstyle
x^{\text{-}1_{\!\mathrm{v}}}_1\!\!\circ_\mathrm{v}\!\alpha_1}\ar[u]
& \cdot\ar@{}[rd]|{\textstyle\mathrm{I}a}\ar[l]\ar[u]&
\cdot\ar[u]\ar@<0.1ex>@{-}[l]\ar@<-0.2ex>@{-}[l]\\
\cdot\ar[u] & \cdot\ar[l]\ar@<0.1ex>@{-}[u]\ar@<-0.1ex>@{-}[u] &
\cdot
\ar@<0.1ex>@{-}[l]\ar@<-0.1ex>@{-}[l]\ar@<0.1ex>@{-}[u]\ar@<-0.1ex>@{-}[u]\ar@{}[u]_(0.0){\textstyle
, }}
\hspace{10pt}
\xymatrix@C=35pt@R=16pt{\cdot\ar@{}[rd]|{\textstyle x_2} &
\cdot\ar@{}[rd]|{\textstyle\mathrm{I}^\mathrm{h}u_2}\ar[l] & \cdot
\ar@<0.1ex>@{-}[l]\ar@<-0.1ex>@{-}[l] \\
\cdot\ar@{}[rd]|{\textstyle
x^{\text{-}1_{\!\mathrm{v}}}_2\!\!\circ_\mathrm{v}\!\alpha_2}\ar[u]
& \cdot\ar@{}[rd]|{\textstyle\mathrm{I}a}\ar[l]\ar[u] &
\cdot\ar[u]\ar@<0.1ex>@{-}[l]\ar@<-0.2ex>@{-}[l]\\
\cdot\ar[u] & \cdot\ar[l]\ar@<0.1ex>@{-}[u]\ar@<-0.1ex>@{-}[u] &
\cdot
\ar@<0.1ex>@{-}[l]\ar@<-0.1ex>@{-}[l]\ar@<0.1ex>@{-}[u]\ar@<-0.1ex>@{-}[u]\ar@{}[u]_(0.0){\textstyle
, }}
 \hspace{10pt}
\xymatrix@C=35pt@R=16pt{\cdot\ar@{}[rd]|{\textstyle x_1} &
\cdot\ar@{}[rd]|{\textstyle
x^{\text{-}1_{\mathrm{h}}}_1\!\!\circ_\mathrm{h}\!x_2 }\ar[l] &
\cdot
\ar[l] \\
\cdot\ar@{}[rd]|{\textstyle
\textstyle\mathrm{I}^\mathrm{v}\!f}\ar[u] &
\cdot\ar@{}[rd]|{\textstyle\mathrm{I}a}\ar[l]\ar[u] &
\cdot\ar[u]\ar@<0.1ex>@{-}[l]\ar@<-0.2ex>@{-}[l]\\
\cdot \ar@<0.1ex>@{-}[u]\ar@<-0.1ex>@{-}[u] &
\cdot\ar[l]\ar@<0.1ex>@{-}[u]\ar@<-0.1ex>@{-}[u] & \cdot
\ar@<0.1ex>@{-}[l]\ar@<-0.1ex>@{-}[l]\ar@<0.1ex>@{-}[u]\ar@<-0.1ex>@{-}[u]\ar@{}[u]_(0.0){\textstyle
. }}
$$
The first one shows that $[x_1]=[\alpha_1]$  in the group
$\pi_1(|\mathcal{G}|,a)$, the second that $[x_2]=[\alpha_2]$, and
the third that $[x_1]=[x_2]$. Whence $[\alpha_1]=[\alpha_2]$, as
required.

Finally, we show the isomorphisms
$\pi_1(|\mathcal{G}|,|a|)\cong\pi_i(\mathcal{G},a)$, for $i\geq 2$.
For $i\geq 3$, it follows from Lemma \ref{lus} that
$\pi_i(|\mathcal{G}|,a)=0$, and the result becomes obvious. For the
case $i=2$, it is also a consequence of the afore-mentioned Lemma \ref{lus}
that the homotopy relation between 2-simplices in
$\dia\dn\mathcal{G}$ is trivial. Then, the group
$\pi_2(|\mathcal{G}|,|a|)$ consists of all 2-simplices in
$\dia\dn\mathcal{G}$ of the form
$$
\xymatrix@C=12pt@R=12pt{\cdot\ar@{}[rd]|{\textstyle\mathrm{I}a} &
\cdot\ar@<0.1ex>@{-}[l]\ar@<-0.1ex>@{-}[l]
\ar@{}[rd]|{\textstyle\sigma} & \cdot
\ar@<0.1ex>@{-}[l]\ar@<-0.1ex>@{-}[l] \\
\cdot\ar@{}[rd]|{\textstyle\sigma^{\text{-}1}}
\ar@<0.1ex>@{-}[u]\ar@<-0.1ex>@{-}[u] & \cdot\ar@{}[rd]|{\textstyle
\mathrm{I}a} \ar@<0.1ex>@{-}[l]\ar@<-0.1ex>@{-}[l]
\ar@<0.1ex>@{-}[u]\ar@<-0.1ex>@{-}[u]  &
 \ar@<0.1ex>@{-}[l]\ar@<-0.1ex>@{-}[l]
 \ar@<0.1ex>@{-}[u]\ar@<-0.1ex>@{-}[u] \\
 \ar@<0.1ex>@{-}[u]\ar@<-0.1ex>@{-}[u] \cdot & \cdot
\ar@<0.1ex>@{-}[l]\ar@<-0.1ex>@{-}[l]
 \ar@<0.1ex>@{-}[u]\ar@<-0.1ex>@{-}[u] & \cdot
 \ar@<0.1ex>@{-}[l]\ar@<-0.1ex>@{-}[l]
 \ar@<0.1ex>@{-}[u]\ar@<-0.1ex>@{-}[u] }
$$
for $\sigma\in\pi_2(\mathcal{G},a)$, whence the isomorphism becomes
clear.
\end{proof}

\begin{corollary}\label{we2}
A double functor $F\!:\!\mathcal{G}\to\mathcal{G}'$ is a weak
equivalence if and only if the induced cellular  map on realizations
$| F|\!:\!|\mathcal{G}|  \to |\mathcal{G}'|$ is a homotopy
equivalence.
\end{corollary}

\section{A left adjoint to the double nerve functor.}\label{lad}
Recall from Theorem \ref{thfc} (ii) that the double nerve
$\dn\mathcal{G}$, of any double groupoid satisfying the filling
condition, satisfies the extension condition. Moreover, since both
simplicial sets $\dn\mathcal{G}_{\ast,0}$ and
$\dn\mathcal{G}_{0,\ast}$ are nerves of groupoids, all homotopy
groups $\pi_2(\dn\mathcal{G}_{\ast,0},a)$ and
$\pi_2(\dn\mathcal{G}_{0,\ast},a)$ vanish. Our goal in this section
is to prove that the double nerve functor, $\mathcal{G} \mapsto
\dn\mathcal{G}$, embeds, as a reflexive subcategory, the category of
double groupoids with filling condition into the category
of those bisimplicial sets $K$ that satisfy the extension condition
and such that $\pi_2(K_{\ast,0},a)=0=\pi_2(K_{0,\ast},a)$ for all
vertices $a\in K_{0,0}$. That is, there is a reflector functor for
such bisimplicial sets
$$
K\mapsto \adj K, $$
which works as a bisimplicial version of Brown's construction in \cite[Theorem 2.1]{brown2}.
Furthermore, as we will prove, the resulting double groupoid $\adj K$  always represents the homotopy 2-type of the
input bisimplicial set $K$, in the sense that there is a natural weak 2-equivalence $|K|\to |\adj K|$.

For any given bisimplicial set $K$, under the assumption that it satisfies the extension condition and both the Kan complexes   $K_{*,0}$ and   $K_{0,*}$ have trivial groups $\pi_2$,  the definition of the {\em homotopy double groupoid} $\adj K$ is as follows:

\vspace{0.2cm}

The objects of $\adj K$ are the vertices $a:\Delta[0,0]\to K$ of
$K$.

\vspace{0.2cm}

The groupoid of horizontal morphisms is the horizontal fundamental groupoid $\mathrm{P} K_{\!\ast,0}$,  and the groupoid
of vertical morphisms is the vertical fundamental groupoid  $\mathrm{P} K_{0,\ast}$ (see the last part of Subsection
\ref{s22}). Thus, a horizontal morphism $[f]_\mathrm{h}:a\to b$ is the horizontal homotopy class of a bisimplex
$f:\Delta[1,0]\to K$ with $fd^0_\mathrm{h}=a$ and $fd^1_\mathrm{h}=b$, whereas a vertical morphism in $\adj K$,
$[u]_\mathrm{v}:a\to b$, is the vertical homotopy class of a bisimplex $u:\Delta[0,1]\to K$ with $ud^0_\mathrm{v}=a$ and
$ud^1_\mathrm{v}=b$.

\vspace{0.2cm}

A square of $\adj K$ is the bihomotopy class $[[x]]$ of a bisimplex $x:\Delta[1,1]\to K$, with boundary
$$
\xymatrix@R=16pt@C=18pt{\cdot\ar @{}[dr]|{} & \cdot\ar[l]_{[xd^1_{\mathrm{v}}]_{\mathrm{h}}} \\
\cdot\ar[u]^{[xd^1_{\mathrm{h}}]_{\mathrm{v}}} &
\cdot\ar[l]^{[xd^0_{\mathrm{v}}]_{\mathrm{h}}}
\ar[u]_{[xd^0_{\mathrm{h}}]_{\mathrm{v}}}}
$$
which is well defined thanks to Lemma \ref{p5}.

\vspace{0.2cm}

The horizontal composition of squares in $\adj K$ is the only one making
the correspondence
$$
\Big(xd^0_{\mathrm{h}}\overset{[x]_{\mathrm{h}}}\to
xd^1_{\mathrm{h}}\Big) \overset{[\ \ ]}\longmapsto
 \Big([xd^0_{\mathrm{h}}]_{\mathrm{v}}\overset{[[x]]}\to
[xd^1_{\mathrm{h}}]_{\mathrm{v}}\Big)
$$
a surjective fibration of groupoids  from the horizontal fundamental groupoid $\mathrm{P} K_{\ast,1}$ to the
horizontal groupoid of squares in $\adj K$. To define this composition, we shall need the following:

\begin{lemma} \label{lad1} Let $x, y:\Delta[1,1]\to K$ be bisimplices such that
$[xd^0_{\mathrm{h}}]_{\mathrm{v}}
=[yd^1_{\mathrm{h}}]_{\mathrm{v}}$. Then, there is a bisimplex
$x'\!:\!\Delta[1,1]\to K$ such that
$[x']_{\mathrm{v}}=[x]_{\mathrm{v}}$ and
$x'd^0_{\mathrm{h}}=yd^1_{\mathrm{h}}$.
\end{lemma}
\begin{proof} Once any vertical homotopy
from $xd^0_{\mathrm{h}}$ to $yd^1_{\mathrm{h}}$ is selected, say $\alpha:\Delta[0,2]\to K$,   let
$\beta:\Delta[1,2]\to K$ be any bisimplex solving the extension
problem
$$
\xymatrix@C=60pt@R=16pt{\Lambda^{1,2}[1,2]\ar[r]^-{(\alpha,-;xd^0_{\mathrm{v}}s^0_{\mathrm{v}},\,x,-)}
\ar @{^{(}->}[d]& K. \\
          \Delta[1,2]\ar @{.>}[ru]_\beta & }
$$
Then, we take $x'=\beta d^2_{\mathrm{v}}:\Delta[1,1]\to K$. Since
$\beta$ becomes a vertical homotopy from $x$ to $x'$, we have
$[x]_{\mathrm{v}}=[x']_{\mathrm{v}}$. Moreover,
$x'd^0_{\mathrm{h}}=\beta d^2_{\mathrm{v}}d^0_{\mathrm{h}}=\beta
d^0_{\mathrm{h}}d^2_{\mathrm{v}}=\alpha
d^2_{\mathrm{v}}=yd^1_{\mathrm{h}}$, as required.
\end{proof}

\noindent {\bf Remark.} Note that, for any such  bisimplex $x'$ as in the lemma, we have $[[x']]=[[x]]$ and
$x'd^i_{\mathrm{v}}= xd^i_{\mathrm{v}}$ for $i=0,1$.

\vspace{0.2cm} Now define the horizontal composition of squares in
$\adj K$ by
\begin{equation}\label{comp}\begin{array}{lll}
[[x]]\circ_{\mathrm{h}}
[[y]]=[[x']_{\mathrm{h}}\circ_{\mathrm{h}}[y]_{\mathrm{h}}]& \text{
if }& [x]_{\mathrm{v}}=[x']_{\mathrm{v}} \text{ and }
x'd^0_{\mathrm{h}}=yd^1_{\mathrm{h}},\end{array}
\end{equation}
where $[x']_{\mathrm{h}}\circ_{\mathrm{h}}[y]_{\mathrm{h}}$ is the composite in the fundamental groupoid $\mathrm{P}
K_{*,1}$, that is,
\begin{equation}
[[x]]\circ_{\mathrm{h}}[[y]] = [[\gamma d^1_{\mathrm{h}}]]
\end{equation}
for $\gamma : \Delta[2,1]\to K$ any bisimplex with $\gamma
d^2_{\mathrm{h}}=x'$ and $\gamma d^0_{\mathrm{h}}=y$.

In view of Lemma \ref{lad1}, our product is given for all squares
$[[x]]$ and $[[y]]$ with
$\mathrm{s}^\mathrm{h}[[x]]=\mathrm{t}^\mathrm{h}[[y]]$. We also have
the lemma below.

\begin{lemma} The horizontal composition of squares in $\adj K$ is well defined.
\end{lemma}
\begin{proof} We first prove that the square in (\ref{comp}) does not depend on
the choice of $x'$. To do so, suppose $x'':\Delta[1,1]\to K$ is
another bisimplex such that $[x]_{\mathrm{v}}=[x'']_{\mathrm{v}}$
and $x''d^0_{\mathrm{h}}=yd^1_{\mathrm{h}}$, and let $\beta,
\beta':\Delta[1,2]\to K$ be vertical homotopies from $x$ to $x'$ and
from $x$ to $x''$ respectively. Then, both bisimplices $\beta
d^0_{\mathrm{h}}:\Delta[0,2]\to K$ and
$\beta'd^0_{\mathrm{h}}:\Delta[0,2]\to K$  have the same vertical
faces. Since the $2^{\!\text{ nd}}$ homotopy groups of the Kan
complex $K_{0,\ast}$ vanish, it follows that $\beta
d^0_{\mathrm{h}}$ and $\beta'd^0_{\mathrm{h}}$ are vertically
homotopic (Fact \ref{p2}). Choose $\omega:\Delta[0,3]\to K$ any
vertical homotopy from $\beta d^0_{\mathrm{h}}$ to
$\beta'd^0_{\mathrm{h}}$, and then let $\Gamma:\Delta[1,3]\to K$ be
a solution to the extension problem
$$
\xymatrix@C=90pt@R=16pt{\Lambda^{\!1,3}[1,3]\ar[r]^-{(\omega,-;xd^0_{\mathrm{v}}s^0_{\mathrm{v}}
s^1_{\mathrm{v}},\,xs^1_{\mathrm{v}},\, \beta,-)} \ar @{^{(}->}[d] &
K \,.\\ \Delta[1,3]\ar @{.>}[ru]_-{\Gamma}}
$$
Then, the bisimplex $\widetilde{\beta}=\Gamma
d^3_{\mathrm{v}}:\Delta[1,2]\to K$ has vertical faces
$$\begin{array}{l}
\widetilde{\beta}d^0_{\mathrm{v}}=\Gamma
d^3_{\mathrm{v}}d^0_{\mathrm{v}}= \Gamma
d^0_{\mathrm{v}}d^2_{\mathrm{v}}=
xd^0_{\mathrm{v}}s^0_{\mathrm{v}}s^1_{\mathrm{v}}d^2_{\mathrm{v}}=
xd^0_{\mathrm{v}}s^0_{\mathrm{v}}, \\
\widetilde{\beta}d^1_{\mathrm{v}}=\Gamma
d^3_{\mathrm{v}}d^1_{\mathrm{v}}= \Gamma
d^1_{\mathrm{v}}d^2_{\mathrm{v}}= xs^1_{\mathrm{v}}d^2_{\mathrm{v}}=
x, \\
 \widetilde{\beta}d^2_{\mathrm{v}}= \Gamma
d^3_{\mathrm{v}}d^2_{\mathrm{v}}= \Gamma
d^2_{\mathrm{v}}d^2_{\mathrm{v}}= \beta d^2_{\mathrm{v}}= x',
\end{array}
$$
so that $\widetilde{\beta}$ is another vertical homotopy from $x$ to
$x'$, and moreover
$$
\widetilde{\beta}d^0_{\mathrm{h}}= \Gamma
d^3_{\mathrm{v}}d^0_{\mathrm{h}}= \Gamma
d^0_{\mathrm{h}}d^3_{\mathrm{v}}=\omega
d^3_{\mathrm{v}}=\beta'd^0_{\mathrm{h}},
$$
that is, $\widetilde{\beta}$ and $\beta'$ have both the same
horizontal $0$-face, say $\alpha$. Now let $\Phi:\Delta[1,3]\to K$
and $\theta:\Delta[2,2]\to K$ be solutions to the following
extension problems
$$
\xymatrix@C=90pt@R=16pt{\Lambda^{\!1,3}[1,3]\ar[r]^-{(\alpha s^1_{\mathrm{v}},-;
xd^0_{\mathrm{v}}s^0_{\mathrm{v}}s^1_{\mathrm{v}},\,\widetilde{\beta},\, \beta',-)} \ar @{^{(}->}[d] & K \\
\Delta[1,3]\ar @{.>}[ru]_-{\Phi}} \hspace{0.6cm}
\xymatrix@C=90pt@R=16pt{\Lambda^{\!1,2}[2,2]\ar[r]^-{(ys^1_{\mathrm{v}},-,\, \Phi d^3_{\mathrm{v}};\gamma
d^0_{\mathrm{v}}s^0_{\mathrm{v}},\,\gamma,-)} \ar @{^{(}->}[d] & K
\\ \Delta[2,2]\ar @{.>}[ru]_-{\theta}}
$$
where $\gamma:\Delta[2,1]\to K$ is any bisimplex such that $\gamma
d^2_{\mathrm{h}}=x'$ and $\gamma d^0_{\mathrm{h}}=y$. Then, $\theta$
is actually a vertical homotopy from $\gamma$ to $\gamma'=\theta
d^2_{\mathrm{v}}$, and this bisimplex $\gamma'$ satisfies that
$$\begin{array}{l}
\gamma'd^2_{\mathrm{h}}=\theta
d^2_{\mathrm{v}}d^2_{\mathrm{h}}=\theta
d^2_{\mathrm{h}}d^2_{\mathrm{v}}=\Phi
d^3_{\mathrm{v}}d^2_{\mathrm{v}}=\Phi
d^2_{\mathrm{v}}d^2_{\mathrm{v}}=\beta'd^2_{\mathrm{v}}= x'',\\
\gamma'd^0_{\mathrm{h}}=\theta d^2_{\mathrm{v}}d^0_{\mathrm{h}}=
\theta d^0_{\mathrm{h}}d^2_{\mathrm{v}}=
ys^1_{\mathrm{v}}d^2_{\mathrm{v}}= y.\end{array}
$$
Hence, $[x']_{\mathrm{h}}\circ_{\mathrm{h}}[y]_{\mathrm{h}}=[\gamma
d^1_{\mathrm{h}}]_{\mathrm{h}}$ whereas
$[x'']_{\mathrm{h}}\circ_{\mathrm{h}}[y]_{\mathrm{h}}=[\gamma'd^1_{\mathrm{h}}]_{\mathrm{h}}$.
Since the bisimplex $\theta d^1_{\mathrm{h}}:\Delta[1,2]\to K$ is a
vertical homotopy from $\gamma d^1_{\mathrm{h}}$ to $\gamma'
d^1_{\mathrm{h}}$, we conclude that $[[\gamma
d^1_{\mathrm{h}}]]=[[\gamma'd^1_{\mathrm{h}}]]$, that is,
$[[x']_{\mathrm{h}}\circ_{\mathrm{h}}[y]_{\mathrm{h}}]=
[[x'']_{\mathrm{h}}\circ_{\mathrm{h}}[y]_{\mathrm{h}}]$,
as required.

\vspace{0.2cm}

Suppose now $x_0,x_1,y:\Delta[1,1]\to K$ bisimplices with
$[[x_0]]=[[x_1]]$ and
$[x_0d^0_{\mathrm{h}}]_{\mathrm{v}}=[yd^1_{\mathrm{h}}]_{\mathrm{v}}$.
Then, for some $x:\Delta[1,1]\to K$, we have
$[x_0]_{\mathrm{v}}=[x]_{\mathrm{v}}$ and
$[x]_{\mathrm{h}}=[x_1]_{\mathrm{h}}$. Let $x'_0:\Delta[1,1]\to K$
be any bisimplex with $[x'_0]_{\mathrm{v}}=[x]_{\mathrm{v}}$ and
$x'_0d^0_{\mathrm{h}}=yd^1_{\mathrm{h}}$. Since
$[x'_0]_{\mathrm{v}}=[x_0]_{\mathrm{v}}$, we have
\begin{equation}\label{eq1}
[[x_0]]\circ_{\mathrm{h}}[[y]]=[[x'_0]_{\mathrm{h}}\circ
[y]_{\mathrm{h}}].
\end{equation}
Letting $\beta:\Delta[1,2]\to K$ be any vertical homotopy from $x$ to
$x'_0$ and $\delta:\Delta[2,1]\to K$ be any horizontal homotopy from
$x_1$ to $x$, we can choose $\theta:\Delta[2,2]\to K$, a bisimplex
making commutative the diagram
$$
\xymatrix@C=85pt@R=18pt{\Lambda^{\!1,2}[2,2]\ar[r]^-{(\beta
d^0_{\mathrm{h}}s^0_{\mathrm{h}},-,\, \beta; \delta
d^0_{\mathrm{v}}s^0_{\mathrm{v}},\, \delta, -)}\ar @{^{(}->}[d]& K
\\ \Delta[2,2]\ar[ru]_-\theta }
$$
Then, $\beta_1=\theta d^1_{\mathrm{h}}:\Delta[1,2]\to K$ is a
vertical homotopy from $x_1$ to $x'_1\!\!:=\beta_1
d^2_{\mathrm{v}}$, and since
$$
x'_1d^0_{\mathrm{h}}=\beta_1d^2_{\mathrm{v}}d^0_{\mathrm{h}}=\beta_1d^0_{\mathrm{h}}d^2_{\mathrm{v}}=\theta
d^1_{\mathrm{h}}d^0_{\mathrm{h}}d^2_{\mathrm{v}}= \theta
d^0_{\mathrm{h}}d^0_{\mathrm{h}}d^2_{\mathrm{v}}= \beta
d^0_{\mathrm{h}}d^2_{\mathrm{v}}=\beta
d^2_{\mathrm{v}}d^0_{\mathrm{h}}= x'_0d^0_{\mathrm{h}}=
yd^1_{\mathrm{h}},
$$
we have
\begin{equation}\label{eq2}
[[x_1]]\circ_{\mathrm{h}}[[y]]=[[x'_1]_{\mathrm{h}}\circ_{\mathrm{h}}[y]_{\mathrm{h}}].
\end{equation}
As $\theta d^2_{\mathrm{v}}:\Delta[2,1]\to K$ is a horizontal
homotopy from $x'_1$ to $x'_0$, we have
$[x'_0]_{\mathrm{h}}=[x'_1]_{\mathrm{h}}$. Therefore,  comparing
(\ref{eq1})  with (\ref{eq2}), we obtain the desired conclusion,
that is,
$$[[x_0]]\circ_{\mathrm{h}}[[y]]=[[x_1]]\circ_{\mathrm{h}}[[y]].$$

Finally, suppose $x,\,y_0,\,y_1:\Delta[1,1]\to K$ with
$[[y_0]]=[[y_1]]$ and
$[xd^0_{\mathrm{h}}]_{\mathrm{v}}=[y_0d^1_{\mathrm{h}}]_{\mathrm{v}}$.
Then, $[y_0]_{\mathrm{v}}=[y]_{\mathrm{v}},\,
[y]_{\mathrm{h}}=[y_1]_{\mathrm{h}}$, for some $y:\Delta[1,1]\to K$.
Let $x':\Delta[1,1]\to K$ be such that
$[x]_{\mathrm{v}}=[x']_{\mathrm{v}}$ and
$x'd^0_{\mathrm{h}}=yd^1_{\mathrm{h}}$. Since
$x'd^0_{\mathrm{h}}=y_1d^1_{\mathrm{h}}$, we have
\begin{equation}\label{eq3}
[[x]]\circ_{\mathrm{h}}[[y_1]]=[[x']_{\mathrm{h}}\circ_{\mathrm{h}}[y_1]_{\mathrm{h}}]=[[x']_{\mathrm{h}}\circ_{\mathrm{h}}[y]_{\mathrm{h}}]=[[\gamma
d^1_{\mathrm{h}}]],
\end{equation}
for $\gamma:\Delta[2,1]\to K$ any bisimplex with $\gamma
d^2_{\mathrm{h}}=x'$ and $\gamma d^0_{\mathrm{h}}=y$. Now, as
$[y_0]_{\mathrm{v}}=[y]_{\mathrm{v}}$, we can select a vertical
homotopy $\delta:\Delta[1,2]\to K$ from $y$ to $y_0$, and then a
bisimplex $\beta_0:\Delta[1,2]\to K$ making commutative the diagram
$$
\xymatrix@C=75pt@R=16pt{\Lambda^{\!1,2}[1,2]\ar[r]^-{(\delta
d^1_{\mathrm{h}},-;x'd^0_{\mathrm{v}}s^0_{\mathrm{v}},\, x', -)}\ar
@{^{(}->}[d]& K. \\ \Delta[1,2]\ar[ru]_-{\beta_0}}
$$
This bisimplex $\beta_0$ becomes a vertical homotopy from $x'$ to
$x'_0\!\!:=\beta_0d^2_{\mathrm{v}}$, and this $x'_0$ verifies that
$x'_0d^0_{\mathrm{h}}=y_0d^1_{\mathrm{h}}$. Hence,
\begin{equation}\label{eq4}
[[x]]\circ_{\mathrm{h}}[[y_0]]=[[x'_0]_{\mathrm{h}}\circ_{\mathrm{h}}[y_0]_{\mathrm{h}}].
\end{equation}
But, by taking $\theta :\Delta[2,2]\to K$ any bisimplex solving the
extension problem
$$
\xymatrix@C=75pt@R=16pt{\Lambda^{\!1,2}[2,2]\ar[r]^-{(\delta,-,\,
\beta_0 ; \gamma d^0_{\mathrm{v}}s^0_{\mathrm{v}},\, \gamma, -)}\ar
@{^{(}->}[d]& K, \\ \Delta[2,2]\ar[ru]_-\theta }
$$
we obtain a bisimplex $\gamma_0\!\!:=\theta
d^2_{\mathrm{v}}:\Delta[2,1]\to K$ satisfying that
$\gamma_0d^0_{\mathrm{h}} = y_0$ and $\gamma_0
d^2_{\mathrm{h}}=x'_0$, whence
$$
[[x]]\circ_{\mathrm{h}}[[y_0]]=[[\gamma_0 d^1_{\mathrm{h}}]].
$$
As the bisimplex $\theta d^1_{\mathrm{h}}:\Delta[1,2]\to K$ is
easily recognized to be a vertical homotopy from $\gamma
d^1_{\mathrm{h}}$ to $\gamma_0d^1_{\mathrm{h}}$, we conclude
$[[\gamma d^1_{\mathrm{h}}]]=[[\gamma_0d^1_{\mathrm{h}}]]$. Consequently,
the required equality
$$[[x]]\circ_{\mathrm{h}}[[y_0]]=[[x]]\circ_{\mathrm{h}}[[y_1]]$$
follows by comparing (\ref{eq3})  with (\ref{eq4}).
\end{proof}

Simply by exchanging the horizontal and vertical directions in the foregoing discussion, we also have a
well-defined vertical composition of squares $[[x]]$ and $[[y]]$ in $\adj K$, whenever
$[xd^0_{\mathrm{v}}]_{\mathrm{h}}=[yd^1_{\mathrm{v}}]_{\mathrm{h}}$, which is given by
\begin{equation}\label{compv}\begin{array}{lll}
[[x]]\circ_{\mathrm{v}} [[y]]=[[x']_{\mathrm{v}}\circ_{\mathrm{v}}[y]_{\mathrm{v}}]& \text{ if }&
[x]_{\mathrm{h}}=[x']_{\mathrm{h}} \text{ and } x'd^0_{\mathrm{v}}=yd^1_{\mathrm{v}},\end{array}
\end{equation}
where $[x']_{\mathrm{v}}\circ_{\mathrm{v}}[y]_{\mathrm{v}}$ is the composite in the fundamental groupoid $\mathrm{P}
K_{1,*}$, that is,
\begin{equation}
[[x]]\circ_{\mathrm{v}}[[y]] = [[\gamma d^1_{\mathrm{v}}]]
\end{equation}
for $\gamma : \Delta[1,2]\to K$ any bisimplex with $\gamma d^2_{\mathrm{v}}=x'$ and $\gamma d^0_{\mathrm{v}}=y$.

\begin{theorem}
$\adj K$ is a double groupoid satisfying the filling condition.
\end{theorem}
\begin{proof}
We first observe that, with both defined horizontal and vertical compositions, the squares in $\adj K$ form groupoids.
The associativity for the horizontal composition of squares in $\adj K$ follows from the associativity of the
composition of morphisms in the fundamental groupoid $\mathrm{P} K_{\ast,1}$. In effect, let $[[x]],\,[[y]]$ and
$[[z]]$ be three horizontally composable squares in $\adj K$. By changing representatives if  necessary, we can
assume that $xd^0_{\mathrm{h}}=yd^1_{\mathrm{h}}$ and $yd^0_{\mathrm{h}}=zd^1_{\mathrm{h}}$. Then,
$$\begin{array}{lllll}
[[x]]\circ_{\mathrm{h}}\left([[y]]\circ_{\mathrm{h}}[[z]]\right) &=&
[[x]]\circ_{\mathrm{h}}\left[[y]_{\mathrm{h}}\circ_{\mathrm{h}}[z]_{\mathrm{h}}\right] &=&
\left[[x]_{\mathrm{h}}\circ_{\mathrm{h}}\left([y]_{\mathrm{h}}\circ_{\mathrm{h}}[z]_{\mathrm{h}}\right)\right]\\
&=&[([x]_{\mathrm{h}}\circ_{\mathrm{h}}[y]_{\mathrm{h}})\circ_{\mathrm{h}}[z]_{\mathrm{h}}]&=&
[[x]_{\mathrm{h}}\circ_{\mathrm{h}}[y]_{\mathrm{h}}]\circ_{\mathrm{h}}[[z]]\\&=&
([[x]]\circ_{\mathrm{h}}[[y]])\circ_{\mathrm{h}}[[z]]. \end{array}$$

The horizontal identity square on the vertical morphism represented by a bisimplex $u\!:\!\Delta[0,1]\to K$  is
\begin{equation}
\mathrm{I}^\mathrm{h}[u]_{\mathrm{v}}=[[us^0_{\mathrm{h}}]]
\end{equation}
(recall Lemma \ref{p5}), as  can be easily deduced from the fact that $[us^0_{\mathrm{h}}]_{\mathrm{h}}$ is the
identity morphism on $u$ in the groupoid $\mathrm{P}K_{\ast,1}$. Thus, for example, for any $x:\Delta[1,1]\to K$,
$$
[[x]]\circ_{\mathrm{h}}\mathrm{I}^\mathrm{h}[xd^0_{\mathrm{h}}]_{\mathrm{v}}=[[x]_{\mathrm{h}}\circ_{\mathrm{h}}
[xd^0_{\mathrm{h}}s^0_{\mathrm{h}}]_\mathrm{h}]=[[x]_{\mathrm{h}}]=[[x]].
$$

The horizontal inverse in $\adj K$ of a square $[[x]]$ is
$[[x]]^{\text{-}1_{\mathrm{h}}}=[[x]^{\text{-}1}_{\mathrm{h}}]$, where $[x]^{\text{-}1}_{\mathrm{h}}$ is the inverse of
$[x]_{\mathrm{h}}$ in $\mathrm{P}K_{\ast,1}$, as  is easy to verify:
$$
[[x]]\circ_{\mathrm{h}}[[x]^{\text{-}1}_{\mathrm{h}}]=[[x]_{\mathrm{h}}\circ_{\mathrm{h}}[x]^{\text{-}1}_{\mathrm{h}}]
=[[xd^1_{\mathrm{h}}s^0_{\mathrm{h}}]]=\mathrm{I}^\mathrm{h}[xd^1_{\mathrm{h}}]_{\mathrm{v}}.
$$

Similarly, we see that the associativity for the vertical composition of squares in $\adj K$ follows from the
associativity of the composition in the fundamental groupoid $\mathrm{P}K_{1,\ast})$,  that the vertical identity square on
the horizontal morphism represented by a bisimplex $f:\Delta[1,0]\to K$  is
$\mathrm{I}^v([f]_{\mathrm{h}})=[[fs^0_{\mathrm{v}}]]$, and that the vertical inverse in $\adj K$ of a square $[[x]]$
is $[[x]^{-1}_{\mathrm{v}}]$, where $[x]^{\text{-}1}_{\mathrm{v}}$ denotes the inverse of $[x]_{\mathrm{v}}$ in
$\mathrm{P} K_{1,\ast}$.

We are now ready to prove that $\adj K$ is actually a double groupoid. {\bf Axiom 1} is easily verified. Thus, for
example, given any $x:\Delta[1,1]\to K$,
$$
\mathrm{s}^\mathrm{h}\mathrm{s}^\mathrm{v}[[x]]=\mathrm{s}^\mathrm{h}[xd^0_{\mathrm{v}}]_{\mathrm{h}}
=xd^0_{\mathrm{v}}d^0_{\mathrm{h}}=xd^0_{\mathrm{h}}d^0_{\mathrm{v}}=\mathrm{s}^\mathrm{v}
[xd^0_{\mathrm{h}}]_{\mathrm{v}}=\mathrm{s}^\mathrm{v}\mathrm{s}^\mathrm{h}[[x]],
$$
or, given any $f:\Delta[1,0]\to K$,
$$
\mathrm{s}^\mathrm{h}\mathrm{I}^\mathrm{v}[f]_{\mathrm{h}}=\mathrm{s}^\mathrm{h}[[fs^0_{\mathrm{v}}]]=
[fs^0_{\mathrm{v}}d^0_{\mathrm{h}}]_{\mathrm{v}}=[fd^0_{\mathrm{h}}s^0_{\mathrm{v}}]_{\mathrm{v}}=
\mathrm{I}^\mathrm{v}fd^0_{\mathrm{h}}= \mathrm{I}^\mathrm{v} \mathrm{s}^\mathrm{h}[f]_{\mathrm{h}},
$$
and so on. Also, for any $a:\Delta[0,0]\to K$,
$$
\mathrm{I}^\mathrm{h}\mathrm{I}^\mathrm{v}a=\mathrm{I}^\mathrm{h}[as^0_{\mathrm{v}}]_{\mathrm{v}}=
[[as^0_{\mathrm{v}}s^0_{\mathrm{h}}]]=[[as^0_{\mathrm{h}}s^0_{\mathrm{v}}]]=
\mathrm{I}^\mathrm{v}[as^0_{\mathrm{h}}]_{\mathrm{h}}= \mathrm{I}^\mathrm{v}\mathrm{I}^\mathrm{h}a.
$$

For {\bf Axiom 2} (i), let  $[[x]]$ and $[[y]]$ be two horizontally composable squares in $\adj K$. We can assume that
$xd^0_{\mathrm{h}}=yd^1_{\mathrm{h}}$, and then $[[x]]\circ_{\mathrm{h}}[[y]]= [[\gamma d^1_{\mathrm{h}}]]$, for any
$\gamma :\Delta[2,1]\to K$ with $\gamma d^2_{\mathrm{h}}=x$ and $\gamma d^0_{\mathrm{h}}=y$. Hence,
$$\begin{array}{lll} \mathrm{s}^\mathrm{v}([[x]]\circ_{\mathrm{h}}[[y]])&=&[\gamma d^1_{\mathrm{h}}
d^0_{\mathrm{v}}]_{\mathrm{h}}=[\gamma d^0_{\mathrm{v}}d^1_{\mathrm{h}}]=[\gamma
d^0_{\mathrm{v}}d^2_{\mathrm{h}}]_{\mathrm{h}} \circ_{\mathrm{h}} [\gamma
d^0_{\mathrm{v}}d^0_{\mathrm{h}}]_{\mathrm{h}}\\[5pt]&=&[\gamma
d^2_{\mathrm{h}}d^0_{\mathrm{v}}]_{\mathrm{h}}\circ_{\mathrm{h}}[\gamma
d^0_{\mathrm{h}}d^0_{\mathrm{v}}]=[xd^0_{\mathrm{v}}]_{\mathrm{h}}\circ_{\mathrm{h}}[yd^0_{\mathrm{v}}]_{\mathrm{h}}
=\mathrm{s}^\mathrm{v}[[x]]\circ_\mathrm{h}
\mathrm{s}^\mathrm{v}[[y]],\\[6pt]\mathrm{t}^\mathrm{v}([[x]]\circ_{\mathrm{h}}[[y]])&=&[\gamma d^1_{\mathrm{h}}
d^1_{\mathrm{v}}]_{\mathrm{h}}=[\gamma d^1_{\mathrm{v}}d^1_{\mathrm{h}}]_{\mathrm{h}}=[\gamma
d^1_{\mathrm{v}}d^2_{\mathrm{h}}]_{\mathrm{h}}\circ_{\mathrm{h}} [\gamma
d^1_{\mathrm{v}}d^0_{\mathrm{h}}]_{\mathrm{h}}\\[5pt]&=&[\gamma
d^2_{\mathrm{h}}d^1_{\mathrm{v}}]_{\mathrm{h}}\circ_{\mathrm{h}}[\gamma
d^0_{\mathrm{h}}d^1_{\mathrm{v}}]_{\mathrm{h}}=[xd^1_{\mathrm{v}}]_{\mathrm{h}}\circ_{\mathrm{h}}
[yd^1_{\mathrm{v}}]_{\mathrm{h}}=\mathrm{t}^\mathrm{v}[[x]]\circ_{\mathrm{h}} \mathrm{t}^\mathrm{v}[[y]].
\end{array}
$$

{\bf Axiom 2} (ii) is proved analogously, and for (iii), let $f,\,f':\Delta[1,0]\to K$ be maps with
$fd^0_{\mathrm{h}}=f'd^1_{\mathrm{h}}$. Then, $[f]_{\mathrm{h}}\circ_{\mathrm{h}}[f']_{\mathrm{h}}=[\gamma
d^1_{\mathrm{h}}]_{\mathrm{h}}$, for $\gamma:\Delta[2,0]\to K$ any bisimplex with $\gamma d^2_{\mathrm{h}}=f$ and
$\gamma d^0_{\mathrm{h}}=f'$, and we have the equalities:
$$
\mathrm{I}^\mathrm{v}([f]_{\mathrm{h}}\circ_{\mathrm{h}}[f']_{\mathrm{h}}) = \mathrm{I}^\mathrm{v}[\gamma
d^1_{\mathrm{h}}]_{\mathrm{h}}=[[\gamma d^1_{\mathrm{h}}s^0_{\mathrm{v}}]] = [[\gamma s^0_{\mathrm{v}}
d^1_{\mathrm{h}}]]=[[\gamma s^0_{\mathrm{v}}d^2_{\mathrm{h}}]]\circ_\mathrm{h} [[\gamma
s^0_{\mathrm{v}}d^0_{\mathrm{h}}]]=
\mathrm{I}^\mathrm{v}[f]_{\mathrm{h}}\circ_{\mathrm{h}}\mathrm{I}^\mathrm{v}[f']_{\mathrm{h}}.
$$
And similarly one sees that
$\mathrm{I}^\mathrm{h}([u]_{\mathrm{v}}\circ_{\mathrm{v}}[u']_{\mathrm{v}})=\mathrm{I}^\mathrm{h}
[u]_{\mathrm{v}}\circ_{\mathrm{v}}\mathrm{I}^\mathrm{h}[u']_{\mathrm{v}}$ for any $u,\,u':\Delta[0,1]\to K$ with
$ud^0_{\mathrm{v}}=u'd^1_{\mathrm{v}}$.

To verify {\bf Axiom 3}, that is, to prove that the interchange law holds in $\adj K$, let
$$
\xymatrix@C=20pt@R=18pt{\cdot \ar @{}[dr]|{\textstyle [[x]]} & \cdot \ar[l]\ar @{}[dr]|{\textstyle[[x']]} & \cdot\ar[l]\\
\cdot \ar[u]\ar @{}[dr]|{\textstyle[[y]]} & \cdot \ar[l]\ar[u]\ar @{}[dr]|{\textstyle[[y']]} & \cdot\ar[u]\ar[l] \\
\cdot \ar[u] & \cdot \ar[l]\ar[u] & \cdot \ar[l]\ar[u]}
$$
be squares in $\adj K$. By an iterated use of Lemma \ref{lad1} (and its corresponding version for vertical direction),
we can assume that $xd^0_{\mathrm{h}}=x'd^1_{\mathrm{h}},\, xd^0_{\mathrm{v}}=yd^1_{\mathrm{v}},\,
x'd^0_{\mathrm{v}}=y'd^1_{\mathrm{v}}$ and $yd^0_{\mathrm{h}}=y'd^1_{\mathrm{h}}$. Let $\alpha:\Delta[2,1]\to K$ and
$\beta:\Delta[1,2]\to K$ be bisimplicial maps such that $\alpha d^2_{\mathrm{h}} =y,\, \alpha d^0_{\mathrm{h}} = y',\,
\beta d^2_{\mathrm{v}}=x'$ and $\beta d^0_{\mathrm{v}}=y'$; therefore, $[[y]]\circ_{\mathrm{h}}[[y']]=[[\alpha
d^1_{\mathrm{h}}]]$ and $[[x']]\circ_{\mathrm{v}}[[y']]=[[\beta d^1_{\mathrm{v}}]]$. Now we select bisimplices
$\gamma:\Delta[1,2]\to K$ and $\delta:\Delta[2,1]\to K$ as respective solutions to the following extension problems:
$$
\xymatrix@C=50pt@R=16pt{\Lambda^{\!1,1}[1,2]\ar@{^{(}->}[d]\ar[r]^-{(\beta d^1_{\mathrm{h}},-;y,-,x)} & K\\
                \Delta[1,2]\ar @{.>}[ur]_-{\gamma}}\hspace{20pt}
\xymatrix@C=50pt@R=16pt{\Lambda^{\!1,1}[2,1]\ar @{^{(}->}[d]\ar[r]^-{(x',-, x;\,\alpha d^1_{\mathrm{v}},-)} & K\\
                \Delta[2,1]\ar @{.>}[ur]_-{\delta}}
$$
Then $[[x]]\circ_{\mathrm{v}}[[y]]=[[\gamma d^1_{\mathrm{v}}]],\, [[x]]\circ_{\mathrm{h}}[[x']]=[[\delta
d^1_{\mathrm{h}}]]$ and, moreover, we can find a bisimplex $\theta:\Delta[2,2]\to K$ making the triangle
below commutative.
$$
\xymatrix@C=50pt@R=16pt{\Lambda^{\!1,1}[2,2]\ar @{^{(}->}[d]\ar[r]^-{(\beta ,-,\gamma;\alpha,-, \delta)} & K\\
                \Delta[2,2]\ar @{.>}[ur]_-{\theta}}
$$
Letting $\phi=\theta d^1_{\mathrm{h}}:\Delta[1,2]\to K$ and $\psi=\theta d^1_{\mathrm{v}}:\Delta[2,1]\to K$, we have
the equalities:
$$
\begin{array}{ll}
 \phi d^2_{\mathrm{v}}=\theta d^2_{\mathrm{v}}d^1_{\mathrm{h}}=\delta d^1_{\mathrm{h}},
 & \phi d^0_{\mathrm{v}}=\theta d^0_{\mathrm{v}}d^1_{\mathrm{h}}=\alpha d^1_{\mathrm{h}},\\[5pt]
 \psi d^2_{\mathrm{h}}=\theta d^2_{\mathrm{h}}d^1_{\mathrm{v}}=\gamma d^1_{\mathrm{v}},
 & \psi d^0_{\mathrm{h}}=\theta d^0_{\mathrm{h}}d^1_{\mathrm{v}}=\beta d^1_{\mathrm{v}},
\end{array}$$
whence,
$$
\begin{matrix}
 ([[x]]\circ_{\mathrm{h}}[[x']])\circ_{\mathrm{v}}([[y]]\circ_{\mathrm{h}}[[y']])=[[\delta d^1_{\mathrm{h}}]]\circ_{\mathrm{v}}[[\alpha d^1_{\mathrm{h}}]]= [[\phi
 d^1_{\mathrm{v}}]],\\[5pt]
 ([[x]]\circ_{\mathrm{v}}[[y]])\circ_{\mathrm{h}}([[x']]\circ_{\mathrm{h}}[[y']])=[[\gamma d^1_{\mathrm{v}}]]\circ_{\mathrm{h}}[[\beta d^1_{\mathrm{v}}]]= [[\psi d^1_{\mathrm{h}}]].
\end{matrix}
$$
Since $\phi d^1_{\mathrm{v}}=\theta d^1_{\mathrm{h}} d^1_{\mathrm{v}}=\theta d^1_{\mathrm{v}}d^1_{\mathrm{h}}=\psi
d^1_{\mathrm{h}}$, the interchange law follows.

Thus, $\adj K$ is a double groupoid and, moreover, it satisfies the filling condition: given morphisms
$$
\xymatrix@C=18pt@R=12pt{\cdot & \cdot \ar[l]_(0.4){[g]_{\mathrm{h}}} \\ & \cdot \ar[u]_-{[u]_{\mathrm{v}}}}
$$
represented by bisimplices $u:\Delta[0,1]\to K$ and $g:\Delta[1,0]\to K$ with $gd^0_{\mathrm{h}}=ud^1_{\mathrm{v}}$, if
$x:\Delta[1,1]\to K$ is any solution to the extension problem
$$
\xymatrix@C=50pt@R=12pt{\Lambda^{\!0,1}[1,1]\ar @{^{(}->}[d]\ar[r]^-{(-,\,g;\,u,-)} & K \\
\Delta[1,1]\ar @{.>}[ur]_-x}
$$
then the bihomotopy class of $x$ is a square in $\adj K$, \parbox{30pt}{
\xymatrix@C=18pt@R=14pt{\cdot\ar @{}[dr]|{\textstyle [[x]]} & \cdot \ar[l]_-{[g]_{\mathrm{h}}} \\
                                        \cdot\ar[u] & \cdot\ar[l] \ar[u]_-{[u]_{\mathrm{v}}}}},
as required.
\end{proof}

The construction of the double groupoid $\adj K$ is clearly functorial on $K$, and we have the following:

\begin{theorem} \label{t1} The double nerve construction, $\mathcal{G} \mapsto \dn\mathcal{G}$, embeds, as a
reflexive subcategory, the category of double groupoids satisfying
the filling condition into the category of those bisimplicial sets
$K$ that satisfy the extension condition and such that
$\pi_2(K_{\ast,0},a)=0=\pi_2(K_{0,\ast},a)$ for all vertices $a\in
K_{0,0}$. The reflector functor for such bisimplicial sets is given
by the above described homotopy double groupoid construction $$
K\mapsto \adj K.$$ Thus, $\adj \dn\mathcal{G}=\mathcal{G}$, and
there are natural bisimplicial maps
\begin{equation}
\epsilon(K):K\to \dn\adj K,
\end{equation}
such that $\adj \epsilon =\mathrm{id}$ and $\epsilon \dn=\mathrm{id}$.
\end{theorem}

\begin{proof} From Theorem \ref{thfc}(ii), if $\mathcal{G}$ is any double groupoid satisfying the filling
condition, then its double nerve $\dn\mathcal{G}$ satisfies the
extension condition and, since both simplicial sets
$\dn\mathcal{G}_{\ast,0}$ and $\dn\mathcal{G}_{0,\ast}$ are nerves
of groupoids, all homotopy groups $\pi_2(\dn\mathcal{G}_{\ast,0},a)$
and $\pi_2(\dn\mathcal{G}_{0,\ast},a)$ vanish.  Moreover, since the
bihomotopy relation is trivial on the bisimplices $\Delta[p,q]\to
\dn\mathcal{G}$, for $p\geq 1$ or $q\geq 1$, it is easy to see that
 $\adj
\dn\mathcal{G}=\mathcal{G}$.

For any bisimplicial set $K$ in the hypothesis of the theorem, there is
a natural bisimplicial map
$$
\epsilon =\epsilon(K):K\to \dn\adj K,
$$
that takes a bisimplex $x:\Delta[p,q]\to K$, of $K$, to the
bisimplex $\epsilon x:[p]\otimes[q]\to\adj K$, of $\dn\adj K$,
defined by the $p\times q$ configuration of squares in $\adj K$
$$
\xymatrix@C=20pt@R=18pt{\ar@{-}@/_1pc/@<1.5pc>[d]&\epsilon_i^rx\ar
@{}[rd]|{\textstyle \epsilon_{i,j}^{r,s}x} &  \epsilon_j^r
\ar[l]_-{\textstyle \epsilon_{i,j}^rx}&
\ar@{-}@/^0.9pc/@<-1.5pc>[d]^(0.7){\hspace{-5pt}\scriptsize
\begin{array}{l}0\leq i\leq j\leq p\\[-1pt] 0\leq r\leq s\leq q\end{array},}\\ &
\epsilon_i^sx\ar[u]^-{\textstyle \epsilon_i^{r,s}x}  & \epsilon_j^sx
\ar[l]^-{\textstyle \epsilon_{i,j}^sx} \ar[u]_-{\textstyle
\epsilon_j^{r,s}x}&}
$$
where
$$
\begin{array}{ccl}
 \epsilon^{r,s}_{i,j}x & = & [[xd^p_{\mathrm{h}}\cdots d^{j+1}_{\mathrm{h}}d^{j-1}_{\mathrm{h}}\cdots
  d^{i+1}_{\mathrm{h}}
 d^{i-1}_{\mathrm{h}}\cdots d^0_{\mathrm{h}}d^q_{\mathrm{v}}\cdots d^{s+1}_{\mathrm{v}}d^{s-1}_{\mathrm{v}}
 \cdots
 d^{r+1}_{\mathrm{v}}d^{r-1}_{\mathrm{v}}\cdots
 d_{\mathrm{v}}^0]],\\[4pt]
 \epsilon^{r,s}_jx & = & [xd^p_{\mathrm{h}}\cdots d^{j+1}_{\mathrm{h}}d^{j-1}_{\mathrm{h}}\cdots d^0_{\mathrm{h}}
 d^q_{\mathrm{v}}\cdots d^{s+1}_{\mathrm{v}}d^{s-1}_{\mathrm{v}}\cdots d^{r+1}_{\mathrm{v}}d^{r-1}_{\mathrm{v}}\cdots
  d^0_{\mathrm{v}}]_{\mathrm{v}},\\[5pt]
 \epsilon^r_{i,j}x & = & [xd^p_{\mathrm{h}}\cdots d^{j+1}_{\mathrm{h}}d^{j-1}_{\mathrm{h}}\cdots d^{i+1}_{\mathrm{h}}
 d^{i-1}_{\mathrm{h}}\cdots d^0_{\mathrm{h}}d^q_{\mathrm{v}}\cdots d^{r+1}_{\mathrm{v}}d^{r-1}_{\mathrm{v}}
 \cdots d^0_{\mathrm{v}}]_{\mathrm{h}},\\[5pt]
 \epsilon^r_ix& = & xd^p_{\mathrm{h}}\cdots d^{i+1}_{\mathrm{h}}d^{i-1}_{\mathrm{h}}\cdots d^0_{\mathrm{h}}
 d^q_{\mathrm{v}}\cdots d^{r+1}_{\mathrm{v}}d^{r-1}_{\mathrm{v}}\cdots d^0_{\mathrm{v}}.
\end{array}
$$

Since a straightforward verification shows that $\adj \epsilon(K)$
is the identity map on $\adj K$, for any $K$, and $\epsilon
(\dn\mathcal{G})$ is the identity map on $\dn\mathcal{G}$, for any
double groupoid $\mathcal{G}$, it follows that $\dn$ is right
adjoint to $\adj$, with $\epsilon$ and the identity being the unit
and the counit of the adjunction respectively.
\end{proof}

With the next theorem we show that the double groupoid $\adj K$
represents the same homotopy 2-type as the bisimplicial set $K$.

\begin{theorem}\label{4.5} Let $K$ be any bisimplicial set satisfying the extension condition and such that
$\pi_2(K_{0,\ast},a)\!=\!0\!=\!\pi_2(K_{\ast,0},a)$ for all base
vertices $a$. Then, the induced map by unit of the adjunction
$|\epsilon|: |K|\to |\dn\adj K|\!=\!|\adj K|$ is a weak homotopy
$2$-equivalence.
\end{theorem}
\begin{proof} By Facts \ref{f18} (1) and (3) and Theorem \ref{thfc},
the map $|\epsilon|:|K|\to |\dn\adj K|$ is, up to natural homotopy
equivalences, induced by the simplicial map $\w\epsilon:\w K\to
\w\dn\adj K$, where both $\w K$ and $\w \dn\adj K$ are
Kan-complexes.

At dimension 0, we have the equalities $\w K_0=K_{0,0}=\w\dn\adj
K_0$, and the map $\w\epsilon$ is the identity on 0-simplices. At
dimension 1, the map $$\w\epsilon: (x_{0,1},x_{1,0})\mapsto
([x_{0,1}]_{\mathrm{v}},[x_{1,0}]_{\mathrm{h}}),$$ is clearly
surjective, whence we conclude that the induced
$$
\pi_0\w\epsilon :\pi_0\w K\to \pi_0\w\dn\adj
K\!\overset{(\ref{i-i})}\cong\!\pi_0\adj K
$$
is a bijection and also that, for any vertex $a\in K_{0,0}$, that
induced on the $\pi_1$-groups
$$
\pi_1\w\epsilon :\pi_1(\w K,a)\to \pi_1(\w\dn\adj
K,a)\!\overset{(\ref{i-i})}\cong\!\pi_1(\adj K,a)
$$
is surjective. To see that $\pi_1\w\epsilon$ is actually an
isomorphism, suppose that $(x_{0,1},x_{1,0})\in \w K_1$, with
$x_{0,1}d^1_{\mathrm{v}}=a=x_{1,0}d^0_{\mathrm{h}}$, represents an
element in the kernel of $\pi_1\w\epsilon$. This implies the
existence of a bisimplex $x:\Delta[1,1]\to K$ whose bihomotopy class
is a square in $\adj K$ with boundary as in
$$
\xymatrix@C=14pt@R=14pt{a\ar @{}[dr]|{\textstyle [[x]]} &
a\ar@<-0.1ex>@{-}[l]\ar@<0.1ex>@{-}[l]_-{[a
s^0_{\mathrm{h}}]_{\mathrm{h}}} \\
\cdot \ar[u]^-{[x_{0,1}]_{\mathrm{v}}} &
a\ar[l]^-{[x_{1,0}]_{\mathrm{h}}}\ar@<-0.1ex>@{-}[u]\ar@<0.1ex>@{-}[u]
_-{[as^0_{\mathrm{v}}]_{\mathrm{v}}}}
$$
Using  Lemma \ref{lad1} twice (one in each direction), we can find a
bisimplex $x_{1,1}:\Delta[1,1]\to K$, such that $[[x_{1,1}]]=[[x]]$,
$x_{1,1}d^1_{\mathrm{v}}=as^0_{\mathrm{h}}$, and
$x_{1,1}d^0_{\mathrm{h}}=as^0_{\mathrm{v}}$. Moreover, since
$[x_{1,1}d^0_{\mathrm{v}}]_{\mathrm{h}}=[x_{1,0}]_{\mathrm{h}}$ and
$[x_{1,1}d^1_{\mathrm{h}}]_{\mathrm{v}}=[x_{0,1}]_{\mathrm{v}}$,
there are bisimplices $x_{2,0}:\Delta[2,0]\to K$ and
$x_{0,2}:\Delta[0,2]\to K$, with faces as in the picture
$$
\xymatrix@C=26pt@R=20pt{a &
a\ar[l]_-{as^0_{\mathrm{v}}}\ar@{}[rd]|{\textstyle x_{1,1}} &
a\ar[l]_-{as^0_{\mathrm{h}}} \\
& \cdot \ar@{}@<3pt>[lu]_(0.5){\textstyle
x_{0,2}}\ar[lu]\ar@{}@<-3pt>[lu]^-{x_{0,1}}\ar[u]
& a\ar[l]\ar[u]_-{as^0_{\mathrm{v}}} \\
& & a\ar[lu]\ar@{}@<3pt>[lu]_-{\textstyle
x_{2,0}}\ar@{}@<-3pt>[lu]^-{x_{1,0}} \ar[u]_-{as^0_{\mathrm{h}}}}
$$
This amounts to saying that the triplet $(x_{0,2},x_{1,1},x_{2,0})$
is a 2-simplex of $\w K$ which is a homotopy from
$(x_{0,1},x_{1,0})$ to $(as^0_{\mathrm{v}},as^0_{\mathrm{v}})$.
Then, $(x_{0,1},x_{1,0})$ represents the identity element of the
group $\pi_1(\w K,a)$. This proves that $\pi_1\w\epsilon$ is an
isomorphism.

Let us now analyze the homomorphism $$ \pi_2\w\epsilon :\pi_2(\w
K,a)\to \pi_2(\w\dn\adj K,a)\!\overset{(\ref{i-i})}\cong\!\pi_2(\adj
K,a).$$ An element of $\pi_2(\adj K,a)$ is a square in $\adj K$ of
the form
$$
\xymatrix@C=14pt@R=14pt{ a\ar @{}[dr]|{\textstyle [[x]]} & a \ar[l]_-{[as^0_{\mathrm{h}}]_{\mathrm{h}}} \\
                                        a \ar[u]^-{[as^0_{\mathrm{v}}]_{\mathrm{v}}} & a\ar[l]^-{[as^0_{\mathrm{h}}]_{\mathrm{h}}}\ar[u]_-{[as^0_{\mathrm{v}}]_{\mathrm{v}}}}
$$
and the homomorphism $\pi_2\w\epsilon$ is induced by the mapping
$$
\begin{array}{c}
\xymatrix@C=26pt@R=20pt{a & a\ar[l]_-{as^0_{\mathrm{v}}}\ar
@{}[rd]|{\textstyle x_{1,1}} &
a\ar[l]_-{as^0_{\mathrm{h}}} \\
& \ar@{}@<3pt>[lu]_(0.5){\textstyle
x_{0,2}}\ar[lu]\ar@{}@<-3pt>[lu]^-{as^0_{\mathrm{v}}}\ar[u]a & a\ar[l]\ar[u]_-{as^0_{\mathrm{v}}} \\
                                & & a\ar[lu]\ar@{}@<3pt>[lu]_-{\textstyle
x_{2,0}}\ar@{}@<-3pt>[lu]^-{as^0_{\mathrm{h}}}
\ar[u]_-{as^0_{\mathrm{h}}}}
\end{array} \longmapsto [[x_{1,1}]].
$$
That $\pi_2\w\epsilon$ is surjective is proven using a parallel
argument to that given previously for proving that $\pi_1\w\epsilon$
is injective (given $[[x]]$, using  Lemma \ref{lad1} twice, we can
find $x_{1,1}:\Delta[1,1]\to K$, etc.). To prove that
$\pi_2\w\epsilon$ is also injective, suppose
$(x_{0,2},x_{1,1},x_{2,0})$ as above, representing an element of
$\pi_2(\w K,a)$ into the kernel of $\pi_2\w\epsilon$, that is, such
that $[[x_{1,1}]]=[[as^0_{\mathrm{h}}s^0_{\mathrm{v}}]]$. Then,
there is a bisimplex $y:\Delta[1,1]\to K$ such that
$[x_{1,1}]_{\mathrm{v}}=[y]_{\mathrm{v}}$ and
$[y]_{\mathrm{h}}=[as^0_{\mathrm{h}}s^0_{\mathrm{v}}]_{\mathrm{h}}$,
whence we can find bisimplices $\alpha':\Delta[1,2]\to K$ and
$\beta':\Delta[2,1]\to K$ such that
$$
\begin{array}{llllll}
   \alpha'd^0_{\mathrm{v}}=yd^0_{\mathrm{v}}s^0_{\mathrm{v}}, & \alpha'd^1_{\mathrm{v}}=y, &
    \alpha'd^2_{\mathrm{v}}=x_{1,1},&
   \beta'd^0_{\mathrm{h}}=as^0_{\mathrm{h}}s^0_{\mathrm{v}}, & \beta' d^1_{\mathrm{h}}=as^0_{\mathrm{h}}s^0_{\mathrm{v}}, & \beta' d^2_{\mathrm{h}}=y.
\end{array}
$$
Let us now choose $\theta:\Delta[2,2]\to K$ and
$\theta':\Delta[1,3]\to K$ as respective solutions to the following
extension problems
$$
\xymatrix@C=120pt@R=15pt{\Lambda^{\!2,0}[2,2]\ar
@{^{(}->}[d]\ar[r]^-{(as^0_{\mathrm{h}}s^0_{\mathrm{v}}
s^0_{\mathrm{v}},\, as^0_{\mathrm{h}}s^0_{\mathrm{v}}s^0_{\mathrm{v}},-;\,-, \beta'd^1_{\mathrm{v}}s^0_{\mathrm{v}},\,\beta')} & K \\
                        \Delta[2,2]\ar@{.>}[ru]_-{\theta}}
\xymatrix@C=90pt@R=15pt{\Delta[1]\otimes \Lambda^{\!2}[3] \ar
@{^{(}->}[d]
\ar[r]^-{(\theta d^2_{\mathrm{h}}d^0_{\mathrm{v}}s^0_{\mathrm{v}},\,\theta d^2_{\mathrm{h}},-,\alpha')} & K \\
  \Delta[1,3]\ar@{.>}[ru]_-{\theta'}}
$$
Then, for $\alpha =\theta'd^2_{\mathrm{v}}:\Delta[1,2]\to K$ and
$\beta = \theta d^0_{\mathrm{v}}:\Delta[2,1]\to K$, we have the
equalities
\begin{equation}\label{rar}
\begin{array}{lllll}
    \alpha d^0_{\mathrm{v}}=\beta d^2_{\mathrm{h}}, & \alpha d^1_{\mathrm{v}}=
    as^0_{\mathrm{h}}s^0_{\mathrm{v}},
    & \alpha d^2_{\mathrm{v}}=x_{1,1},&
    \beta d^0_{\mathrm{h}}=as^0_{\mathrm{h}}s^0_{\mathrm{v}},
    & \beta d^1_{\mathrm{h}}=as^0_{\mathrm{h}}s^0_{\mathrm{v}}.
\end{array}
\end{equation}

By
Lemma \ref{p2}, as the 2$^{\text {nd}}$ homotopy groups of $K_{0,\ast}$ vanish and
both bisimplices $\alpha d^0_{\mathrm{h}}$ and
$as^0_{\mathrm{v}}s^0_{\mathrm{v}}$ have the same vertical faces, there is a vertical homotopy
$\omega:\Delta[0,3]\to K $ from $as^0_{\mathrm{v}}s^0_{\mathrm{v}}$
to $\alpha d^0_{\mathrm{h}}$. And similarly, since $\beta
d^1_{\mathrm{v}}$ and $x_{2,0}$ have the same horizontal faces and
the 2$^{\text {nd}}$ homotopy groups of $K_{\ast,0}$ are all
trivial, there is a horizontal homotopy, say $\omega':\Delta[3,0]\to
K$,  from $\beta d^1_{\mathrm{v}}$ to $x_{2,0}$. Now, let
$\Gamma:\Delta[1,3]\to K$ and $\Gamma':\Delta[3,1]\to K$ be
bisimplices solving, respectively, the extension problems
$$
\xymatrix@C=105pt@R=15pt{\Lambda^{\!1,2}[1,3]\ar@{^{(}->}[d]\ar[r]^-{(\omega,-;\,\alpha
d^0_{\mathrm{v}}s^1_{\mathrm{v}},\,as^0_{\mathrm{v}}s^0_{\mathrm{v}}s^0_{\mathrm{h}},-,\, \alpha)} & K \\
                        \Delta[1,3]\ar@{.>}[ru]_-{\Gamma}}\hspace{0.3cm}
\xymatrix@C=120pt@R=15pt{\Lambda^{\!3,0}[3,1]\ar
@{^{(}->}[d]\ar[r]^-{(as^0_{\mathrm{v}}
s^0_{\mathrm{h}}s^0_{\mathrm{h}},\,
as^0_{\mathrm{v}}s^0_{\mathrm{h}}s^0_{\mathrm{h}},\,\beta,-;-,
\, \omega')} & K \\
                        \Delta[3,1]\ar@{.>}[ru]_-{\Gamma'}}
$$
and take $x_{1,2}=\Gamma d^2_{\mathrm{v}}:\Delta[1,2]\to K$ and
$x_{2,1}=\Gamma'd^3_{\mathrm{h}}:\Delta[2,1]\to K$. Then, the same
equalities as in (\ref{rar})  hold for $x_{1,2}$ instead of $\alpha$
and $x_{2,1}$ instead of $\beta$, and moreover
$x_{1,2}d^0_{\mathrm{h}}=as^0_{\mathrm{v}}s^0_{\mathrm{v}}$ and
$x_{2,1}d^1_{\mathrm{v}}=x_{2,0}$. Finally, by taking
$x_{0,3}\!:\!\Delta[0,3]\to K$ any bisimplex with
$x_{0,3}d^0_{\mathrm{v}}=x_{1,2}d^1_{\mathrm{h}},\,
x_{0,3}d^1_{\mathrm{v}}=as^0_{\mathrm{v}}s^0_{\mathrm{v}},\,
x_{0,3}d^2_{\mathrm{v}}=as^0_{\mathrm{v}}s^0_{\mathrm{v}}$ and
$x_{0,3}d^3_{\mathrm{v}}=x_{0,2}$, and $x_{3,0}\!:\!\Delta[3,0]\to
K$ any horizontal homotopy from $as^0_{\mathrm{h}}s^0_{\mathrm{h}}$
to $x_{2,1}d^0_{\mathrm{v}}$ (which exist thanks to Lemma \ref{p2}),
we have the 3-simplex $(x_{0,3},x_{1,2},x_{2,1},x_{3,0})$ of $\w K$,
which is easily recognized as a homotopy from
$(as^0_{\mathrm{v}}s^0_{\mathrm{v}},as^0_{\mathrm{h}}s^0_{\mathrm{v}},
as^0_{\mathrm{h}}s^0_{\mathrm{h}})$ to $(x_{0,2},x_{1,1},x_{2,0})$.
Consequently, $(x_{0,2},x_{1,1},x_{2,0})$ represents the identity of the
group $\pi_2 (\w K,a)$. Therefore, $\pi_2\w\epsilon$ is an isomorphism,
and the proof is complete \end{proof}

\section{The equivalence of homotopy categories}

Recall that the category of weak homotopy types is defined to be the
localization of the category of topological spaces with respect to
the class of weak equivalences, and the {\em category of homotopy
$2$-types}, hereafter denoted by $\mathrm{Ho(\mathbf{
2\text{-}types)}}$, is its full subcategory given by those spaces
$X$ with $\pi_i(X,a)=0$ for any integer $i\!>2$ and any base point
$a$.

We now define the {\em homotopy category of double groupoids
satisfying the filling condition}, denoted by
$\mathrm{Ho(\mathbf{DG}_{fc})}$, to be the localization of the
category $\mathrm{\mathbf{DG}_{fc}}$, of these double groupoids,
with respect to the class of weak equivalences, as  defined
in Subsection \ref{we}.

By Corollaries \ref{we2} and \ref{we1}, both the geometric
realization functor $\mathcal{G}\mapsto |\mathcal{G}|$ and the
homotopy double groupoid funtor $X \mapsto \dpi X$ induce equally
denoted functors
\begin{equation}\label{igr}|\hspace{7pt}|:\mathrm{Ho(\mathbf{DG}_{fc})}\to
\mathrm{Ho(\mathbf{ 2\text{-}types)}},\end{equation}
 \begin{equation}\label{idpi}\dpi:\mathrm{Ho(\mathbf{
2\text{-}types)}}\to \mathrm{Ho(\mathbf{DG}_{fc})}.\end{equation}

One of the  main goals in this section is to prove the following:

\begin{theorem}\label{mth1}  Both induced functors $(\ref{igr})$ and
$(\ref{idpi})$ are mutually quasi-inverse, establishing an
equivalence of categories
$$\mathrm{Ho(\mathbf{DG}_{fc})}\simeq\mathrm{Ho(\mathbf{
2\text{-}types)}}.$$
\end{theorem}

The proof of this Theorem \ref{mth1} is somewhat indirect.
Previously, we  shall establish the following result, where
$\mathbf{KC}$ is the category of Kan complexes and
$$\mathrm{Ho}(L\in \mathbf{KC} ~|~ \pi_iL=0,\,i>2)$$
is the full subcategory of the homotopy category of Kan complexes
given by those $L$ such that $\pi_i(L,a)=0$ for all $i>2$ and base
vertex $a\in L_0$:

\begin{theorem}\label{math2} There are adjoint functors,
$\w\dn:\mathrm{\mathbf{DG}_{fc}}\to \mathbf{KC}$, the right adjoint,
and $\adj\dec: \mathbf{KC}\to \mathrm{\mathbf{DG}_{fc}}$, the left
adjoint, that induce an equivalence of categories
$$\mathrm{Ho(\mathbf{DG}_{fc})}\simeq \mathrm{Ho}(L\in \mathbf{KC} ~|~ \pi_iL=0,\,i>2).$$
\end{theorem}
\begin{proof} The pair of adjoint functors $\adj\dec\dashv \w\dn$ is
obtained by composition of the pair of adjoint functors
$\dec\dashv\w$, recalled in (\ref{d-w}), with the pair of adjoint
functors $\adj\dashv \dn$, stated in Theorem \ref{t1}. For any
double groupoid $\mathcal{G}\in \mathrm{\mathbf{DG}_{fc}}$, its
double nerve $\dn\mathcal{G}$ satisfies the extension condition, by
Theorem \ref{thfc}, and therefore, by Fact \ref{f18} $(4)$, the
simplicial set $\w\dn \mathcal{G}$ is a Kan complex. Conversely, if $L$ is any Kan complex, then the bisimplicial set
$\dec L$ satisfies the extension condition by Fact \ref{f18} $(5)$
and, moreover, $\pi_2(\dec L_{\ast,0},a)\!=\!0\!=\!\pi_2(\dec
L_{0,\ast},a)$ for all vertices $a$, since both augmented simplicial
sets $\dec L_{*,0}\overset{d_0}\to L_0$ and $\dec
L_{0,*}\overset{d_1}\to L_0$ have  simplicial contractions,  given
respectively by the families of degeneracies $(s_p\!:\!L_p\to
L_{p\text{+}1})_{p\geq 0}$ and $(s_0\!:\!L_q\to
L_{q\text{+}1})_{q\geq 0}$. Therefore, in accordance  with Theorem
\ref{t1}, the composite functor $L\mapsto \adj \dec L$ is well
defined on Kan complexes.

By Fact \ref{f1} $(3)$, the homotopy equivalences in Fact \ref{f18}
(1), and Corollary \ref{we2}, it follows that a double functor
$F:\mathcal{G}\to\mathcal{G}'$, in $\mathrm{\mathbf{DG}_{fc}}$, is a
weak equivalence if and only if the induced simplicial map $\w\dn
F:\w\dn\mathcal{G}\to\w\dn\mathcal{G}'$ is a homotopy equivalence.

By Facts \ref{f1} $(3)$ and \ref{f18} $(2)$, Theorem \ref{4.5}, and Corollary \ref{we2},
if $f:L\to L'$ is any simplicial map between Kan complexes $L,L'$
such that $\pi_i(L,a)\!=\!0\!=\!\pi_i(L',a')$ for all $i\geq 3$ and
base vertices $a\in L_0$, $a'\in L'_0$, then $f$ is a homotopy
equivalence if and only if the induced $\adj\dec f:\adj\dec L\to
\adj\dec L'$ is a weak equivalence of double groupoids.

If $L$ is any Kan complex such that $\pi_i(L,a)=0$ for all $i\geq 3$
and all base vertices $a\in L_0$, then the unit of the adjunction
$L\to \w\dn\adj\dec L$ is a homotopy equivalence since it is the
composition of the simplicial maps
$$\xymatrix@C=17pt{L\ar[r]^-{\mathrm{u}}&\w\dec L\ar[rr]^-{\w\epsilon(\dec
L)}&& \w\dn\adj\dec L,}$$ where $u$ is a homotopy equivalence by
Fact \ref{f18}$(3)$ and Fact \ref{f1} $(3)$, and then
$\w\epsilon(\dec L)$ is also a homotopy equivalence by Theorem
\ref{4.5} and Fact \ref{f1} $(3)$.

Finally, the counit $\adj\mathrm{v}(\dn \mathcal{G}):
\adj\!\dec\w\dn\mathcal{G}\to \adj\dn\mathcal{G}=\mathcal{G}$, at
any double groupoid $\mathcal{G}$, is a weak equivalence, thanks to
Fact \ref{f18}$(3)$, Theorem \ref{4.5}, and Corollary \ref{we2}.
This makes the proof complete.
\end{proof}

Since, by Facts \ref{f1}, the adjoint pair of functors
$|\hspace{7pt}|\dashv \mathrm{S}:{\bf Top}\leftrightarrows
\mathbf{KC}$ induces mutually quasi-inverse equivalences of
categories
$$
\mathrm{Ho(\mathbf{ 2\text{-}types)}}\simeq \mathrm{Ho}(L\in
\mathbf{KC} ~|~ \pi_iL=0,\,i>2),
$$
the following follows from Theorem \ref{math2} above, and Fact \ref{f18}$(1)$:
\begin{theorem}\label{math3} The induced functor $(\ref{igr})$, $|\hspace{7pt}|:\mathrm{Ho(\mathbf{DG}_{fc})}\to
\mathrm{Ho(\mathbf{ 2\text{-}types)}}$, is an equivalence of
categories with a quasi-inverse the induced by the functor $X\mapsto
\adj\dec\,\mathrm{S}X$.
\end{theorem}

Theorem \ref{math3} gives half of Theorem \ref{mth1}. The remaining
part, that is, that the induced functor $(\ref{idpi})$ is a
quasi-inverse equivalence of $(\ref{igr})$, follows from the proposition below.
\begin{proposition} The two induced functors $\dpi,\adj\!\dec\,\mathrm{S} :\mathrm{Ho(\mathbf{
2\text{-}types)}}\to \mathrm{Ho(\mathbf{DG}_{fc})}$  are naturally
equivalent.
\end{proposition}
\begin{proof} The proof consists in displaying a natural double functor
$$\eta:\adj\!\dec\mathrm{S}X\to\dpi X,$$ which is a weak equivalence
for any topological space $X$. This is as follows:

\vspace{0.2cm} On objects of $\adj\!\dec\mathrm{S}X$, the double
functor $\eta$  carries a continuous map ${u:\Delta_1\to X}$ to
the path $\eta_u:I\to X$ given by $\eta_u(x)=u(1-x,x)$.

\vspace{0.2cm}

On horizontal morphisms of $\adj\!\dec\mathrm{S}X$, $\eta$ acts by
$$\xymatrix@C=2pt{(gd^0\overset{[g]_\mathrm{h}\ \ }\longrightarrow
gd^1)&\overset{\eta}\mapsto& (\eta_{gd^0}\to \eta_{gd^1}),}$$ the
unique horizontal morphism in $\dpi X$ from the path $\eta_{gd^0}$
to the path $\eta_{gd^1}$, for any continuous map
${g\!:\!\Delta_2\to X}$. This correspondence is well defined since
$\eta_{gd^0}(1)=gd^0(0,1)=g(0,0,1)=gd^1(0,1)=\eta_{gd^1}(1)$, and,
moreover, if $[g]_\mathrm{h}=[g']_\mathrm{h}$ in $\dec\mathrm{S}X$,
then $gd^i=g'd^i$ for $i=0,1$. And, similarly, on vertical morphisms,
$\eta$ is given by
$$\xymatrix@C=2pt{(gd^1\overset{[g]_\mathrm{v}\ \ }\longrightarrow
gd^2)&\overset{\eta}\mapsto& (\eta_{gd^1}\to \eta_{gd^2}).}$$

\vspace{0.2cm} On squares in $\adj\!\dec\mathrm{S}X$, $\eta$ is
defined by
$$
\xymatrix@R=16pt@C=24pt{\cdot\ar@{}[dr]|{\textstyle [[\alpha]]} &
\cdot\ar[l]_{[\alpha d^3]_{\mathrm{h}}}
&&&\eta_{\alpha d^1\!d^2}&\eta_{\alpha d^0\!d^2}\ar[l]\\
\cdot\ar[u]^{[\alpha d^1]_{\mathrm{v}}} & \cdot\ar[l]^{[\alpha
d^2]_{\mathrm{h}}} \ar[u]_{[\alpha
d^0]_{\mathrm{v}}}\ar@{}[rrru]|{\textstyle \overset{\eta}\mapsto}&&
&\eta_{\alpha d^1\!d^1}\ar[u]\ar@{}[ru]|{\textstyle
[\eta_\alpha]}&\eta_{\alpha d^0\!d^1}\ar[l]\ar[u] }
$$
where, for any continuous map $\alpha:\Delta_3\to X$, the map
$\eta_\alpha:I\times I\to X$ is given by the formula
$$
\eta_\alpha(x,y)=\alpha(xy,(1-x)(1-y),(1-x)y,x(1-y)).
$$

To see that $\eta$ is well defined on squares in
$\adj\!\dec\mathrm{S}X$, suppose $[[\alpha_1]]=[[\alpha_2]]$. This
means that $[\alpha_1]_\mathrm{h}=[\alpha]_\mathrm{h}$ and
$[\alpha_2]_\mathrm{v}=[\alpha]_\mathrm{v}$, for some
$\alpha:\Delta_3\to X$,  in the bisimplicial set $\dec\mathrm{S}X$.
Then, there are maps $\beta,\gamma:\Delta_4\to X$ such that the
following equalities hold:
$$
\beta d^0=\alpha_1d^0s^0,\ \beta d^1=\alpha_1,\ \beta
d^2=\alpha=\gamma d^3,\ \gamma d^4=\alpha_2, \ \gamma d^2=\alpha_2
d^2 s^2;
$$
whence  the equalities of squares in $\dpi X$,
$[\eta_{\alpha_1}]=[\eta_\alpha]=[\eta_{\alpha_2}]$, follow from the
homotopies $F_1,F_2:I^2\times I\to X$, respectively defined by the
formulas
$$\begin{array}{l}F_1(x,y,t)=\beta(xy,t(1-x)(1-y),(1-t)(1-x)(1-y),(1-x)y,
x(1-y)),\\[4pt]
F_2(x,y,t)=\gamma(xy,(1-x)(1-y),(1-x)y,tx(1-y),(1-t)
x(1-y)).\end{array}$$

Most of the details to confirm  $\eta$ is actually a double functor are
routine and easily verifiable. We leave them to the reader since the
only ones with any difficulty are those a) and b) proven below.

\vspace{0.2cm} \noindent a) For $\omega:\Delta_4\to X$,
$[\eta_{\omega d^1}]=[\eta_{\omega
d^2}]\circ_{\mathrm{h}}[\eta_{\omega d^0}]$ and $[\eta_{\omega
d^3}]=[\eta_{\omega d^4}]\circ_{\mathrm{v}}[\eta_{\omega d^2}]$.

\vspace{0.2cm} \noindent b) For $g:\Delta_2\to X$, $[\eta_{ g
s^2}]=\mathrm{I}^\mathrm{v}(\eta_{gd^1},\eta_{gd^0})$ and $[\eta_{ g
s^0}]=\mathrm{I}^\mathrm{h}(\eta_{gd^2},\eta_{gd^1})$.

\vspace{0.2cm}
However, all these equalities in a) and b) hold thanks to the relative
homotopies $$\begin{array}{ll}H_1:\eta_{\omega d^1}\to \eta_{\omega
d^2}\circ_{\mathrm{h}}\eta_{\omega d^0},& H_2:\eta_{\omega d^3}\to
\eta_{\omega d^4}\circ_{\mathrm{v}}\eta_{\omega d^2},\\[4pt] H_3:\eta_{ g
s^2}\to e^\mathrm{v},& H_4:\eta_{ g s^0}\to
e^\mathrm{h},\end{array}$$ which are, respectively, defined by the
maps $H_i:I^2\times I\to X$ such that

 \noindent
$H_1(x,y,t)\!\!=\hspace{-5pt}\left\{\begin{array}{l}\hspace{-6pt}\omega\scriptstyle{((1\text{-}t)xy,2tx(x\text{+}y),(1\text{-}x)
(1\text{-}y)\text{+}tx(2x\text{-}
2\text{+}y),y(1\text{-}x)+tx(1\text{-}2x\text{-}y),x(1\text{-}y)\text{+}tx(1\text{-}2x\text{-}y)
)}\\[-2pt] \text{if } \scriptstyle{x\text{+}y\leq 1,\, x\leq y},\\[4pt]
\hspace{-6pt}\omega\scriptstyle{((1\text{-}t)xy,2ty(x\text{+}y),(1\text{-}x)(1\text{-}y)\text{+}ty(2y\text{-}
2\text{+}x),y(1\text{-}x)+ty(1\text{-}x\text{-}2y),x(1\text{-}y)\text{+}ty(1\text{-}x\text{-}2y)
)}\\[-2pt] \text{if }  \scriptstyle{x\text{+}y\leq 1,\, x\geq y},\\[4pt]
\hspace{-6pt}\omega\scriptstyle{(xy\text{+}t(1\text{-}y)(1\text{-}x\text{-}2y),2t(1\text{-}y)(2\text{-}x\text{-}y),
(1\text{-}t)(1\text{-}x)(1\text{-}y),y(1\text{-}x)\text{-}t(1\text{-}y)(2\text{-}x\text{-}2y),(1\text{-}y)
(x\text{-}t(2\text{-}x\text{-}2y)
)}\\[-2pt]\text{if }\scriptstyle{ x\text{+}y\geq 1,\, x\leq y},\\[4pt]
\hspace{-6pt}\omega\scriptstyle{(xy\text{+}t(1\text{-}x)(1\text{-}2x\text{-}y),2t(1\text{-}x)(2\text{-}x\text{-}y),
(1\text{-}t)(1\text{-}x)(1\text{-}y),(1\text{-}x)(y\text{-}t(2\text{-}2x\text{-}y)),x(1\text{-}y)
\text{-}t(1\text{-}x)(2\text{-}2x\text{-}y)) }\\
\text{if }\scriptstyle{ x\text{+}y\geq 1,\, x\geq y},
\end{array}\right.$

\noindent
$H_2(x,y,t)\!\!=\hspace{-5pt}\left\{\begin{array}{l}\hspace{-6pt}
\omega\scriptstyle{(t(1\text{-}x)(2x\text{-}1\text{-}y)\text{+}xy,(1\text{-}x)(1\text{-}y)+t(1\text{-}x)
(2x\text{-}1\text{-}y),(1\text{-}t)(1\text{-}x)y,2t(1\text{-}x)(1\text{-}x\text{+}y),x(1\text{-}y)
\text{+}t(1\text{-}x)(y\text{-}2x)
)}\\[-2pt] \text{if }\scriptstyle{x\text{+}y\geq 1,\, x\geq y},\\[4pt]
\hspace{-6pt}\omega\scriptstyle{(xy\text{+}ty(x\text{-}2y),(1\text{-}x)(1\text{-}y)\text{+}ty(x\text{-}2y),(1\text{-}t)(1\text{-}x)y,
2ty(1\text{-}x\text{+}y),
x(1\text{-}y)\text{+}ty(2y\text{-}1\text{-}x)
)}\\ \text{if }\scriptstyle{x\text{+}y\leq 1,\, x\geq y},\\[4pt]
\hspace{-6pt}\omega\scriptstyle{(xy\text{+}t(1\text{-}y)(2y\text{-}x\text{-}1),(1\text{-}x)(1\text{-}y)+t(1\text{-}y)(2y\text{-}1\text{-}x),
(1\text{-}x)y\text{+}t(1\text{-}y)(x\text{-}2y),2t(1\text{-}y)(1\text{+}x\text{-}y),(1\text{-}t)x(1\text{-}y)
)}\\[-2pt]\text{if }\scriptstyle{ x\text{+}y\geq 1,\, x\leq y},\\[4pt]
\omega\scriptstyle{(xy\text{+}tx(y\text{-}2x),(1\text{-}x)(1\text{-}y)+tx(y\text{-}2x),y(1\text{-}x)
\text{+}tx(2x\text{-}1\text{-}y),
2tx(1\text{+}x\text{-}y),(1\text{-}t)x(1\text{-}y) )}\\[-2pt]\text{if }\scriptstyle{
x\text{+}y\leq 1,\, x\leq y},
\end{array}\right.$

\noindent$H_3(x,y,t)\!\!=\hspace{-5pt}\left\{\begin{array}{lll}\hspace{-6pt}g\scriptstyle{
((1\text{-}t)xy,(1\text{-}x)(1\text{-}y)-txy,x\text{+}
y\text{+}2xy(t\text{-}1))}&\text{if}& \scriptstyle{x\text{+}y\leq
1},\\[4pt]\hspace{-6pt}
g\scriptstyle{(xy\text{+}t(x\text{+}y\text{-}1\text{-}xy),(1\text{-}t)(1\text{-}x)(1\text{-}y),
x\text{+}y\text{-}2xy\text{+}2t(1\text{-}x)(1\text{-}y))}&\text{if}&
\scriptstyle{x\text{+}y\geq 1},
\end{array}\right.
$

\noindent$H_4(x,y,t)\!\!=\hspace{-5pt}\left\{\begin{array}{lll}\hspace{-6pt}g\scriptstyle{
(1\text{-}x\text{-}y\text{+}2xy\text{+}2tx(1\text{-}y),(1\text{-}x)y\text{+}tx(
y\text{-}1),(1\text{-}t)x(1\text{-}y))}&\text{if}&
\scriptstyle{x\leq
y},\\[4pt]\hspace{-6pt}
g\scriptstyle{(1\text{-}x\text{-}y\text{+}2xy\text{+}2ty(1\text{-}x),(1\text{-}t)(1\text{-}x)y,
x(1\text{-}y)\text{+}ty(x\text{-}1))}&\text{if}& \scriptstyle{x\geq
y}.
\end{array}\right.
$

\vspace{0.2cm} This double functor, $\eta:\adj\dec\s X\to\dpi X$,
which is clearly natural on the topological space $X$, is actually a
weak equivalence since, for any $1$-simplex $u:\Delta_1\to X$ and
integer $i\geq 0$, the induced map $\pi_i \eta:\pi_i(\adj\dec\s X,
u)\to \pi_i(\dpi X,\eta_u)$ occurs in this commutative diagram
$$
\xymatrix{\pi_i(\adj\dec\s X,u)\ar[dd]_-{\textstyle
\pi_i\eta}&\ar[l]_(0.52){\text{Th.\ref{th0}}}|-\cong
\pi_i(|\adj\dec\s X|,u)&\pi_i(|\dec\s
X|,u)\ar[l]_(0.48){\text{Th.\ref{4.5}}}|-\cong \ar[d]^(0.52){\text{ Fact\ref{f18}(1)}}|-\cong \\
&&\pi_i(|\w\dec\s X|,u)\ar[d]^(0.52){\text{
Fact\ref{f18}(3)}}|-\cong\\ \pi_i(\dpi
X,\eta_u)\ar[r]^(0.52){\text{Th.\ref{ps}}}|-\cong&\pi_i(X,u(1,0))&\pi_i(|\s
X|,u(1,0))\ar[l]_(0.52){\text{Fact\ref{f1}(6)}}|-\cong}
$$
in which all other  maps are bijections (group isomorphisms for
$i\geq 1$) by the references in the labels.
\end{proof}

\end{document}